%% file: SCLC-arxiv.tex
\documentclass[11pt,reqno]{amsart}

\pdfoutput=1
\usepackage[all, cmtip]{xy}

\usepackage{epsfig}
\usepackage{amscd}
\usepackage[mathscr]{eucal}
\usepackage{amssymb}
\usepackage{bm}
\usepackage{amsxtra}
\usepackage{amsmath}
\usepackage{latexsym} 
\usepackage[all]{xy}
\usepackage{enumerate}

\usepackage[shortlabels]{enumitem}

\usepackage{color}
\usepackage{xcolor, soul}
\sethlcolor{green}

\usepackage{bbm}
\usepackage[colorlinks=true,urlcolor=blue,
bookmarks=true,bookmarksopen=true,citecolor=blue
]{hyperref}
 \usepackage{cleveref}

 \newcommand{\field}[1]{\mathbb{#1}}
 \newcommand{\R}{\field{R}}

 \newcommand{\Z}{\field{Z}}
 \newcommand{\C}{\field{C}}
 \newcommand{\fieldd}[1]{\mathcal{#1}}
 \newcommand{\A}{\fieldd{A}}
\newcommand{\B}{\fieldd{B}}
 \newcommand{\Rc}{\fieldd{R}}
\newcommand{\LL}{\fieldd{L}}
\newcommand{\SSS}{\fieldd{S}}
 \newcommand{\Ic}{\fieldd{I}}
 \newcommand{\D}{\fieldd{D}}
\newcommand{\QQ}{\fieldd{Q}}
 \newcommand{\CC}{\fieldd{C}}
 \renewcommand{\P}{\fieldd{P}}
 \newcommand{\M}{\fieldd{M}}
 \newcommand{\NN}{\fieldd{N}}
 \newcommand{\OO}{\fieldd{O}}
 \newcommand{\Y}{\fieldd{Y}}
\newcommand{\kk}{\mathbf{k}}

\newcommand{\wt}{\widetilde}

\newcommand{\CP}{\mathbb{C}P}
\newcommand{\Hom}{\operatorname{Hom}}

\newcommand{\id}{\operatorname{id}}

\newcommand{\dif}{\mathrm{d}}
\newcommand{\fuk}{\mathcal{F}uk}

\newcommand{\Ext}{\mathrm{Ext}}
\newcommand{\rk}{\mathrm{rk}}
\newcommand{\vphi}{\varphi}

\newcommand{\om}{\omega}
\newcommand{\Om}{\Omega}
\newcommand{\La}{\Lambda}

\newcommand{\abs}[1]{|#1|}
\newcommand{\norm}[1]{\lVert{#1}\rVert}

\newcommand{\To}{\longrightarrow}
\newcommand{\Mapsto}{\longmapsto}
\theoremstyle{plain}
\newtheorem{thm}{Theorem}[section]
\newtheorem*{thm*}{Theorem}
\newtheorem{lem}[thm]{Lemma}
\newtheorem*{lem*}{Lemma}
\newtheorem{prop}[thm]{Proposition}
\newtheorem*{prop*}{Proposition}
\newtheorem{cor}[thm]{Corollary}
\newtheorem*{cor*}{Corollary}

\newtheorem{claim}[thm]{Claim}
\newtheorem*{claim*}{Claim}

\newtheorem*{assumptions*}{Assumptions}
\newtheorem*{assumption*}{Assumption}

\theoremstyle{definition}
\newtheorem{df}[thm]{Definition}%[section]
\newtheorem*{notation}{Notation}
\newtheorem*{notedf}{Note}

\newtheorem{ex}[thm]{Example}%[section]

\theoremstyle{remark}
\newtheorem{rem}[thm]{Remark}%[section]

\title{Stability Conditions and Lagrangian Cobordisms}
\author{Felix Hensel}
\thanks{The author was partially supported by the Swiss National Science Foundation (grant number 200021-156000).}
\date{\today}                                           % Activate to display a given date or no date

\begin{document}

\input{Abstract}
\maketitle
\tableofcontents

\input{Introduction}
\input{Floer_Theory}
\input{Lagrangian_Cobordisms}
\input{SC_and_main_result}
\input{Cone_decomp}
%\renewcommand{\thethm}{\arabic{chapter}.\arabic{thm}}
\input{Theta_map}
\input{Proof_of_main_result}

%\clearpage
%\chapterstyle{default}
%\begin{appendices}
\begin{appendix}
\renewcommand{\thethm}{\Alph{section}.\arabic{thm}}
   \input{Appendix_triangles}
   \input{Appendix_HN-filtration}
   \input{Appendix_qac}
\end{appendix}

%\end{appendices}

 \bibliographystyle{alpha}
% \bibliography{./Stability_Conditions.bbl}
%\bibliography{./Bibliography/PhDbib}
\bibliography{SCLC-arxiv.bbl}

\end{document}

%% file: Abstract.tex
\begin{abstract}

In this paper we study the interplay between Lagrangian cobordisms and stability conditions. %\par
% The Fukaya category, and its derived version, are central objects of study in the field of symplectic topology.
% In recent years, Biran and Cornea developed a notion of Lagrangian cobordism which proved to provide valuable insight into the Fukaya category of a symplectic manifold.
% In particular, they established a surjective group homomorphism 
% $$
% \Theta:\Om_{Lag}(M)\To K_0(D\fuk(M))
% $$
% from the Lagrangian cobordism group to the Grothendieck group of the derived Fukaya category.
% This links, to some extent, the abstract algebraic structure of the derived Fukaya category to the geometry of Lagrangian cobordisms.
% In order to understand more about the structure of the derived Fukaya category, it is of importance to investigate the kernel of the homomorphism $\Theta$.\par
%
% Bridgeland introduced the concept of stability conditions which is a useful tool to study triangulated categories.
% We bring together the notions of Lagrangian cobordisms and stability conditions to investigate the derived Fukaya category, which is a triangulated category.
We show that any stability condition on the derived Fukaya category $D\fuk(M)$ of a symplectic manifold $(M,\om)$ induces a stability condition on the derived Fukaya category of Lagrangian cobordisms $D\fuk(\C \times M)$.
In addition, using stability conditions, we provide general conditions under which the homomorphism
 $
 \Theta:\Om_{Lag}(M)\To K_0(D\fuk(M)),
 $
introduced by Biran and Cornea \cite{BC-CobI, BC-CobII}, is an isomorphism.
%This allows us to elucidate Haug's result, that the Lagrangian cobordism group of $T^2$ is isomorphic to $K_0(D\fuk(T^2))$ \cite{Haug-T^2}, and bring it into the context of stability conditions.
This yields a better understanding of how stability conditions affect $\Theta$ and it allows us to elucidate Haug's result, that the Lagrangian cobordism group of $T^2$ is isomorphic to $K_0(D\fuk(T^2))$ \cite{Haug-T^2}. %, and bring it into the context of stability conditions.
%This provides a general framework in which Haug's result \cite{Haug-T^2}, that the Lagrangian cobordism group of $T^2$ is isomorphic to $K_0(D\fuk(T^2))$, can be understood.

\end{abstract}

%%% Local Variables:
%%% mode: latex
%%% TeX-master: "../SCLC-arxiv"
%%% End:

%% file: Introduction.tex
\section{Introduction}
%\chapter{Introduction}

The notion of \emph{stable vector bundles} was first introduced by Mumford in \cite{Mumford-63} and was extensively studied and further developed by many people since then.
Building on advances in the study of Dirichlet branes in string theory, and in particular on the work of Douglas \cite{Douglas-stability}, Bridgeland \cite{Bridgeland-SC} generalised this notion of stability to the setting of triangulated categories by introducing so called \emph{stability conditions} on triangulated categories.
Bridgeland \cite{Bridgeland-SC} further shows that the set of stability conditions on a triangulated category has a natural topology and hence gives rise to an invariant of the triangulated category.
This makes stability conditions a useful tool in the study of triangulated categories.
Turning to more specific terms, a stability condition on a triangulated category $\D$ is a pair $(Z,\P)$ consisting of a slicing $\P=\{\P(\phi)\}_{\phi \in \R}$ where each $\P(\phi)$ is the subcategory of \emph{semistable objects of phase $\phi$} and $Z:K_0(\D)\to \C$ is a compatible central charge function.
The pair $(Z,\P)$ has to satisfy a number of axioms (see~\cite{Bridgeland-SC}, resp.~Def.~\ref{df:SC} below).
One important axiom is that each non-zero object of $\D$ admits a unique \emph{Harder-Narasimhan filtration}.
That is, an iterated cone decomposition over semistable objects, which are unique up to isomorphism, of \emph{strictly decreasing phases} (see~Def.~\ref{df:SC},~\ref{axiom-4} below).
In particular, the semistable objects (resp.~stable objects) generate $\D$ as a triangulated category.\par
In recent years, stability conditions gained much popularity and generated many advances, most prominently in the study of bounded derived categories of coherent sheaves $D^b\mathrm{Coh}(X)$.
In general it is difficult to construct, or establish the existence of, a stability condition on $D^b\mathrm{Coh}(X)$.
In \cite{Bridgeland-SC, Bridgeland-SC_on_K3} Bridgeland constructs stability conditions on $D^b\mathrm{Coh}(X)$ for the cases when $X$ is an elliptic curve or an algebraic $K3$-surface.
There are more results in dimensions 2 and 3, some of which can be found in \cite{Bayer-Macri-Toda-2014, Maciocia-Piyaratne-I, Maciocia-Piyaratne-II}.
\par
On the other hand, the derived Fukaya category is a central object of study in symplectic geometry.
The derived Fukaya category $D\fuk(M)$ of a symplectic manifold $(M,\om)$ is a triangulated category which encodes information on the Lagrangian submanifolds of $M$.
So it is natural to ask whether the concept of stability conditions may be used to facilitate a better understanding of $D\fuk(M)$.
Inspired by Thomas \cite{Thomas-MMMMM} and Thomas and Yau \cite{Thomas-Yau} and by earlier insights from string theory, Joyce \cite{joyce-conj} formulates a conjecture on the existence of a stability condition on $D\fuk(M)$ for a Calabi-Yau manifold $M$, where \emph{special Lagrangians} are expected to be semistable objects.\par

In 1994 Kontsevich \cite{Kontsevich-94} formulated the \emph{homological mirror symmetry conjecture} which states that there should be a triangulated equivalence between the split-closed derived Fukaya category of a Calabi-Yau manifold and the bounded derived category of coherent sheaves of its mirror manifold, that is
$$
D^\pi\fuk(X) \simeq D^b\mathrm{Coh}(X^\vee).
$$ 
This conjecture has been proven in some cases, including the tori $T^2$ by Polishchuk-Zaslow \cite{Polishchuk-Zaslow} and $T^4$ by Abouzaid-Smith \cite{Abouzaid-Smith-T^4}, the genus two curve and the quartic surface in $\CP^3$ by Seidel  and for Calabi-Yau hypersurfaces (of dimension $\geq 3$) of projective spaces by Sheridan \cite{Sheridan-HMS_for_CY_HS}.\par
On the $B$-side, $D^b\mathrm{Coh}(X)$ is the derived category of the abelian category of coherent sheaves in the ordinary sense.
Stability conditions are better understood in this setting than on the $A$-side where the derived Fukaya category is not the derived category of any abelian category but arises from the $A_\infty$-category $\fuk(M)$.
In cases where the mirror symmetry conjecture holds and it is known that stability conditions on the $B$-side exist, one can obtain stability conditions on the $A$-side by the equivalence of the respective derived categories.
For example, stability conditions on the bounded derived category of coherent sheaves over an elliptic curve are well understood (see also Example~\ref{ex:SC-on-DCoh-EllCurve}) and so by mirror symmetry one gets stability conditions on the derived Fukaya category of the torus $T^2$.\par

Biran and Cornea developed a theory of \emph{Lagrangian cobordism} in \cite{BC-CobI, BC-CobII, BC-Lef}, where they constructed the \emph{Fukaya category of cobordisms} and established links between some of the  algebraic relations in the derived Fukaya category of the underlying symplectic manifold and the geometry of Lagrangian cobordisms.
In particular, one may express a Lagrangian over a horizontal end of a cobordism as an iterated cone, in the derived Fukaya category, of the Lagrangians occurring over the remaining ends of the cobordism, c.f.~Theorem~2.2.1 in \cite{BC-CobI} (see also \eqref{eq: cone-decomp in M}).\par

Both stability conditions and the theory of Lagrangian cobordism are useful tools to understand relations in the derived Fukaya category from the algebraic respectively geometric viewpoint.
So it is natural to ask in what way these concepts are related.
In this paper we are working with a restricted class of symplectic manifolds and Lagrangian cobordisms.
The precise setting is defined in Section~\ref{sec:Fukaya_Category} (roughly speaking we are working with graded, weakly exact Lagrangian submanifolds of Calabi-Yau manifolds).
The main result we prove in this paper is the following (see Theorem~\ref{thm:(Z,P)-is-SC} for a more precise statement):

\begin{thm*}[\textbf {A}]\label{thm:thm-A}
  A stability condition on the derived Fukaya category $D\fuk(M)$ of a symplectic manifold $(M,\om)$ induces a stability condition on the derived Fukaya category of cobordisms $D\fuk(\C\times M)$.
\end{thm*}

Since a triangulated category that admits a stability condition is necessarily \emph{split-closed} (see \cite[Prop.~2.4]{huybrechts-introSC} and \cite{Le-Chen}) we get the following corollary (see Corollary~\ref{cor:DF(M)loc.fin.SC-DF(CxM)split-closed}):
\begin{cor*}
   If $D\fuk(M)$ admits a stability condition, then $D\fuk(\C\times M)$ is \emph{split-closed}.
\end{cor*}

In particular, as indicated above, the derived Fukaya category of the torus $T^2$ (as defined in \cite{Haug-T^2}) admits a stability condition and hence the derived Fukaya category of cobordisms on $T^2$ is split-closed (cf.~Example~\ref{ex:T^2-cob-split-closed}).

%\subsection{The relation of $\Om_{Lag}(M)$ and $K_0(D\fuk(M))$}
\subsection{The relation of $\Om_{Lag}(M)$ and $K_0(D\fuk(M))$}
In \cite{BC-CobII} Biran and Cornea defined the \emph{Lagrangian cobordism group} $\Om_{Lag}(M)$ of a symplectic manifold $(M,\om)$ (see also Definition~\ref{df:Lag.Cob.group} below).
They observed (see \cite{BC-CobII, Haug-T^2}) that there is a natural surjective group homomorphism
\begin{align*}
  \Theta:\Om_{Lag}(M) \To K_0(D\fuk(M))
\end{align*}
from the Lagrangian cobordism group to the Grothendieck group of the derived Fukaya category of $M$ (see Appendix~\ref{subsubsec:cones} for the definition of the Grothendieck group).
It is an interesting question to understand when $\Theta$ is an isomorphism or, more generally, to understand the kernel of $\Theta$.
Haug \cite{Haug-T^2} used homological mirror symmetry to show that $\Theta$ is an isomorphism in the case of the torus $T^2$.
In Section~\ref{sec:theta-map} we explain the role of stability conditions in $\Theta$ being an isomorphism and we bring Haug's result into the perspective of stability conditions.
Moreover, we will establish a general criterion under which $\Theta$ is an isomorphism (see Corollary~\ref{cor:theta-iso}).

%\subsection{Strategy of the proof of Theorem~A}
\subsection{Strategy of the proof of Theorem~A}
In \cite{BC-Lef} Biran and Cornea show that the derived Fukaya category of cobordisms is generated by (the image under the Yoneda-embedding of) Lagrangian cobordisms of the form $\gamma_j\times L \subset \C\times M$, where $L\subset M$ is a Lagrangian submanifold and $\gamma_j$ is a curve in $\C\cong \R^2$ with horizontal ends at heights $1$ and $j$.
In particular, any Lagrangian cobordism $V$ with positive ends $(L_1,\ldots,L_s)$ and without any negative ends (see Section~\ref{sec:Lagrangian_cobordisms}) is isomorphic in $D\fuk(\C\times M)$ to an iterated cone $(\gamma_{h_s}\times L_s\to \ldots \to\gamma_{h_2}\times L_2)$ where $h_s>\ldots >h_2>h_1=1$ are integers indicating the heights of the cylindrical ends of $V$ (see~Prop.~\ref{Prop:conedecomp}).
So, given a stability condition $(Z^M,\P^M)$ on $D\fuk(M)$ we define a candidate $(Z,\P)$ for a stability condition on $D\fuk(\C\times M)$ (cf.~\eqref{df:slicing-phase}).
Given our setup the definition of $(Z,\P)$ is relatively straightforward.
Roughly speaking, objects of the form $\gamma_j\times L$ are semistable if $L\in Ob(D\fuk(M))$ is semistable with respect to $(Z^M,\P^M)$.
We then check that the candidate $(Z,\P)$ satisfies axioms \ref{axiom-1}-\ref{axiom-4} of Definition \ref{df:SC}, which is more involved.
It is important to note that semistable objects in $M$ are not necessarily \emph{geometric}, that is they are not necessarily isomorphic to the image of an honest Lagrangian under the Yoneda-embedding.\par
The most intricate axiom to check is axiom \ref{axiom-4}, i.e. the existence of a Harder-Narasimhan filtration.
A Lagrangian cobordism $V$ with positive ends $(L_1,\ldots,L_s)$ and without any negative ends can be viewed, via the Yoneda-embedding, as a \emph{geometric} object of $D\fuk(\C\times M)$.
We now use Biran and Cornea's iterated cone decomposition of $V$ as a starting point to construct the HN-filtration.
However, the Lagrangians $(L_1,\ldots,L_s)$ are possibly not semistable in $D\fuk(M)$.
So we need to use the HN-filtration of each of these Lagrangians with respect to $(Z^M,\P^M)$ in order to refine the iterated cone decomposition.
This relies on the fact that there are inclusion respectively restriction functors between the derived categories $D\fuk(M)$ and $D\fuk(\C\times M)$ (cf.\eqref{eq:inclusion-functor} and \eqref{eq:restriction-functor}).
After such a refinement we have an iterated cone decomposition of $V$ over \emph{semistable objects}.
However, the phases of these semistable objects do not necessarily form a strictly decreasing sequence.
By algebraic arguments we may adapt this cone decomposition by, for example, switching the order of adjacent objects.
After finitely many steps we arrive at the desired HN-filtration.\par

The case where we consider a general object $W$ of $D\fuk(\C\times M)$ is more delicate.
In this case we first decompose $W$ into an iterated cone over geometric objects.
Then we have to reshuffle the factors of the HN-filtrations of the individual geometric objects in order to obtain the HN-filtration of $W$ itself.
Compared to the previous case there is an additional complication here, arising from the fact that reordering the brackets in an iterated cone decomposition comes along with \emph{shifts} in the triangulated category (cf.~Remark~\ref{rem:shifting-iterated-cones}).

\subsection{Organization of the paper}
%\subsection{Organization of the thesis}

The definition of the Fukaya category and related concepts will be reviewed in Section~\ref{sec:Fukaya_Category}.
Section~\ref{sec:Lagrangian_cobordisms} is dedicated to Lagrangian cobordisms and inclusion respectively restriction functors between the categories $D\fuk(M)$ and $D\fuk(\C\times M)$.
In Section~\ref{Section: Stab.Cond.} we will introduce the notion of stability condition, state the main result and discuss an important example.
Section~\ref{sec:cone_decomp} deals with the cone decomposition induced by Lagrangian cobordisms and contains some algebraic remarks on $A_\infty$-modules.
In Section~\ref{chap:theta-map} we will investigate a relation between the Lagrangian cobordism group and stability conditions on the derived Fukaya category and discuss the case of the torus $T^2$.
In Section~\ref{Sec:proof_of_main_result} we will prove the main result, Theorem~\ref{thm:(Z,P)-is-SC}. %
Appendix~\ref{subsubsec:cones} deals with iterated cone decompositions in triangulated categories.
In Appendix~\ref{appendix:HN-filtration} we prove the uniqueness of the Harder-Narasimhan filtration.
Some basic information on quasi-abelian categories can be found in Appendix~\ref{appendix:qac}.

 \subsubsection*{Acknowledgement}
I would like to express my gratitude to my advisor Paul Biran for numerous suggestions, patient explanations, and for generously sharing his insights with me during many long discussions.
I also thank Octav Cornea and Luis Haug for helpful comments and their interest in my work.

%%% Local Variables:
%%% mode: latex
%%% TeX-master: "../SCLC-arxiv"
%%% End:

%% file: Floer_Theory.tex
\section{Preliminaries on the Fukaya Category}\label{sec:Fukaya_Category}
%\chapter{Preliminaries on the Fukaya Category}\label{sec:Fukaya_Category}

In this section we will outline the setup of Lagrangian Floer theory and the Fukaya category and we will define the setting of the paper.
For a more detailed and in depth treatment of these subjects we refer the reader to \cite{BC-CobI, BC-CobII, seidel-book, seidel-gradedlag, joyce-conj, FOOO-bookI, FOOO-bookII}.
Experts on Lagrangian Floer theory may want to skip this section.

\subsection{The Fukaya category}
%\section{The Fukaya category}

The Fukaya category $\fuk(M)$ of a symplectic manifold $(M,\omega)$ is an $A_\infty$-category over some ground field $K$.
To be more precise one should write $\fuk(\CC)$, where $\CC$ is an admissible class of Lagrangians in $M$.
There are many different versions of the Fukaya category. For us it doesn't really matter which class $\CC$ of objects we are working with, as long as the Fukaya category $\fuk(\CC)$ can be set up properly and \emph{$\Z$-gradings} are incorporated.
For this reason we will often just write $\fuk(M)$ for the Fukaya category without specifying the class $\CC$ of admissible objects.

\subsubsection{Objects}\label{subsubsect: objects}
The objects of $\fuk(\CC)$ are \emph{Lagrangian branes}.
These are triples $(L,\theta,P)$ where $L\subset M$ is a Lagrangian in $M$ of a certain class $\CC$, $\theta$ is a \emph{$\Z$-grading} on $L$ and $P$ is a \emph{Pin-structure} on $L$.
Usually the Pin-structure will be omitted from the notation of Lagrangian branes and sometimes we will also omit the grading.
We will explain more about the grading later on in this section, for the definition and more details on Pin-structures we refer to \cite{Lawson-Blaine-spin_geometry} and \cite[(11i)]{seidel-book}.

There are different possible choices for the class $\CC$ of admissible Lagrangians.
For instance one may take $\CC$ to be contained in one of the following classes of Lagrangians.
\begin{itemize}
 \item Exact Lagrangians in $M$. In this case one may take $\C$ as the ground field.
 \item Weakly exact Lagrangians in $M$, i.e. Lagrangians such that $\om$ vanishes on $\pi_2(M,L)$.
   In this case it is necessary to work over the \emph{Novikov-field}
   \begin{align}\label{eq:Novikov-field}
  \Lambda := \left\{ \sum_{i=0}^\infty c_i T^{a_i} \, \Big|\, c_i\in\C,\, a_i\in\R,\, a_i<a_{i+1},\, \lim_{i\to\infty} a_i = \infty\right\}.     
   \end{align}

\end{itemize}

\subsubsection{Grading}

In this section we collect some remarks about graded Lagrangian submanifolds following the references \cite{seidel-gradedlag, seidel-book, joyce-conj}. By \emph{grading} we will always mean \emph{$\Z$-grading} from now on.\par
Suppose that $(M^{2n},\om)$ is a symplectic manifold with a compatible almost complex structure $J$ and Riemannian metric $g$, induced by $\om$ and $J$.
There is a natural fibre bundle $\mathcal{L}^M\to M$ associated to $M$ with fibre $\mathcal{L}^M_x$ the Lagrangian Grassmannian of $(T_xM,\om_x)$.
If $2c_1(M)=0\in H^2(M;\Z)$ then the line bundle $K_M^{\otimes 2}$ is trivial, where $K_M:=\bigwedge^{n,0}(T^*M,J)$ is the canonical line bundle over $M$.
Fix a non-vanishing section $\Theta:M\to K_M^{\otimes 2}$, which determines a non-vanishing complex-linear $n$-form on $M$ up to sign.
Such a section exists since $K_M^{\otimes 2}$ is trivial.
This induces a well-defined map
$$
 \alpha_\Theta:\mathcal{L}^M\To S^1; \quad (x,\Lambda_x) \Mapsto \frac{\Theta_x((v_1\wedge\cdots\wedge v_n)\otimes (v_1\wedge\cdots\wedge v_n))}{\abs{\Theta_x((v_1\wedge\cdots\wedge v_n)\otimes (v_1\wedge\cdots\wedge v_n))}}
$$
where $v_1,\ldots,v_n$ is any basis of the Lagrangian subspace $\Lambda_x\subset T_xM$.
To any Lagrangian submanifold $L\subset M$ we can associate a Lagrangian subbundle $TL\subset TM$ or equivalently a section $s_L:L\to \mathcal{L}^M |_L$, given by its Gauss-map.

\begin{df}
  A \emph{grading} (or \emph{$\Z$-grading}) of an oriented Lagrangian $L\subset M$ is a lift $\theta$ of the map $\alpha_\Theta \circ s_L: L \to S^1$ to $\R$.
  $$ \xymatrix{
    & \R\ar[d]^{e^{2\pi i \bullet}}\\
    L\ar@{-->}[ur]^\theta \ar[r]_{\alpha_\Theta \circ s_L} & S^1
  }
  $$
  The tuple $(L,\theta)$ consisting of an oriented Lagrangian together with a grading is called a \emph{graded Lagrangian}.
  A graded Lagrangian $(L,\theta)$ is said to be \emph{special} if $\theta$ is constant.
\end{df}

A grading on $L$ does not always exist, such a lift exists if and only if $(\alpha_\Theta \circ s_L)_\#(\pi_1(L))=0\subset \pi_1(S^1)\cong \Z$, i.e.~if the Maslov-class of $L$ vanishes.
If $L$ is connected and a grading exists, then it is unique up to adding integers.\\

\begin{df}
  A \emph{Calabi-Yau manifold} of dimension $n$ is a quadruple $(M,J,g,\Om)$ consisting of a complex manifold $(M,J)$ of complex dimension $n$ together with a K\"ahler metric $g$, K\"ahler form $\om$ and a holomorphic $(n,0)$-form $\Om$ on $(M,J)$ satisfying
$$
\om^n / n! = (-1)^{n(n-1)/2}(i/2)^n\Om\wedge \bar{\Om}.
$$
\end{df}

\begin{ex}
  The simplest example of a Calabi-Yau manifold is the standard complex space $\C^n$ with coordinates $(z_1,\ldots,z_n)$ endowed with the standard complex structure $J$, symplectic form $\om_{\mathrm{std}}$, standard K\"ahler metric $g_{\mathrm{std}}$ and holomorphic $n$-form $\Om_{\mathrm{std}}$ given by
  \begin{align*}
    \om_{\mathrm{std}} = \frac{i}{2}\sum_{j=1}^n\dif z_j\wedge \dif \bar{z}_j\quad\text{and}\quad \Om_{\mathrm{std}}=\dif z_1\wedge\cdots \wedge \dif z_n.
  \end{align*}
In particular $(\C,i,g_{\mathrm{std}},\Om_{\mathrm{std}})$ is Calabi-Yau.

\end{ex}

If $(M^{2n},J,g,\Om)$ is a Calabi-Yau manifold then we can set $\Theta:= \Om^{\otimes 2}$ and apply the above construction.
 Let $(L_1,\theta_1), (L_2,\theta_2)$ be two graded Lagrangians in $M$, intersecting transversely at a point $p\in M$.
  We can choose an isomorphism $(T_pM,J_p,g_p,\Om_p)\cong (\C^n,J_{\mathrm{std}},g_{\mathrm{std}}, \Om_{\mathrm{std}})$ (see \cite{FOOO-Ch10} and \cite[Def.~2.20]{joyce-conj}) under which $T_pL_1$ is identified with $\R^n\subset \C^n$ and $T_pL_2$ is identified with the plane $\{e^{i\phi_1}x_1,\ldots,e^{i\phi_n}x_n \,|\, x_j\in \R\}\subset \C^n$.
For a fixed K\"ahler structure, the \emph{K\"ahler angles} $\phi_1,\ldots,\phi_n\in (0,\pi)$ are unique up to ordering.
%\end{df}
\begin{df}
  The \emph{degree} $\deg_p(L_1,L_2)$ of the intersection point $p\in L_1\pitchfork L_2$ is defined by 
$$\deg_p(L_1,L_2) =n + \theta_2(p)-\theta_1(p)-\frac{1}{\pi}\sum_{j=1}^n\phi_j.$$
\end{df}

\begin{rem}
  Note that $\deg_p(L_1,L_2)$ is an integer since $\theta_2(p)=\theta_1(p)+\frac{1}{\pi}\sum_{j=1}^n\phi_j \mod \Z$.\par
  Changing the order of $L_1$ and $L_2$ has the effect of replacing the K\"ahler angles $\phi_1,\ldots,\phi_n$ by $\pi-\phi_1,\ldots,\pi-\phi_n$ and therefore we get the relation
  $$
  \deg_p(L_1,L_2)+\deg_p(L_2,L_1)=n.
  $$
  Moreover, since the K\"ahler angles all lie in the open interval $(0,\pi)$ we have that
  $$
  \theta_2(p)-\theta_1(p) < \deg_p(L_1,L_2) < \theta_2(p)-\theta_1(p) + n.
  $$
\end{rem}

\begin{df}\label{df:shift}
  For $\sigma \in\Z$ there is a \emph{$\sigma$-fold shift operator} on graded Lagrangians given by subtracting $\sigma$ from the grading
  $$
  S^\sigma(L,\theta):=(L,\theta - \sigma)
  $$
  and changing the orientation on $L$ if $\sigma$ is odd.
  We also write $(L,\theta)[\sigma]:=S^\sigma(L,\theta)$.
\end{df}
\begin{rem}
  The shift operator also has an effect on the Pin-structure, for this and details on Pin-structures we refer to Seidel \cite[(11i),(11k)]{seidel-book}.
\end{rem}

\subsubsection{Morphisms and the Floer-complex}
Let $\CC$ be a class of admissible Lagrangians, we take the Novikov-field $\Lambda$ (see \eqref{eq:Novikov-field}) as our ground field.
Let $L_i:=(L_i,\theta_i,P_i)$, $i=0,1$ be two Lagrangian branes such that $L_0$ and $L_1$ intersect \emph{transversely}.
The Floer-cochain complex $CF(L_0,L_1)$ is the graded $\Lambda$-vector space whose degree $k$-part is given by
\begin{equation*}
  CF^k(L_0,L_1) = \bigoplus_{\substack{p\in L_0\cap L_1 \\ \deg_p(L_0,L_1)=k}}\Lambda\cdot p.
\end{equation*}
We will now briefly describe the Floer-differential.
Let $Z:= \R\times [0,1]$ be the strip with coordinates $(s,t)$ and standard complex structure $j_Z$ and let $J$ be a generic 1-parametric family of complex structures on $M$ depending on $t\in [0,1]$.
Given $p,q\in L_0\cap L_1$ we denote by $\M_0(p,q)$ the zero-dimensional component (which is a finite set if $\deg_p(L_0,L_1)=\deg_q(L_0,L_1)+1$) of all maps $u:Z\to M$ that satisfy
\begin{align}\label{eq:J-hol}
  \partial_s u+J(t,u)\partial_t u = 0
\end{align}
and the asymptotic conditions
\begin{align*}
  u(s,0)\in L_0, \; u(s,1)\in L_1, \; \lim_{s\to -\infty} u(s,\cdot)=p,\; \lim_{s\to \infty} u(s,\cdot)=q
\end{align*}
modulo the $\R$-action coming from translation in the $s$-direction.
On generators $q\in L_0\cap L_1$ of degree $k$ (i.e. $\deg_q(L_0,L_1)=k$), the Floer-differential $\partial:CF^k(L_0,L_1)\to CF^{k+1}(L_0,L_1)$ is given by
\begin{equation*}
  \partial q = (-1)^k \mkern-30mu\sum_{\substack{p\in L_0\cap L_1 \\ \deg_p(L_0,L_1)=k+1}} \sum_{u\in \M_0(p,q)} \mathrm{sign}(u)\, T^{\omega(u)} \cdot p.
\end{equation*}
For details on the sign of $u$, $\mathrm{sign}(u)\in\{\pm1\}$, and how it is determined, we refer to \cite[(12f)]{seidel-book}.\par

In the general case, that is when $L_0$ and $L_1$ do not intersect transversely, one has to incorporate Hamiltonian perturbations as follows.
To each pair $(L_0,L_1)$ one associates a Floer datum $(H,J)$ consisting of a Hamiltonian function $H$ with Hamiltonian flow $\phi^H_t$ and a generic almost complex structure $J$ such that $\phi^H_1(L_0)$ intersects $L_1$ transversely.
One then considers the following perturbed version of equation \eqref{eq:J-hol}
\begin{equation*}
  \partial_s u+J(t,u)(\partial_tu - X_H(t,u)) = 0,
\end{equation*}
where $X_H$ is the Hamiltonian vector field of $H$, and proceeds as before.
Note that it is slightly imprecise to write $CF(L_0,L_1)$, one should write $CF(L_0,L_1;H,J)$ since the cochain complex depends on the choice of Floer datum.
However, it can be shown that the \emph{Floer cohomology} $HF^*((L_1,\theta_1),(L_2,\theta_2))$, which is the cohomology of this cochain complex, doesn't depend on the choice of Floer datum up to isomorphism.

\begin{rem}
  One can also endow the Lagrangians with \emph{local systems}, in which case one has to incorporate the parallel transport along the boundary of $u:Z\to M$ in the definition of the differential (see e.g.~\cite{Haug-T^2}).
\end{rem}

% Let $K$ be some ground field% \textbf{(***it might be necessary to work over the Novikov-ring and include local systems***)}
% , the Floer complex of two transverse graded Lagrangians $(L_1,\theta_1)$, $(L_2,\theta_2)$ is the free graded $K$-module $CF^*((L_1,\theta_1),(L_2,\theta_2))$  whose degree $k$ part is generated by the intersection points $p\in L_1\cap L_2$ of degree $k$ with a differential that counts elements of the moduli space of $J$-holomorphic strips with boundary on $L_1\cup L_2$ and corners at intersection points of $L_1$ and $L_2$. The generic almost complex structure $J$ does not necessarily coincide with the one from the definition of a Calabi-Yau manifold.
% The \emph{Floer cohomology} $HF^*((L_1,\theta_1),(L_2,\theta_2))$ is the cohomology of this cochain-complex.\par

The space of \emph{morphisms} in the Fukaya category $\fuk(M)$ from $L_0$ to $L_1$ is given by the graded vector space
\begin{equation*}
  \hom(L_0,L_1) = CF(L_0,L_1).
\end{equation*}

\begin{rem}
  The shift operator (see Definition~\ref{df:shift}) has the following properties (cf.~\cite{seidel-gradedlag}):
  \begin{align*}
    CF^*((L_1,\theta_1)[\sigma],(L_2,\theta_2)[\tau])& = CF^{*-\sigma+\tau}((L_1,\theta_1),(L_2,\theta_2)) \\
    HF^*((L_1,\theta_1)[\sigma],(L_2,\theta_2)[\tau])& \cong HF^{*-\sigma+\tau}((L_1,\theta_1),(L_2,\theta_2)).
  \end{align*}
\end{rem}

\subsubsection{$A_\infty$-operations}

The $A_\infty$-operations of $\fuk(M)$ are maps 

\begin{equation*}
  \mu^d:\hom(L_{d-1},L_d)\otimes \ldots\otimes \hom(L_0,L_1)\to \hom(L_0,L_d)[2-d],\quad d\geq 1.
\end{equation*}
As before we will only briefly describe these maps in the situation where all subsequent Lagrangians intersect transversely, that is we assume that $L_0\pitchfork L_d$ and $L_{i-1}\pitchfork L_i$ for $1\leq i\leq d$ (cf.~\cite{Auroux-beginner, Haug-T^2}).
For more details and the proof that $\fuk(M)$ is actually an $A_\infty$-category, we refer to \cite{seidel-book}.
\par
Let $S$ be the unit disk with one incoming boundary puncture $z_0$ and $d$ outgoing boundary punctures $z_1,\ldots,z_d$, lying in counterclockwise order on the unit circle and let $p_0, p_1,\ldots, p_d\in M$
be such that $p_0\in L_0\cap L_d$ and $p_i\in L_{i-1}\cap L_i$ for $1\leq i \leq d$.
We denote by $\M(p_0;p_1,\ldots,p_k;[u],J)$ the moduli space of maps $u:S\to M$ in the homotopy class $[u]$, extending continuously to the unit disc and mapping the arcs from $z_i$ to $z_{i+1}$ (resp. from $z_d$ to $z_0$) to $L_i$ satisfying
\begin{equation*}
Du(z) + J(z,u)\circ Du(z)\circ j_S=0
\end{equation*}
up to the reparametrization action by automorphisms of the unit disk.
Here, $J$ is a generic $\om$-compatible almost complex structure depending on $z\in S$.
Under the transversality assumption the expected dimension of the moduli space is (see e.g.~\cite[(12f)]{seidel-book})
\begin{equation*}
  \dim(\M(p_0;p_1,\ldots,p_k;[u],J))= d-2+\deg_{p_0}(L_0,L_d)-\sum_{i=1}^d\deg_{p_i}(L_{i-1},L_i)
\end{equation*}
and in this case we define the ($\Lambda$-linear) maps 
\begin{equation*}
  \mu^d:\hom(L_{d-1},L_d)\otimes \ldots\otimes \hom(L_0,L_1)\to \hom(L_0,L_d)[2-d],\quad d\geq 1,
\end{equation*}
by
\begin{equation*}
 \mu^d(p_d,\ldots,p_1) =\mkern-40mu \sum_{\substack{p_0\in L_0\cap L_d \\\dim(\M(p_0;p_1,\ldots,p_k;[u],J))=0}} \sum_{u\in\M(p_0;p_1,\ldots,p_k;[u],J)} (-1)^\dagger \mathrm{sign}(u) \cdot T^{\om([u])} p_0,
%  \mu^d(p_d,\ldots,p_1) =\mkern-40mu \sum_{\substack{p_0\in L_0\cap L_d \\\deg_{p_0}(L_0,L_d)= \sum_{i=1}^d \deg_{p_i}(L_{i-1},L_i) + 2-d}} \sum_{u\in\M(p_0;p_1,\ldots,p_k;[u],J)} (-1)^\dagger \mathrm{sign}(u) \cdot T^{\om([u])} p_0,
\end{equation*}
where $\dagger = \sum_{i=1}^d i\cdot \deg_{p_i}(L_{i-1},L_i)$.
In the general case, one has to consistently choose perturbation data associated to each tuple of Lagrangians, and perturb the Cauchy-Riemann equation for the curves $u:S\to M$.
For an extensive treatment of these issues we refer to Seidel \cite{seidel-book}.

\begin{rem}
  Note that $\mu^1$ coincides with the Floer-differential up to a sign, i.e. $\mu^1(p)=(-1)^{\abs{p}}\partial(p)$, where $\abs{p}$ denotes the degree of $p$ in the graded vector space $CF(L_0,L_1)$.
\end{rem}

\subsubsection{The derived Fukaya category}\label{sec:Dfuk}

The derived Fukaya category $D\fuk(M)$ is constructed from the Fukaya category $\fuk(M)$ by first completing it to a triangulated $A_\infty$-category (i.e.~taking a \emph{triangulated envelope} \cite[(3j)]{seidel-book}) and then taking its degree zero cohomology.
There are various realization of the derived Fukaya category, we will briefly describe one possibility via the Yoneda-embedding and refer to Seidel \cite[(3j)]{seidel-book} for another construction, yielding an equivalent category, via twisted complexes.\par

Recall that the \emph{Yoneda-embedding} $\Y:\fuk(M)\to mod(\fuk(M))$ is a cohomologically full and faithful embedding of $\fuk(M)$ into the triangulated $A_\infty$-category $mod(\fuk(M))$ of $A_\infty$-modules over itself.
On objects it acts as $\Y((L,\theta))=CF(-,(L,\theta))$.
The \emph{derived Fukaya-category} $D\fuk(M)$ is, up to equivalence, the degree $0$ cohomology of the triangulated completion $\Y(\fuk(M))^{\wedge}$ of the image of the Yoneda-embedding inside $mod(\fuk(M))$, i.e.
\begin{equation*}
D\fuk(M) = H^0(\Y(\fuk(M))^{\wedge} ).  
\end{equation*}
 The \emph{shift functor} descends to an autoequivalence of the derived Fukaya category $D\fuk(M)$.
 We will call an object of $D\fuk(M)$ \emph{geometric}, if it isomorphic to a shift of an object in the image of the Yoneda-embedding.
 The geometric objects generate the derived Fukaya category in the sense that each object of $D\fuk(M)$ is isomorphic to an iterated cone over geometric objects.

\begin{rem}
  Let $\mathcal{F}:\A \to \B$ be an $A_\infty$-functor between two $A_\infty$-categories $\A$ and $\B$.
  Then (\cite[(1k)]{seidel-book}) $\mathcal{F}$ induces a \emph{pullback} functor $\mathcal{F}^*:mod(\B)\to mod(\A)$, so we may pull back $A_\infty$-modules over $\B$ to $A_\infty$-modules over $\A$ via $\mathcal{F}$.
  Such modules cannot be pushed forward in general.
 If $\mathcal{F}^*$ has the property that it sends Yoneda-modules over $\B$ to Yoneda-modules over $\A$, then it induces a triangulated functor on the derived level $D(\B)\to D(\A)$.
  On the other hand, $\mathcal{F}$ also induces an $A_\infty$-functor $Tw\mathcal{F}: Tw \A \to Tw \B$ (\cite[Lemma~3.23]{seidel-book}) between the $A_\infty$-categories of twisted complexes over $\A$ and $\B$ (\cite[(3l)]{seidel-book}).
  This in turn induces a triangulated functor $H^0(Tw \mathcal{F}): H^0(Tw \A) \to H^0(Tw \B)$, that is a triangulated functor on the derived level $\wt{\mathcal{F}}:D(\A)\to D(\B)$ (\cite[Lemma~3.30]{seidel-book}).
  The derived  category of an $A_\infty$-category is only defined up to equivalence, depending on the choice of triangulated envelope.
  Twisted complexes $Tw \A$ and the triangulated completion of the Yoneda-embedding in the category of $A_\infty$-modules %(cf.~Section~\ref{sec:Dfuk})
over $\A$ both form triangulated envelopes of $\A$.
  We now fix triangulated equivalences $H^0(Tw \A)\overset{\eta_\A}{\To}H^0(\Y(\A)^\wedge)$ and $H^0(Tw \B)\overset{\eta_\B}{\To}H^0(\Y(\B)^\wedge)$.
  Recall that $\Y(\A)^\wedge$ denotes the triangulated completion of $\Y(\A)$ in $mod(\A)$, similarly for $\B$.
  Given an object $\M$ in $H^0(\Y (\A)^\wedge)$, i.e. an $A_\infty$-module $\M$ in the triangulated completion $ \Y (\A)^\wedge$, let $C$ be an object of $H^0(Tw \A)$ such that $\eta_\A(C)$ is isomorphic to $\M$ in $H^0(\Y(\A)^\wedge)$.
  Then there exists an $A_\infty$-module $\NN$ in $\Y(\B)^\wedge$ which is isomorphic to $\eta_\B(\wt{\mathcal{F}} C)$ in $H^0(\Y(\B)^\wedge)$.
  
\end{rem}

%----------------------------------------------------------------------------------------------------------------------------------------------------------------------------------------------------------------

%%% Local Variables:
%%% mode: latex
%%% TeX-master: "../SCLC-arxiv"
%%% End:

%% file: Lagrangian_Cobordisms.tex
\section{Lagrangian Cobordisms}\label{sec:Lagrangian_cobordisms}
%\chapter{Lagrangian Cobordisms}\label{sec:Lagrangian_cobordisms}

\subsection{Graded Lagrangian Cobordisms and the induced grading on the ends}

In \cite{BC-CobI, BC-CobII, BC-Lef} Biran and Cornea developed a theory of \emph{Lagrangian cobordism}.
Our setting will be very similar to the one employed in \cite{BC-Lef} and we will make use of the machinery and notation that Biran and Cornea established.
The difference is that we will consider \emph{graded} Lagrangian cobordism and we will work in a \emph{co}homological rather than a homological setting.\par

  Let $(M,\om)$ be a symplectic manifold.
  We denote by $\wt{M}$ the symplectic manifold $\C\times M$ endowed with the symplectic form $\wt{\om}=\om_{\mathrm{std}}\oplus \om$ and by $\pi:\wt{M}\to \C$ the projection to the complex plane.
  For subsets $X\subset \wt{M}$ and $S\subset \C$ we write $X|_S:=X\cap \pi^{-1}(S)$.

\begin{df}[cf.~\cite{BC-CobI}]\label{df_Lagrangian_cobordism} 
 Given two families $(L'_k)_{1\leq k\leq s_-}$ and $(L_j)_{1\leq j \leq s_+}$ of closed Lagrangians in $M$, a \emph{Lagrangian cobordism} $V$ between these two families is a Lagrangian submanifold $V\subset \wt{M}$ such that for some $R >0$ we have that $\pi(V)\cap [-R,R]\times \R \subset \C$ is compact and
  \begin{align*}
    V|_{(-\infty,-R]\times \R} &= \bigsqcup_k ((-\infty,-R]\times \{h'_k\})\times L'_k\\
    V|_{[R,\infty)\times \R} &= \bigsqcup_j ([R,\infty)\times \{h_j\})\times L_j,
  \end{align*}
where we have identified $\C\cong \R^2$ via $(x+iy) \leftrightarrow (x,y)$.
Moreover $h_1<\ldots <h_{s_+}$ and $h'_1<\ldots <h'_{s_-}$ are integers and $h_k$ is called the \emph{height} of the horizontal end $L_k$ in the Lagrangian cobordism $V$.
Such a Lagrangian cobordism is denoted by $V:(L_j)_{1\leq j \leq s_+}\leadsto (L'_k)_{1\leq k\leq s_-}$.
\end{df}

\begin{rem}
  Note that either of the families $(L'_k)_{1\leq k\leq s_-}$ respectively $(L_j)_{1\leq j \leq s_+}$ may be empty, in which case we have
  \begin{align*}
    V|_{(-\infty,-R]\times \R} = \emptyset \quad\text{respectively}\quad V|_{[R,\infty)\times \R} =\emptyset.
  \end{align*}
  Lagrangian cobordisms of the type $V:(L_j)_{1\leq j \leq s_+}\leadsto \emptyset$ respectively $V:\emptyset\leadsto (L'_k)_{1\leq k\leq s_-}$ are called \emph{positively ended} (see Figure~\ref{fig:pos-cob}) respectively \emph{negatively ended}.
\end{rem}

\begin{figure}[ht]
\def\svgwidth{0.6\columnwidth}
\centering
                        \includegraphics[width=0.6\columnwidth]{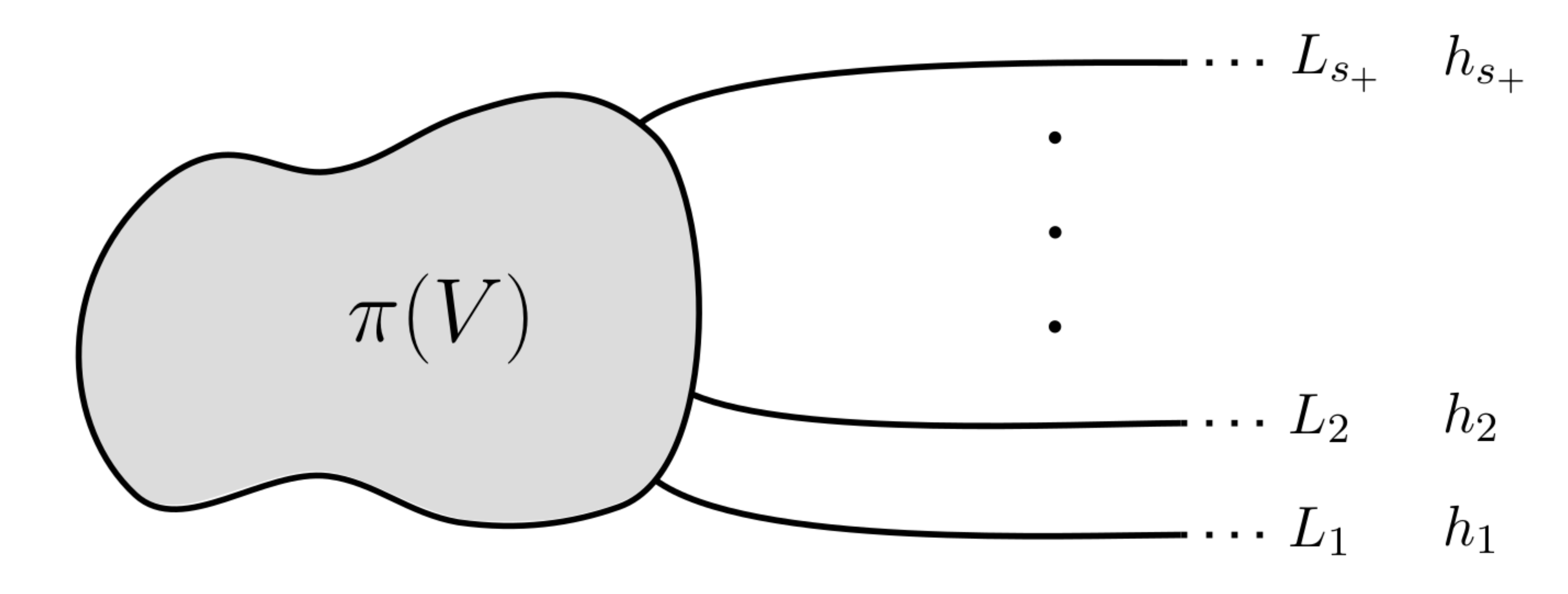}
\label{fig:pos-cob}
\caption{Projection to $\C$ of a positively ended Lagrangian cobordism.}
\end{figure}

Let now $(M^{2n},J,g,\Om)$ be a Calabi-Yau manifold.
In their work Biran and Cornea constructed the \emph{Fukaya category of cobordisms} $\fuk(\C\times M)$.
In our setting all the cobordisms will be graded.
That is, the objects of the Fukaya category $\fuk(\C\times M)$ of \emph{graded}, \emph{positively ended} Lagrangian cobordisms are graded Lagrangian cobordisms $(V,\theta_V):(L_1,\ldots,L_m)\leadsto \emptyset$ in $(\C\times M, \widetilde J,\widetilde{g},\widetilde{\Om})$ that are cylindrical in the complement of $K\times M$ and whose cylindrical ends are on heights $h_1 < \ldots < h_m$ in $\Z_{\geq 1}$.
Here $K\subset \C$ is a fixed compact subset and $L_1,\ldots,L_m$ are \emph{gradeable} Lagrangians in $M$.
In addition, to simplify the notation, we always assume that the height $h_1$ of the lowest cylindrical end $[R,\infty)\times \{h_1\}\times L_1$ is $1$, i.e.~$h_1=1$.
If necessary we will set $L_1=\emptyset$.

%----------------------------------------------------------------

Moreover, $\widetilde{\om}:= \om_{\mathrm{std}}\oplus\om$ and $\widetilde\Om$ is a holomorphic $(n+1)$-form on $\widetilde{M}:=\C\times M$.
In the complement of $K\times M$ the additional structure on $\widetilde M$ satisfies:
\begin{align*}
 \widetilde J= i\oplus J,\quad \widetilde{g}=g_{\mathrm{std}}\oplus g,\quad \text{and}\quad \widetilde{\Om}=\pi_\C^*\dif z \wedge \pi_M^*\Om
\end{align*}
where $\widetilde M \overset{\pi_\C}{\To} \C$ and $\widetilde M \overset{\pi_M}{\To} M$ are the respective projections.

%----------------------------------------------------------------
%----------------------------------------------------------------

In \cite{BC-Lef} Biran and Cornea consider only negatively ended cobordisms (for notational convenience) and in the setup of the Floer-complex $CF(V,W)$ of two cobordisms, they choose the perturbation datum in such a way that the horizontal ends of the \emph{second} cobordism $W$ are slightly perturbed downwards, i.e.~in the $-i$ direction in $\C$.
The only change we are making to Biran and Cornea's setup is that we are working with \emph{positively} ended cobordisms.
And we will still choose the perturbation datum in such a way that the horizontal ends of the \emph{second} cobordism $W$ are slightly perturbed downwards.
The reason for this change is that we are working in a \emph{co}homological and not in a homological setting.
This adaptation will ensure that the inclusion functor $\iota^{(\beta,\rho)}$, which will be defined in Section~\ref{sec:incl-restr-functors}, will indeed be an $A_\infty$-functor and respect the degrees of morphisms correctly (cf.~\cite[Section 4]{Haug-T^2}).
All the relevant constructions and results from \cite{BC-CobI, BC-CobII, BC-Lef} carry over to this setting.
\begin{rem}
We could have just as well worked with negatively ended cobordisms while choosing the opposite perturbation datum. However, this is not convenient for notational purposes.
\end{rem}

  \begin{rem}\label{rem:extra-data}
   In addition to a grading we may also endow our Lagrangians with other additional data, such as local systems or Pin-structures.
   Given Lagrangians $(L_1,\ldots,L_m)$ equipped with such extra data, a cobordism $V:(L_1,\ldots,L_m)\leadsto \emptyset$ is a cobordism of the underlying Lagrangians, which is itself endowed with the same types of extra data. The restriction of the extra data on $V$ to the ends (provided one can suitably define the restriction of the additional structure to the ends) should coincide with the given data on the ends.
  \end{rem}

  \begin{rem}
    Note also that local systems on Lagrangians in $M$ are often necessary in order to define a stability condition on $D\fuk(M)$.
    In the case of $T^2$ for example, local systems are needed in order to get a split closed derived Fukaya category (see \cite{Haug-T^2}), which is a necessary condition for the existence of a stability condition.
    However, in order to construct a stability condition on $D\fuk(\C\times M)$ from a given one on $D\fuk(M)$, we do not need to make explicit use of local systems.
    They will be implicitly carried along in case they are required in $D\fuk(M)$.
  \end{rem}

\begin{notation}
  From now on, whenever convenient, we will omit the grading from the notation, i.e.~we will simply write $L$ or $V$ instead of $(L,\theta_L)$ or $(V,\theta_V)$.
  We will sometimes also just write $(V,\theta_V)$ (resp. $(L,\theta_L)$) for the Yoneda-module $\Y(V,\theta_V)\in D\fuk(\C\times M)$ (resp. $\Y(L,\theta_L)\in D\fuk(M)$).
  It should be clear from the context which is meant.
\end{notation}

  Let $(V,\theta_V)$ be a graded Lagrangian cobordism with
  $$
  V|_{[R,\infty) \times \{h_j\} }= [R,\infty)\times \{h_j\} \times L_j
  $$
  for some Lagrangians $L_j\subset M$ and $1\leq j\leq s$ and some $R>0$.
  For $x\in [R,\infty),\, p\in L_j$ the tangent space of $V$ at $((x,h_j),p)$ splits as a direct sum
  $$
  T_{((x,h_j),p)}V \cong T_x\R \oplus T_pL_j\subset T_{(x,h_j)}\C\oplus T_pM \cong T_{((x,h_j),p)}\widetilde M.
  $$
  The part of this Lagrangian subspace that lies in $T\C$ is independent of $x\in [R,\infty)$ and hence the same is true for the grading $\theta_V$.
  Therefore, the grading $\theta_V$ on the cobordism induces a grading $\theta_{V,j}$ on each of its ends $L_j$ by
  $$
  \theta_{V,j}:L_j\To \R;\quad \theta_{V,j}(p):=\theta_V((x_0,h_j),p)
  $$
  for any fixed choice of $x_0\in [R,\infty)$.% and any $\kappa\in 2\cdot\Z_{>1}$. 
  % \begin{rem}
  %   The reason why we choose to add $\kappa h_j$ in the above definition of the induced grading will be made clear in Section \ref{Sec:proof_of_main_result}.
  % \end{rem}
  % \begin{notation}
  %   We will fix any $\kappa\in 2\cdot\Z_{>1}$ (i.e. $\kappa$ is \emph{even} and $\kappa\in\Z_{\geq 4}$) and simply write $\theta_{V,j}:=\theta_{V,j}$ in the remainder of the text in order to relax the notation.
  % \end{notation}
  \begin{rem}
    Notice that for any $p\in L_i$ and $q\in L_j$, the difference $\theta_i(p)-\theta_j(q)$ of the induced gradings on two ends $L_i$, $L_j$ of the cobordism $(V,\theta_V)$ remains unchanged after a grading shift of $V$.
    That is, we have that $(V,\theta_V)[\sigma]$ induces gradings $\theta_i-\sigma$, $\theta_j-\sigma$ and indeed $(\theta_i(p)-\sigma)-(\theta_j(q)-\sigma)=\theta_i(p)-\theta_j(q)$.
  \end{rem}

\subsection{Inclusion and restriction functors}\label{sec:incl-restr-functors}
Let $\beta:\R\to \C$ be an embedded graded smooth curve with grading $\rho:\mathrm{im}(\beta)\to \R$ and such that $\beta((-\infty,-R'])=[R,\infty)\times \{c_+\}$ and $\beta([R',\infty))=[R,\infty)\times\{c_-\}$ for some positive constants $R,R'$ and integers $0<c_-<c_+$.
If $\beta$ satisfies these conditions we say that it is a graded curve with \emph{positive horizontal ends} and by abuse of notation we sometimes write $\beta\subset \C$ for the image $\mathrm{im}(\beta)\subset\C$ of the curve.\par
Each such curve induces an $A_\infty$-inclusion functor (see \cite[Section 4.2]{BC-CobII} and \cite[Section 4]{Haug-T^2} for the graded and cohomological version)
$$
 \iota^{(\beta,\rho)}:\fuk(M)\To\fuk(\C\times M)
$$
which acts on objects by
$$
(L,\theta)\Mapsto (\beta\times L, \rho\oplus \theta).
$$
This functor induces a triangulated functor on the derived level (see~\cite[Lem\-ma~3.30]{seidel-book})
\begin{equation}\label{eq:inclusion-functor}
 \Ic^{(\beta,\rho)}:D\fuk(M)\To D\fuk(\C\times M). 
\end{equation}

The grading on such a curve $(\beta,\rho)$ is completely determined by its value on the \emph{top horizontal end}, that is by its constant value $\rho|_{[R,\infty)\times \{c_+\}}$.
For $r\in\R$ we denote by $(\beta, r)$ the graded curve where the grading is the unique grading $\rho_r$ on $\beta$ such that its value on the top horizontal end is equal to $r$.
If $r=0$ we will simply write $\iota^\beta$ (resp. $\Ic^\beta$) for the inclusion $\iota^{(\beta,0)}$ (resp. $\Ic^{(\beta,0)}$).\par

As in \cite{BC-Lef} we will fix two distinguished collections of smooth embedded curves in $\C$:
\begin{itemize}
     \item For $j\in \Z_{\geq 1}$, $\gamma_j$ is the curve with horizontal ends on the heights $1$ and $j$.
           More precisely
           $$
           \gamma_j(\R)\subset (R,\infty)\times [\frac12,\infty),\, \gamma_j(-1,1)\subset [R +1,R +2]\times [1,j]
           $$
           and
           $$
           \gamma_j((-\infty,-1])=[R +2,\infty)\times\{j\},\, \gamma_j([1,\infty))=[R +2,\infty)\times \{1\}
           $$
           where the constant $R\in\R$ is chosen in such a way that $\gamma_j$ completely lies in the region of $\C$ over which the symplectic and complex structure on $\C\times M$ are of split type.
     \item For $j\in \Z_{\geq 1}$, $\eta_j$ is the curve which looks exactly like $\gamma_j$ except that it has horizontal ends on heights $j+\frac12$ and $j-\frac12$.
\end{itemize}

	\begin{figure}[ht]
\def\svgwidth{0.3\columnwidth}
\centering
                        \includegraphics[width=0.35\columnwidth]{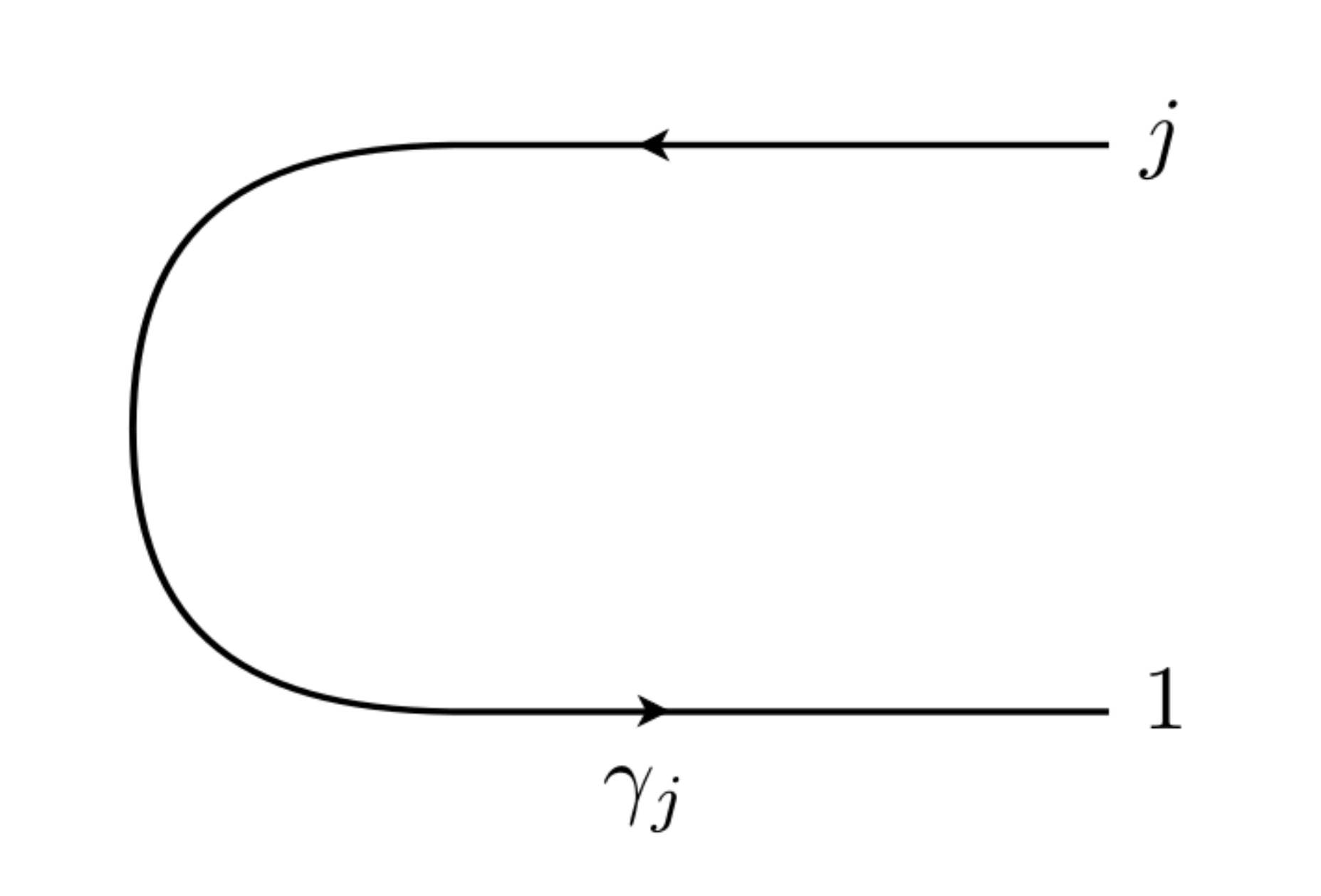}
			\caption{The image of the curve $\gamma_j$ with horizontal ends on heights $1$ and $j$.}
			\label{fig: gamma_jexample_graph}
\end{figure}

\begin{notation}
As we will mostly consider the inclusion functor along the distinguished curves $\gamma_j$, we abbreviate the notation and write
\begin{equation*}
  \label{eq:inclusion-notation}
  \iota^j:=\iota^{\gamma_j}:\fuk(M)\To\fuk(\C\times M) \quad\mkern-7mu \text{and}\quad\mkern-7mu \Ic^j:=\Ic^{\gamma_j}:D\fuk(M)\To D\fuk(\C\times M)
\end{equation*}
 for the respective inclusion functors.
 We say that an object of $D\fuk(\C\times M)$ is of \emph{height} $j$, if it is isomorphic to an object in the image of $\Ic^j$.
\end{notation}

Recall from \cite{BC-Lef} that there is a full and faithful embedding $e:\fuk(\C\times M) \to \fuk_{\frac12}(\C\times M)$ where $\fuk_{\frac12}(\C\times M)$ is the same category as $\fuk(\C\times M)$ except that cobordisms are also allowed to have horizontal ends on heights in $\frac12\Z$.
There is a corresponding Yoneda-embedding $\Y_{\frac12}:\fuk_{\frac12}(\C\times M) \to mod(\fuk_{\frac12}(\C\times M))$ and by $\Y':\fuk(\C\times M) \to mod(\fuk_{\frac12}(\C\times M))$ we will denote the composition $\Y'=\Y_{\frac12}\circ e$.
Moreover, $H^0(\Y'(\fuk(\C\times M))^{\wedge})$ is equivalent to the derived Fukaya category $D\fuk(\C\times M)$.
\begin{rem}
  In the complement of $K\times M\subset \widetilde{M}$, i.e. in the region where $\widetilde{\Om}$ is of the form $\pi_\C^*\dif z \wedge \pi_M^*\Om$, the map $\alpha_{\widetilde{\Om}^{\otimes 2}}:\mathcal{L}^{\widetilde{M}}\to S^1$ is equal to the product $\alpha_{\dif z^{\otimes 2}}\cdot \alpha_{{\Om}^{\otimes 2}}$ of the maps $\alpha_{\dif z^{\otimes 2}}:\mathcal{L}^{\C}\to S^1$ and $\alpha_{{\Om}^{\otimes 2}}:\mathcal{L}^M\to S^1$.
Hence, in this region, the grading $\theta_{\widetilde{L}}$ on a product Lagrangian $\widetilde{L}=\beta\times L\subset \widetilde{M}$ may be written as a sum $\theta_{\widetilde{L}}=\theta_\beta + \theta_L$ where $\theta_\beta$ respectively $\theta_L$ are gradings of the curve $\beta\subset \C$ respectively the Lagrangian $L\subset M$.
\end{rem}

\begin{rem}\label{rem:curve-grading-increase}
  The tangent space of the curve $\gamma_j$ at a point $t$ is given by $\R\gamma'_j(t)$ and
  $$
  \alpha_{\dif z^{\otimes 2}}(\R\gamma'_j(t)) = \frac{\gamma'_j(t)^2}{\norm{\gamma'_j(t)}^2}.
  $$
  As $t$ goes from $-\infty$ to $\infty$, $\frac{\gamma'_j(t)^2}{\norm{\gamma'_j(t)}^2}$ makes a full counterclockwise turn on $S^1$ from $1$ to $1$.
  Therefore the lift of $ \alpha_{\dif z^{\otimes 2}}\circ s_{\gamma_j}$ increases by exactly one.
  This means that the value of the grading of $(\gamma_j,r)$ on the bottom horizontal end is given by $r+1$.
  In particular, the difference between the grading on the bottom and top end is always $1$.\par

  Similarly, if we consider the curves $(\eta_j,r)$ and the horizontal curve $(\lambda_j:=\R\times \{j\},\rho')$ with the constant grading $\rho'\equiv s\in\R$, we get that the degree at the unique intersection point $y\in\C$ is $\deg_y(\eta_j,\lambda_j)= 1+ s-(r+\frac12) - \frac{1}{\pi}(\frac{\pi}{2}) = s-r$.
  This is because the K\"ahler angle between the curves $\eta_j$ and $\lambda_j$ is $\frac{\pi}{2}$ and the grading of $\eta_j$ at the point $y$ is given by $r+\frac12$.\par
 Combining these two observations leads to $\deg_z((\gamma_{j},r),(\gamma_{k},s))= s-r$, for $j\geq k$ and where $z\in\C$ is the unique intersection point of $\gamma_j$ with a slightly downward perturbed copy of $\gamma_k$.
\end{rem}

Let $\M'_{V,\theta_V}=\Y'(V,\theta_V)$ be the Yoneda-module associated to an object $(V,\theta_V)\in\fuk(\C\times M)$.
The pullback module $(\iota^{\eta_j,r})^*(\M'_{V,\theta_V})$ is exactly the Yoneda-module of the $j$-th end of $V$, shifted by $r$. %$r+\kappa h_j$.
More precisely, if $L_j$ is the $j$-th end of the cobordism $V$ and $(L,\theta)$ any object of $\fuk(M)$, then
\begin{align*}
(\iota^{\eta_{h_j},r})^*(\M'_{V,\theta_V})(L,\theta)&=CF^*((L,\theta+r),(L_j,\theta_{V,j})) \\
&=CF^*((L,\theta)[-r],(L_j,\theta_{V,j}))\\
&=CF^{*}((L,\theta),(L_j,\theta_{V,j})[r])\\
&=\Y(L_j,\theta_{V,j})[r](L,\theta)
\end{align*}
and therefore
\begin{equation}\label{eq:restriction-shift}
(\iota^{\eta_{h_j},r})^*(\M'_{V,\theta_V})=\Y(S^{r}(L_j,\theta_{V,j}))=\Y(L_j,\theta_{V,j})[r].  
\end{equation}

Hence $(\iota^{\eta_{h_j},r})^*$ takes Yoneda-modules to Yoneda-modules, and since the category $H^0(\Y'(\fuk(\C\times M)))$ is equivalent to the derived Fukaya category $D\fuk(\C\times M)$ it induces a triangulated functor on the derived level
\begin{equation}
  \label{eq:restriction-functor}
\Rc_{h_j}^{r}:D\fuk(\C \times M)\To D\fuk(M),  
\end{equation}
which we will call the $r$-\emph{restriction} to the $j$-th end.
The $0$-restriction $\Rc_{h_j}^0$ will also be denoted by $\Rc_{h_j}$.\par

\begin{rem}\label{rem:restriction-inclusion-relation}
  In particular, for $(L,\theta)\in\fuk(M)$ we have
  \begin{align*}
     \Rc_j^r\circ\Y(\iota^{{j}}(L,\theta)) = \Y(L,\theta_{\iota^{{j}}(L,\theta),2})[r] =  \Y(L,\theta)[r],
  \end{align*}
  and similarly
   \begin{align*}
     \Rc_1^r\circ\Y(\iota^{{j}}(L,\theta)) = \Y(L,\theta_{\iota^{{j}}(L,\theta),1})[r] =  \Y(L,\theta+1)[r] =  \Y(L,\theta)[r-1].
  \end{align*}

  To summarize, by Remark \ref{rem:curve-grading-increase}, and for given objects $(L,\theta)$ in $\fuk(M)$ and $\M$ in $D\fuk(M)$ we have the following relations:
  \begin{align*}
    &\Rc_j \circ\Y( \iota^{j}(L,\theta)) \cong \Y(L,\theta), & &\Rc_j\circ \Ic^{j}(\M) \cong \M \\
    & \Rc_1 \circ\Y( \iota^{j}(L,\theta))  \cong  \Y(L,\theta)[-1], & & \Rc_1\circ \Ic^{j}(\M) \cong \M[-1]\\
    &\Rc_1 \circ\Y( \iota^{(j,-1)}(L,\theta))  \cong  \Y(L,\theta), & & \Rc_1\circ \Ic^{(j,-1)}(\M) \cong \M.
  \end{align*}
\end{rem}

%%% Local Variables:
%%% mode: latex
%%% TeX-master: "../SCLC-arxiv"
%%% End:

%% file: SC_and_main_result.tex
\section{A Stability Condition on $D\fuk(\C \times M)$.}\label{Section: Stab.Cond.}
%\chapter{A Stability Condition on $D\fuk(\C \times M)$}\label{Section: Stab.Cond.}

In this section we will give the definition of a stability condition on a triangulated category.
Moreover will define a candidate for a stability condition on $D\fuk(\C\times M)$ under the assumption that there exists a stability condition on $D\fuk(M)$, state the main result, and discuss an example.\par

\subsection{Stability conditions and related notions}
We will now give the definition of stability condition and discuss some related concepts.
The definition of the Grothendieck group $K_0(\D)$ of a triangulated category $\D$, as well as some properties of exact triangles can be found in Appendix~\ref{subsubsec:cones}.

\begin{df}[Stability condition \cite{Bridgeland-SC}]\label{df:SC}
  A \emph{stability condition} on a triangulated category $\D$ is a pair $(Z,\P)$ consisting of an additive group homomorphism $Z:K_0(\D)\to \C$ and a collection of full additive subcategories $\P(\phi)\subset \D$ for each $\phi\in \R$, satisfying the following axioms:
  \begin{enumerate}[label=\textbf{(A\arabic*)}]
    \item\label{axiom-1} if $E\in \P(\phi)$ then $Z([E])=m(E)\exp (i\pi \phi)$ for some $m(E)\in\R_{>0}$,
    \item\label{axiom-2} for all $\phi\in \R$, $\P(\phi+1) = \P(\phi)[1]$,
    \item\label{axiom-3} if $\phi_1>\phi_2$ and $A_j\in \P(\phi_j)$ for $j=1,2$ then $\Hom_\D(A_1,A_2)=0$,
    \item\label{axiom-4} for each nonzero object $E\in \D$ there is a finite sequence of real numbers
      $$
      \phi_1>\phi_2>\ldots > \phi_n
      $$
      and a collection of exact triangles
      $$\xymatrix@R16pt@C8pt{
        0=E_0 \ar[rr]  & & E_1 \ar[dl]\ar[rr]  & & E_2 \ar[dl]\ar[r] & \cdots  \ar[r] & E_{n-1} \ar[rr]  & & E_n=E \ar[dl]\\
        & A_1\ar@{-->}[ul] & & A_2\ar@{-->}[ul] &&& & A_n\ar@{-->}[ul]
      }$$
      with $A_j\in \P(\phi_j)$ for all $j$.
      The dotted arrows represent morphisms of degree 1.
  \end{enumerate}
A collection $\P$ of full and additive subcategories $\P(\phi)\subset \D$ for each $\phi\in\R$ satisfying \ref{axiom-2}-\ref{axiom-4} is called a \emph{slicing} of $\D$.
The objects $E\in\P(\phi)$ are called \emph{semistable objects} of phase $\phi$.
An additive group homomorphism $Z:K_0(\D)\to \C$ satisfying \ref{axiom-1} is called \emph{central charge}.
\end{df}
The decomposition in \ref{axiom-4} is called \emph{Harder-Narasimhan filtration} or \emph{HN-fil\-tra\-tion}, it is unique (up to isomorphism of the objects $A_j\in \P(\phi_j)$) as a consequence of the axioms.
A proof of this fact can be found in Appendix~\ref{appendix:HN-filtration}.

\begin{notation}
  The phase of a semistable object $E$ is denoted by $\phi(E)$, i.e. $E\in\P(\phi(E))$.
\end{notation}

We would like to mention the following useful lemma whose proof can be found in \cite[Lemma 5.2]{Bridgeland-SC}.

\begin{lem}\label{lem:P(phi)-abelian}
  If $(Z,\P)$ is a stability condition on a triangulated category $\D$, then, for each $\phi\in \R$, the subcategory $\P(\phi)\subset \D$ is \emph{abelian}.
\end{lem}

\begin{df}\label{df:stable-Jordan-Hoelder}
\begin{enumerate}
  \item Let $(Z,\P)$ be a stability condition on a triangulated category $\D$. 
    The simple objects (i.e. the objects without any proper subobjects or quotients) of the abelian category $\P(\phi)$ are called \emph{stable} objects of phase $\phi$.
  \item An abelian category $\mathcal{E}$ is of \emph{finite length} if every object $E\in\mathcal{E}$ admits a \emph{Jordan-H\"older filtration}, namely a finite sequence of subobjects
$$
0=E_0\subset E_{1}\subset \ldots E_{n-1}\subset E_n=E
$$
such that each quotient $E_j/E_{j-1}$ is a stable object.

  \item A stability condition $(Z,\P)$ on $\D$ is \emph{discrete} if the image of the central charge $Z:K_0(\D)\to \C$ is a discrete subgroup of $\C$.
  \item  We say that the object $Y$ is an \emph{extension} of $Z$ by $X$ if there exists an exact triangle $X\to Y\to Z\to X[1]$ in $\D$.
A subcategory $\D'$ of a triangulated category $\D$ is \emph{extension-closed}, if for each exact triangle $X\to Y\to Z\to X[1]$ with $X$ and $Z$ in $\D'$, we also have that $Y$ is in $\D'$ (see also \cite[1.2.6]{Faisceaux-Pervers}).

  \item For each interval $I\subset \R$, $\P(I)\subset \D$ is defined to be the extension-closed subcategory generated by the family of subcategories $\P(\phi)$, $\phi\in I$.
\end{enumerate}
\end{df}

\begin{rem}
 For any $\phi\in \R$, the two extension closed subcategories $\P((\phi,\phi+1])$ and $\P([\phi,\phi+1))$ of $\D$ are \emph{abelian}.
 This follows from the fact that a slicing $\P$ on a triangulated category $\D$ gives rise to $t$-structures (for more on $t$-structures see e.g. \cite[Sect.~1.3]{Faisceaux-Pervers} or \cite{huybrechts-introSC}) which have the two subcategories mentioned above as their \emph{hearts} (see~\cite[Def.~1.3.1]{Faisceaux-Pervers}) and the heart of a $t$-structure is always an abelian subcategory (see \cite[Thm.~1.3.6]{Faisceaux-Pervers}).
\end{rem}

The proof of the following auxiliary lemma is given in \cite[Lemma~4.3]{Bridgeland-SC}, it allows us to define the notion of a \emph{locally-finite} slicing \cite[Def.~5.7]{Bridgeland-SC} which will be of relevance in what follows.
The notion of \emph{quasi-abelian} category is not of prime importance for us, nevertheless we give its definition together with a few facts in Appendix~\ref{appendix:qac}.
For us, the only important fact about quasi-abelian categories is that one can define the notion of strict subobjects and strict quotients which allows us to make sense of the notion of \emph{finite-length} in a quasi-abelian category.

\begin{lem}
  Let $\P$ be a slicing on a triangulated category $\D$ and $I\subset \R$ any interval of length $<1$.
  Then, the full subcategory $\P(I)\subset \D$ is \emph{quasi-abelian} and the strict short exact sequences in $\P(I)$ are in one-to-one correspondence with those exact triangles in $\D$ for which all vertices are objects of $\P(I)$.
\end{lem}

\begin{df}\label{df:locally-finite}
  A slicing $\P$ of a triangulated category $\D$ is \emph{locally-finite} if there exists $\eta\in (0,\frac12)$ such that for each $\phi\in\R$, the quasi-abelian subcategory $\P((\phi-\eta,\phi+\eta))\subset \D$ is of finite length.
  A stability condition $(Z,\P)$ on $\D$ is said to be \emph{locally-finite} if the slicing $\P$ is locally-finite.
\end{df}

\begin{rem}\label{rem:extension-closed+JH-filtration}
  \begin{enumerate}[label={(\roman*)}, ref=\ref{rem:extension-closed+JH-filtration}~(\roman*)]
    \item If $E,F\in \D$ are stable with $\phi(E)\geq \phi(F)$, then they are either isomorphic or $\Hom_\D(E,F)=0$.
          This follows from the definitions and Lemma~\ref{lem:P(phi)-abelian}.

    \item \label{rem:extension-closed}
  Let $(Z,\P)$ be a stability condition on a triangulated category $\D$ and let $E\to F\to C\to E[1]$ be an exact triangle with $E\in\P(\phi_1)$ and $F\in\P(\phi_2)$.
Suppose that $\phi_2\leq \phi_1+1$.
  By rotating the triangle we get an exact triangle $F\to C\to E[1]\to F[1]$ with $F,E[1]\in \P([\phi_2,\phi_1+1])$.
  Since $\P([\phi_2,\phi_1+1])$ is extension-closed by definition, it follows that $C\in \P([\phi_2,\phi_1+1])$ as well.
  A similar result holds if we assume that $\phi_2\geq \phi_1+1$.
  
    \item An abelian category is of finite length if and only if it is artinian and noetherian (see e.g.~\cite[Ex.~8.20]{Kashiwara-Schapira-Categories_and_Sheaves}).

    \item \label{rem:JH-filtration} If  $(Z,\P)$ is a locally-finite stability condition on $\D$, the HN-filtration of a non-zero object $E\in\D$ has a (finite) refinement where all the semistable factors $A_i$ are in fact stable and $\phi_1\geq\phi_2\geq\ldots\geq\phi_n$.
          This refinement is the \emph{Jordan-H\"older filtration}, it does not contradict the uniqueness of the HN-filtration since the phases are not strictly decreasing.
          The Jordan-H\"older filtration is in general not unique, the stable factors $A_i$ of $E$ are unique only up to permutation.
          If $E$ is semistable, then all the stable factors of $E$ have the same phase $\phi(E)$.
          In particular, if the stability condition is locally-finite, then each $\P(\phi)$ is a subcategory of finite length and each semistable object has a Jordan-H\"older filtration into stable factors of the same phase.
  \end{enumerate}
\end{rem}

\begin{notedf}
  For the following lemma, recall that $(E\to F)$ stands for the cone (see Appendix~\ref{subsubsec:cones}).
\end{notedf}

\begin{lem}\label{lem:semistable-extension-closed}
  Let $(Z,\P)$ be a locally-finite stability condition on a triangulated category $\D$.
  For each $\phi\in \R$ the subcategory $\P(\phi)$ is extension-closed, that is, if $E,F\in \P(\phi)$ and $E\to C\to F \to E[1]$ is an exact triangle, then $C\in\P(\phi)$ as well.
\end{lem}
\begin{proof}
  Since the stability condition is locally-finite, we may write
  $$
  E\cong (E_r[-1]\to\ldots \to E_1[-1]\to 0) \quad \text{and}\quad F\cong (F_s[-1]\to\ldots \to F_1[-1]\to 0)
  $$
  where $E_i$, for $1\leq i\leq r$, and $F_j$, for $1\leq j\leq s$, are the stable factors of the respective Jordan-H\"older filtrations of $E$ and $F$.
  These stable factors are all of phase $\phi$.
  By rotating the triangle $E\to C\to F \to E[1]$ we get $C \cong (F[-1]\overset{\alpha}{\to} E)$.
  If $\alpha$ is an isomorphism or the zero-morphism, there is nothing to show since $\P(\phi)$ is an additive subcategory.
  So, suppose that $\alpha\neq 0$ is not an isomorphism.
  We can express $C$ as
  \begin{align*}
    C & \cong ((F_s[-2]\to\ldots \to F_1[-2]\to 0) \to (E_r[-1]\to\ldots \to E_1[-1]\to 0))\\
      & \cong (F_s[-1]\to\ldots \to F_1[-1] \to E_r[-1]\to\ldots \to E_1[-1]\to 0)
  \end{align*}
  and hence $F_j$, $1\leq j\leq s$, and $E_i$, $1\leq i\leq r$, are the stable factors of the Jordan-H\"older filtration of $C$.
  The stable factors of $C$ are unique up to permutation and they are all of phase $\phi$.
  The Jordan-H\"older filtration is a refinement of the unique Harder-Narasimhan filtration, and hence we conclude that $C$ is indeed semistable of phase $\phi$.
\end{proof}

The following lemma, proven by Bridgeland in \cite{Bridgeland-SC_on_K3}, contains a sufficient condition for a stability condition to be locally-finite.
\begin{lem}[see {\cite[Lemma 4.4]{Bridgeland-SC_on_K3}}]\label{lem:discrete-SC-locally-finite}
Suppose that $(Z,\P)$ is a discrete stability condition on $\D$, and let $0<\epsilon<\frac12$.
Then for every $\phi\in\R$, the subcategory $\P((\phi-\epsilon,\phi+\epsilon))$ is of finite length.
In particular, the stability condition is locally-finite.  
\end{lem}

%%%%%%%%%%%%%%%%%%%%%EXAMPLE%%%%%%%%%%%%%%%%%%%%%%%%%%%
\begin{ex}\label{ex:SC-on-DCoh-EllCurve}
  The following is a well known example of a stability condition on the bounded derived category of coherent sheaves over a smooth projective curve.
  We will focus on the case of elliptic curves, since this is the situation that will be of use in Section~\ref{sec:example-T^2}.
  For more details see also \cite{huybrechts-introSC, Bridgeland-SC}.\par

  Let $C$ be an elliptic curve over a (not necessarily algebraically closed) field $\mathbf{k}$.
  By $\mathrm{Coh}(C)$ we denote the category of coherent sheaves over $C$ and $D^b\mathrm{Coh}(C)$ denotes the bounded derived category of coherent sheaves over $C$.
  If $E\in\mathrm{Coh}(C)$, then the cohomology groups $H^q(C,E)$ are finite dimensional vector spaces (\cite[Thm.~2.3.1]{LePotier-VB}), so we may define the \emph{Euler-Poincar\'e characteristic} of $E$ as
$$
\chi(E) = \sum_q (-1)^q\dim H^q(C,E) = \dim H^0(C,E)-\dim H^1(C,E).
$$
  Moreover, there is an open and dense subset $U\subset C$ over which $E|_U$ is a locally free $\mathcal{O}(U)$-module (\cite[Lemma~2.6.1]{LePotier-VB}).
  The \emph{rank} $\rk(E)$ of $E$ is defined as the rank of the locally free sheaf $E|_U$.
  Now, the \emph{degree} of a rank $r$ coherent sheaf $E$ on $C$ is defined as (see~\cite[Sect.~2.6]{LePotier-VB})
  \begin{align}\label{eq:def-deg}
  \deg (E) := \chi(E) - r \chi(\mathcal{O}_C).
  \end{align}
  Since $C$ is an elliptic curve we have that $\deg (E) = \chi(E)$.
  This follows from the fact that $\chi(\mathcal{O}_C)=\frac12 \chi(C) = 1- g(C) = 0$, which is a consequence of the Hodge decomposition (cf.~\cite{Griffiths-Harris-AG}) and Serre duality.\par % Serre duality: See Vakil's book: Thm.18.5.1 (see also 18.5.A)  %% Hodge decomp.: GH-p.116
  We define $Z':\mathrm{Coh}(C)\setminus \{0\}\to \{z\in\C\, |\, \mathrm{Im}(z)>0\}\cup \R_{<0}\subset \C$ by 
  $$
  Z'(E):= -\deg(E)+i\cdot \rk(E)
  $$
  for each non-zero coherent sheaf $E$ on $C$.
  The phase $\phi(E)\in (0,1]$ of a non-zero sheaf $E$ is determined by
  $$Z'(E)\in\exp(i\pi\phi(E))\cdot\R_{>0}.$$
  Moreover we set the phase of its shifts $E[k]\in D^b\mathrm{Coh}(C)$ to 
  $$\phi(E[k]):=\phi(E)+k$$
  for $k\in\Z$.\par
  By definition, a coherent sheaf $E\neq 0$ on $C$ is \emph{semistable} (resp.~\emph{stable}) if for every proper subsheaf $0\neq F\subset E$, the phase inequality $\phi(F)\leq \phi(E)$ (resp.~$\phi(F)< \phi(E)$) is satisfied.
  It is well known that if $E$ and $F$ are two semistable sheaves with $\phi(E)>\phi(F)$, then $\Hom(E,F)=0$ (see e.g.~\cite{LePotier-VB, huybrechts-introSC}).\par 
  % A non-zero object of $D^b\mathrm{Coh}(C)$ is \emph{semistable}, if it is isomorphic to a shift of a semistable sheaf (here we identify the category of coherent sheaves over $C$ with its image under the full and faithful embedding into $D^b\mathrm{Coh}(C)$).\par
  By a result of Harder and Narasimhan \cite{Harder-Narasimhan}, every coherent sheaf $0\neq E$ on $C$ admits a unique finite (Harder-Narasimhan-) filtration
  $$
  0=E_0\subset E_1\subset\ldots\subset E_{n-1}\subset E_n=E,
  $$
  such that the quotients $A_j:=E_j/E_{j-1}$ are semistable of descending phases, i.e.~the inequalities $\phi(A_1)>\ldots >\phi(A_n)$ hold.
  In addition, by a classical result (see for example \cite[Cor.~3.15]{Huybrechts-Fourier-Mukai}), every object of $D^b\mathrm{Coh}(C)$ is isomorphic to a direct sum of shifted coherent sheaves over the smooth projective curve $C$, where $\mathrm{Coh}(C)$ is identified with its image under the full and faithful embedding into $D^b\mathrm{Coh}(C)$.
  This follows from the fact that $\mathrm{Coh}(C)$ is a hereditary (i.e.~$\Ext^i(\cdot,\cdot)=0$ for $i\notin\{0,1\}$) abelian category (see e.g.~\cite[Prop.~3.13]{Huybrechts-Fourier-Mukai}).
  More precisely, any object $E^\bullet$ of $D^b\mathrm{Coh}(C)$ (i.e.~a complex of coherent sheaves) is isomorphic to the direct sum of its shifted cohomology objects, that is
  \begin{align}\label{eq:direct-sum-of-cohomology-objects}
    E^\bullet \cong  \bigoplus_{i\in\Z}H^i(E^\bullet)[-i],
  \end{align}
  where $H^i(E^\bullet)[-i]\in \mathrm{Coh}(C)$ for $i\in \Z$.
  Moreover, if we consider the truncation of a complex $E^\bullet$ given by
  \begin{align*}
    \tau_{\leq i}(E^\bullet) := (\ldots \to E^{i-2} \overset{\delta^{i-2}}{\to} E^{i-1}\overset{\delta^{i-1}}{\to} \ker(\delta^i) \to 0\to \ldots),
  \end{align*}
  there are exact triangles
  \begin{align}\label{eq:truncation-triangles}
    \tau_{\leq i-1}(E^\bullet)\to \tau_{\leq i}(E^\bullet) \to H^i(E^\bullet)[-i]\to \tau_{\leq i-1}(E^\bullet)[1].
  \end{align}
  As any object $E^\bullet $of $D^b\mathrm{Coh}(C)$ is a bounded complex, it admits, by successively truncating, a finite filtration whose factors are shifts of the cohomology objects.\par
 A non-zero object of $D^b\mathrm{Coh}(C)$ is \emph{semistable}, if it is isomorphic to a shift of a semistable sheaf.\par% (here we identify the category of coherent sheaves over $C$ with its image under the full and faithful embedding into $D^b\mathrm{Coh}(C)$).\par

  From these observations, together with the fact that each non-zero coherent sheaf admits a Harder-Narasimhan filtration, we deduce that  each non-zero object of $D^b\mathrm{Coh}(C)$ admits a unique filtration by semistable objects of strictly descending phases as well (i.e.~a Harder-Narasimhan filtration).
  Also, the map $Z'$ induces a well-defined additive homomorphism
  $$
  Z:K_0(D^b\mathrm{Coh}(C))\To \C.
  $$
  
  In conclusion, we have obtained a stability condition on the bounded derived category $D^b\mathrm{Coh}(C)$ of coherent sheaves over a smooth projective curve.\par
  One can say more about this stability condition.
  Namely, since the image of the homomorphism $Z:K_0(D^b\mathrm{Coh}(C))\To \C$ is a \emph{discrete} subgroup of $\C$, it follows from Lemma~\ref{lem:discrete-SC-locally-finite}, that the stability condition is in fact \emph{locally-finite}.
Alternatively, this also follows from Theorem~\ref{thm:semistable-indecomposable-sheaf} below.\par
  In order to make this example more transparent we will collect some well known facts about (semi)stable sheaves on an elliptic curve, following \cite{Bodnarchuk-Burban-Drozd-Greuel} for the most part.\par
  Historically, Atiyah \cite{Atiyah-VBEC} classified vector bundles over elliptic curves before the notion of (semi)stable sheaves was introduced by Mumford \cite{Mumford-63}.
  Many of the results that follow have their roots in Atiyah's work but they are cast in a more up to date framework.
  A modern proof of the following result can be found in \cite[Thm.~10]{Bodnarchuk-Burban-Drozd-Greuel}.
  \begin{thm}\label{thm:semistable-indecomposable-sheaf}
    If $C$ is an elliptic curve over a field $\mathbf{k}$, then
    \begin{enumerate}
      \item Any indecomposable coherent sheaf $E\in \mathrm{Coh}(C)$ is semistable.
      \item If $E\in \mathrm{Coh}(C)$ is semistable and indecomposable then all its Jordan-H\"older factors are isomorphic.
      \item A coherent sheaf $E$ is stable if and only if $\mathrm{End}(E) = \mathbf{K}$ for some finite field extension $\mathbf{k}\subset \mathbf{K}$.            
    \end{enumerate}
  \end{thm}
  
Note that if $\mathbf{k}$ is algebraically closed, which we will assume from now on, then $\mathbf{K}=\mathbf{k}$.\par
Recall that, given a point $p\in C$, we can form the following short exact sequence of sheaves
\begin{align*}
  0\To \OO (-p)\To \OO_C\To \mathbf{k}(p)\To 0,
\end{align*}
where $\mathbf{k}(p)$ is the so called \emph{skyscraper sheaf} supported at $p$.
The skyscraper sheaf is of rank $0$ and the stalk of $\kk(p)$ at $p$ is isomorphic to $\kk$ since the stalk $\mathfrak{m}_p=\OO(-p)_p$ is precisely the maximal ideal in the local ring $\OO_{C,p}$.
Therefore, as a consequence of Theorem~\ref{thm:semistable-indecomposable-sheaf}, the skyscraper sheaves $\kk(p)$ are \emph{stable}.
Since for every point $p\in C$, the stalk $\OO_{C,p}$ is a local principal ideal domain it follows from the classification of finitely generated torsion modules over principal ideal domains (see e.g.~\cite[Thm.~7.5]{Lang-Algebra}) that indecomposable coherent torsion sheaves on $C$ are parametrized by points $p\in C$ and positive integers $\ell\in \Z_{>0}$ and are of the form $\kk(\ell\cdot p)$, where
\begin{align*}
  0\To \OO (-\ell\cdot p)\To \OO_C\To \mathbf{k}(\ell \cdot p)\To 0.
\end{align*}
By Theorem~\ref{thm:semistable-indecomposable-sheaf} the only stable torsion sheaves are the skyscraper sheaves $\kk(p)$.\par
For two coherent sheaves $E, F$ on $C$ we define the anti symmetric bilinear \emph{Euler form} by
\begin{align*}\label{eq:Euler-Form-1}
  \langle E, F \rangle := \rk(E)\deg(F) - \rk(F)\deg(E).
\end{align*}
Since the rank and the degree are both additive they are well defined on $K_0(\mathrm{Coh}(C))$.
Moreover, by \eqref{eq:direct-sum-of-cohomology-objects}, \eqref{eq:truncation-triangles} and Remark~\ref{rem:K-group-cone-sign-relation} the Euler form induces a well defined form on $K_0(D^b\mathrm{Coh}(C))$.
Note also that the action of the group of triangulated auto-equivalences $\mathrm{Aut}(D^b\mathrm{Coh}(C))$ on $K_0(D^b\mathrm{Coh}(C))$ preserves the Euler form (this follows from Remark~\ref{rem:Euler-form}).
Since $Z([\OO_C])=i$ and $Z([\kk(p)])=-1$, we get that the homomorphism
\begin{align*}
  Z:K_0(D^b\mathrm{Coh}(C))\To \C
\end{align*}
maps onto $\Z\oplus i\cdot \Z\subset \C$.
The pullback of the non-degenerate anti-symmetric bilinear form on $\Z\oplus i\cdot \Z$, which is induced by the restriction of the volume form $\frac{i}{2}\dif z \wedge \dif \bar{z}$ on $\C$, coincides with the Euler form.
Hence the radical
\begin{align*}
  \mathrm{rad}\langle\cdot,\cdot\rangle = \{E^\bullet\in D^b\mathrm{Coh}(C) \,|\, \langle E^\bullet,\cdot\rangle = 0\}
\end{align*}
coincides with the kernel of $Z$.
Therefore we obtain an isomorphism
%\begin{align*}
 $$ K_0(D^b\mathrm{Coh}(C)) /  \mathrm{rad}\langle\cdot,\cdot\rangle \cong \Z^2$$
%\end{align*}
and a group homomorphism $\mathrm{Aut}(D^b\mathrm{Coh}(C))\to SL(2,\Z)$.

\begin{rem}\label{rem:Euler-form}
  The induced Euler form on $K_0(D^b\mathrm{Coh}(C))$ can be expressed as follows:
  \begin{align*}
    \langle E^\bullet, F^\bullet \rangle = \dim \Ext^0(E^\bullet, F^\bullet) - \dim \Ext^1(E^\bullet, F^\bullet).
  \end{align*}
  In order to see this, it is enough to check it on generators of the $K_0$-group, e.g. on locally free sheaves (see also~\cite[Prop.~2.6.6]{LePotier-VB}).
  Let $E,F$ be locally free sheaves and $E^\vee$ the dual of $E$.
  Then we have the following:
  \begin{align*}
    \dim \Ext^0(E, F) - \dim \Ext^1(E, F) &= \dim \Ext^0(\OO_C, E^\vee\mkern-3mu\otimes \mkern-2mu F) - \dim \Ext^1(\OO_C,E^\vee\mkern-3mu\otimes\mkern-2mu F)\\
    &= \chi(E^\vee\mkern-3mu\otimes \mkern-2mu F)\\
    &= \deg(E^\vee\mkern-3mu\otimes \mkern-2mu F)\\
    &= \rk(E)\deg(F) - \rk(F)\deg(E).
  \end{align*}
The third equality follows from \eqref{eq:def-deg} and the last equality follows for example by using the properties of the Chern character (see e.g.~\cite[Ex.~3.2.3]{Fulton-IT}).
\end{rem}

Using the concept of \emph{Fourier-Mukai-transforms} one can show the following result, see Mukai \cite{Mukai81-Duality} and \cite{Bodnarchuk-Burban-Drozd-Greuel} for a proof.
\begin{thm}\label{thm:Aut-to-SL-surj}
  The group homomorphism $\mathrm{Aut}(D^b\mathrm{Coh}(C))\to SL(2,\Z)$ is surjective.
\end{thm}

Now, let $\phi\in (0,1)$ be any phase, such that $\P(\phi)\subset \mathrm{Coh}(C)$ is non-empty.
Pick $r,d\in \Z$ coprime, satisfying $r>0$ and $-d+i\cdot r\in e^{i\pi\phi}\cdot \R_{>0}$ (such $r$ and $d$ exist by definition of the phase $\phi$).
We can then find an element $A\in SL(2,\Z)$ which satisfies $A\cdot\left(\begin{smallmatrix}-d \\ r\end{smallmatrix}\right) = \left(\begin{smallmatrix}-1 \\ 0\end{smallmatrix}\right)$, and by the previous theorem a lift $\Psi\in \mathrm{Aut}(D^b\mathrm{Coh}(C))$ of $A$.
For any indecomposable sheaf $E\in \P(\phi)$, $\Psi(E)$ is an indecomposable object of $D^b\mathrm{Coh}(C)$ with $Z(\Psi(E))\in e^{i\pi}\cdot \R^*$ and therefore it is a shift, say by $k\in \Z$, of an indecomposable torsion sheaf.
As $\Psi$ is a triangulated auto-equivalence, the composition of $\Psi$ with the shift by $-k$ yields an equivalence of $\P(\phi)\simeq \P(1)$.
Recall also that $\P(\phi+1)=\P(\phi)[1]$.
Combining this we have shown the following theorem (cf.~\cite{Bodnarchuk-Burban-Drozd-Greuel}).

\begin{thm}\label{thm:equivalence-of-semistable-subcat}
  For any $\phi\in \R$ such that $\P(\phi)\subset D^b\mathrm{Coh}(C)$ is non-empty, there is a triangulated auto-equivalence of $D^b\mathrm{Coh}(C)$ inducing an equivalence of categories $\P(\phi)\simeq \P(1)$.
  In particular, an indecomposable coherent sheaf of rank $r$ and degree $d$ is stable if and only if $\gcd(r,d)=1$.
  Moreover, indecomposable semistable coherent sheaves are parametrized by their rank $r$ and degree $d$ and a point on the curve, i.e. by $(r,d)\in\Z^2$ and $p\in C$.
\end{thm}

This together with the classification of indecomposable torsion sheaves provides us with a good understanding of what the (semi)stable objects of $D^b\mathrm{Coh}(C)$ are.

\end{ex}

\subsection{Construction of the stability condition on $D\fuk(\C\times M)$}

Suppose now that $(Z^M,\P^M)$ is a given stability condition on the derived Fukaya category $D\fuk(M)$ of $M$.
Given any $\kappa\in 2\cdot \Z_{>1}$ (i.e. $\kappa \geq 4$ is an \emph{even} integer), we define the slicing $\P_\kappa$ on $D\fuk(\C\times M)$ as follows:

\begin{center} \emph{for each $\phi\in\R$, $\P_\kappa(\phi)$ is the full \emph{additive} subcategory generated by the objects}\end{center}
\begin{equation}\label{df:slicing-phase}
\{ \Ic^{{h}}X[r] \, |\, r\in \Z, \, h\in\Z_{>1},\, X\in\P^M(\phi-r+\kappa h) \}.  
\end{equation}

\begin{notation}
From now on we will fix $\kappa\in 2\cdot \Z_{>1}$ and simply write $\P$ for $\P_\kappa$, leaving the dependency on $\kappa$ implicit.
\end{notation}
\begin{rem}
  The reason why we choose $\kappa\in 2\cdot \Z_{>1}$ will be made clear in the proof of Theorem~\ref{thm:(Z,P)-is-SC}.
\end{rem}
The class of Yoneda-modules generates $K_0(D\fuk(\C\times M))$. % (see Proposition~\ref{Prop:conedecomp} below).
Therefore, in order to define the central charge it is enough to define it on Yoneda-modules, extend it linearly, and show that it is well-defined.
Recall that the Yoneda-embedding is cohomologically full and faithful and we will sometimes, by abuse of notation, write $(V,\theta_V)$ both for the corresponding object of $\fuk(\C\times M)$ as well as for its Yoneda-module $\Y(V,\theta_V)$.
From the context it should be clear what is meant. %\par

\begin{rem}
 It is worth pointing out the relation of the phases of semistable objects of $D\fuk(M)$ and of $D\fuk(\C\times M)$ more explicitly.
 To this end we will denote the collections of all semistable objects in the respective derived categories by
 $$
 \P^M_\mathrm{tot}:=\bigcup_{\phi\in \R}\P^M(\phi)\subset D\fuk(M) , \quad\text{and}\quad \P_\mathrm{tot}:=\bigcup_{\phi\in \R}\P(\phi)\subset D\fuk(\C\times M)
 $$
 and the maps that associate to each semistable object its phase by
 $$
  \Phi^M:\P^M_\mathrm{tot}\to\R:\ X\mapsto \phi(X)\quad\text{and} \quad \Phi^{\C\times M}:\P_\mathrm{tot}\to\R:\ E\mapsto \phi(E).
 $$
 Then, for any semistable object $X$ of $D\fuk(M)$ and any $h\in\Z_{>1}$, we get the relation
 \begin{align}\label{eq:inclusion-phase-relation}
 \Phi^{\C\times M}(\Ic^h X) = \Phi^M(X) - \kappa h.
 \end{align}
 In particular, note that if $(L,\theta)\in\P^M(\phi)$ then $\Ic^{h}(L,\theta)\in\P(\phi-\kappa h)$. 
\end{rem}

\begin{rem}
  Recall that the direct sum $E\oplus F$ (i.e. the cone over the zero morphism) of two semistable objects $E,F\in\P(\phi)$ is again semistable of phase $\phi$, since $\P(\phi)$ is an \emph{additive} subcategory.
\end{rem}

Define the central charge on Yoneda-modules as follows:
$$
Z:K_0(D\fuk(\C\times M)) \To \C\,;\quad [V,\theta_V]_{K_0}\Mapsto \sum_{j\geq 2} -Z^M([\Rc_{h_j}(V,\theta_V)[-1]]_{K_0}),
$$ 
and extend it linearly.
The motivation for this definition comes from the cone decomposition \eqref{prop: cone decomp in MxC} which we will discuss in Section~\ref{sec:cone_decomp_Lagcob}.
In the proof of Theorem~\ref{thm:(Z,P)-is-SC} we will see that $Z$ is well-defined (see Sec.~\ref{Sec:proof_of_main_result}).

Notice that $(L,\theta)\cong ((L,\theta)[-1]\to 0)$, hence in the Grothendieck group we have that $[L,\theta]_{K_0}=-[(L,\theta)[-1]]_{K_0}$ and therefore we can rewrite
\begin{align}\label{eq:def_central_charge}
Z([V,\theta_V]_{K_0}) = \sum_{j\geq 2} Z^M([\Rc_{h_j}(V,\theta_V)]_{K_0}).  
\end{align}

We now state the main theorem whose proof is deferred to Section~\ref{Sec:proof_of_main_result}.
\begin{thm}\label{thm:(Z,P)-is-SC}
  If $(Z^M,\P^M)$ is a locally-finite stability condition on $D\fuk(M)$ and $\kappa\in 2\cdot \Z_{>1}$ then $(Z,\P_\kappa)$, as defined above, is a locally-finite stability condition on $D\fuk(\C\times M)$.
\end{thm}

\subsection{A consequence and an example}\label{sec:consequence-example}

In \cite{Le-Chen} Le and Chen show that if a triangulated category $\D$ admits a bounded $t$-structure, then it is \emph{split-closed} (or \emph{Karoubian}).
Moreover, the existence of a stability condition on a triangulated category implies the existence of a bounded $t$-structure on it (see \cite[Prop. 2.4]{huybrechts-introSC}).
Combining these results with Theorem~\ref{thm:(Z,P)-is-SC} we get the following corollary.

\begin{cor}\label{cor:DF(M)loc.fin.SC-DF(CxM)split-closed}
  If $D\fuk(M)$ admits a locally-finite stability condition, then the derived Fukaya category of cobordisms $D\fuk(\C\times M)$ is \emph{split-closed}.
\end{cor}
%\qed
\begin{ex}\label{ex:T^2-cob-split-closed}
  By homological mirror symmetry for the 2-torus $T^2$ there is an equivalence of triangulated categories (see~\cite{Abouzaid-Smith-T^4})
  $$
  D^b\mathrm{Coh}(X)\simeq D^\pi\fuk^{\sharp}(T^2).
  $$
  Here $X$ denotes the \emph{Tate curve} which is an elliptic curve over the Novikov field $\Lambda$ (cf.~\cite{Abouzaid-Smith-T^4}) and $D^\pi\fuk^{\sharp}(T^2)$ is the split-closure of the derived Fukaya-category.
  The $\sharp$ signifies that we take extra data, comprised of certain local systems and Pin-structures, into account (cf.~Remark~\ref{rem:extra-data}).
  For more details we refer to Haug \cite{Haug-T^2}.
  Haug has shown that the inclusion $D\fuk^{\sharp}(T^2)\hookrightarrow D^\pi\fuk^{\sharp}(T^2)$ is an equivalence (\cite[Cor.~7.5]{Haug-T^2}) and hence we obtain a triangulated equivalence $D^b\mathrm{Coh}(X)\simeq D\fuk^{\sharp}(T^2)$.
  In Example~\ref{ex:SC-on-DCoh-EllCurve} we have seen that the bounded derived category $D^b\mathrm{Coh}(X)$ of coherent sheaves over the elliptic curve $X$ admits a locally-finite stability condition.
  Thus, by the above equivalence, $D\fuk^{\sharp}(T^2)$ admits a locally-finite stability condition as well.
  Theorem~\ref{thm:(Z,P)-is-SC} implies that the derived Fukaya category of cobordisms $D\fuk^{\sharp}(\C\times T^2)$ admits a locally-finite stability condition, and moreover by Corollary~\ref{cor:DF(M)loc.fin.SC-DF(CxM)split-closed} it is \emph{split-closed}.
\end{ex}

%%% Local Variables:
%%% mode: latex
%%% TeX-master: "../SCLC-arxiv"
%%% End:

%% file: Cone_decomp.tex
\section{Iterated cone decompositions induced by Lagrangian Cobordisms}\label{sec:cone_decomp}
%\chapter{Iterated cone decompositions induced by Lagrangian Cobordisms}\label{sec:cone_decomp}

Before we can begin with the proof of Theorem~\ref{thm:(Z,P)-is-SC} we need another ingredient.
Namely the fact that a Lagrangian cobordism gives rise to an iterated cone decomposition in the derived Fukaya category of cobordisms.
This was shown by Biran and Cornea in \cite[Prop.~4.3.1]{BC-Lef} in the ungraded setting and over the coefficient field $\Z/2\Z$.
We will explain the main idea of the proof of this result and discuss some remarks on iterated cone decompositions.

\subsection{Iterated cone decompositions in $D\fuk(\C\times M)$ via Lagrangian cobordisms}\label{sec:cone_decomp_Lagcob}
We will now adapt Proposition 4.3.1 of \cite{BC-Lef} to our setting.
Haug \cite[Section~4]{Haug-T^2} adapted an analogous iterated cone decomposition of \cite{BC-CobI} to the oriented and graded setting.
Incorporating Haug's results, the proof of the following proposition is very similar to the original proof by Biran and Cornea in \cite{BC-Lef}.
We will explain the main idea in Section \ref{subsect:idea-of-proof-of-conedecomp}.
The notation regarding iterated cone decompositions can be found in Appendix~\ref{subsubsec:cones}.

\begin{prop}\label{Prop:conedecomp}
 Let $(V,\theta_V)$ be an object of $\fuk(\C\times M)$ with Lagrangian $L_j\subset M$ over the horizontal end on height $h_j\in\Z_{\geq 1}$ for $1\leq j\leq s$ and with $1=h_1<\ldots < h_s$.
Then the Yoneda-module $\M_{V,\theta_V}=\Y(V,\theta_V)$ admits a cone decomposition in $D\fuk(\C\times M)$ of the form:

\begin{align}
\M_{V,\theta_V} \cong (\Ic^{{h_s}}\circ \Rc_{h_s}(\M_{V,\theta_V})[-1]  &\to \Ic^{{h_{s-1}}}\circ \Rc_{{h_{s-1}}}(\M_{V,\theta_V})[-1]\to \ldots \label{prop: cone decomp in MxC} \\\nonumber
  \ldots \to \Ic^{{h_3}}\circ \Rc_{h_{3}}(\M_{V,\theta_V})[-1] &\to \Ic^{{h_2}}\circ \Rc_{h_2}(\M_{V,\theta_V})).
\end{align}
\end{prop}

\begin{rem}
In the above proposition, we assume that $V$ has $h_1=1$. If this is not the case we artificially set $L_1=\emptyset$ with $h_1=1$.
Note that, even though this is a \emph{cohomological} version of the proposition, the arrows are not reversed with respect to the setting in \cite{BC-Lef}.
This is because of the way we have chosen the perturbation datum and since we are considering \emph{positively} ended cobordisms.
\end{rem}

This cone decomposition can also be written as

\begin{align*}
\M_{V,\theta_V} \cong (\Ic^{{h_s}}\circ \Rc_{h_s}(\M_{V,\theta_V})[-1]  &\to \Ic^{{h_{s-1}}}\circ \Rc_{{h_{s-1}}}(\M_{V,\theta_V})[-1]\to \ldots \\
  \ldots \to \Ic^{{h_3}}\circ \Rc_{h_{3}}(\M_{V,\theta_V})[-1] &\to \Ic^{{h_2}}\circ \Rc_{h_2}(\M_{V,\theta_V})[-1]\to 0) 
\end{align*}

and, by \eqref{eq:restriction-shift}, as

 \begin{align}\label{eq:cone-decomp-wt(L)}
 \M_{V,\theta_V} \cong (\wt{L}_s[-1]  \to \wt{L}_{s-1}[-1]\to \ldots\to \wt{L}_3[-1]\to \wt{L}_2[-1]\to 0),
 \end{align}

where $\widetilde{L}_j$ is isomorphic to $\Ic^{{h_j}}\circ \Rc_{h_j}(\M_{V,\theta_V})=\Ic^{{h_j}}\Y(L_j,\theta_{V,j})$, for $2\leq j\leq s$.

Moreover, if we apply the restriction $\Rc_1$ to both sides of this decomposition, we obtain the following iterated cone decomposition in $D\fuk(M)$ (cf.~\cite{BC-CobI})

 \begin{align}\label{eq: cone-decomp in M}
 \Y(L_1,\theta_{V,1}) \cong (\Y(L_s,\theta_{V,s})[-2]  \to% \Y(L_{s-1},\theta_{V,s-1})[\kappa h_{s-1}-2]\to 
   \ldots\to \Y(L_2,\theta_{V,2})[-2]\to 0),
 \end{align}
since $\Rc_1\wt{L}_j \cong \Rc_1\circ \Ic^{{h_j}}\Y(L_j,\theta_{V,j}) \cong \Y(L_j,\theta_{V,j})[-1] = \Rc_{h_j}(\M_{V,\theta_V})[-1]$ for $2\leq j$.\par

\begin{rem}
Recall that the central charge on Yoneda-modules is given by \eqref{eq:def_central_charge}.
Decomposition \eqref{eq: cone-decomp in M} and Remark~\ref{rem:K-group-cone-sign-relation} allows us to further rewrite \eqref{eq:def_central_charge} as
$$
Z([V,\theta_V]_{K_0}) = Z^M\left( \sum_{j\geq 2}[\Rc_{h_j}(V,\theta_V)]_{K_0}\right) = -Z^M[\Rc_1(V,\theta_V)]_{K_0}.
$$
In particular, this means that $Z([V,\theta_V]_{K_0})=0$, whenever the cobordism $V$ does not have a horizontal end at height $1$ (i.e. $L_1=\emptyset$).
\end{rem}

\begin{rem}
Proposition \ref{Prop:conedecomp} provides us with a cone decomposition of the Yo\-ne\-da-module $(V,\theta_V)$ which, together with Remark \ref{rem:K-group-cone-sign-relation}, implies the following equality in the Grothendieck group:
$$
[V,\theta_V]_{K_0} = \sum_{j=2}^s -[\Ic^{{h_j}}\circ \Rc_{h_j}(V,\theta_V)[-1]]_{K_0} = \sum_{j=2}^s [\Ic^{{h_j}}\circ \Rc_{h_j}(V,\theta_V)]_{K_0}  = \sum_{j=2}^s [\widetilde{L}_j]_{K_0},
$$
where $s\in\Z$ is the number of horizontal ends of $V$ and $\widetilde{L}_j:=\Ic^{h_j}\circ \Rc_{h_j}(V,\theta_V)$.
This leads to:
\begin{align*}
Z\left( \sum_{j=2}^s  [\Ic^{{h_j}}\circ \Rc_{h_j}(V,\theta_V)]_{K_0}\right) &= \sum_{j=2}^s  Z[\Ic^{{h_j}}\circ \Rc_{h_j}(V,\theta_V)]_{K_0}\\
              &=  \sum_{j=2}^s  \sum_{k\geq 2}Z^M[\Rc_{k}\circ \Ic^{{h_j}}\circ \Rc_{h_j}(V,\theta_V)]_{K_0}\\
              &= \sum_{j=2}^s  Z^M[ \Rc_{h_j}(V,\theta_V)]_{K_0}\\
              &= Z[V,\theta_V]_{K_0}.
\end{align*}

  In the second line of the above equation the restriction functor $\Rc_{k}$ occurs and is applied to the Yoneda-module $\Ic^{{h_j}}\circ \Rc_{h_j}(V,\theta_V)=\Ic^{{h_j}}(L_j,\theta_{V,j})=\Y(\iota^{{h_j}}(L_j,\theta_{V,j}))$ of the cobordism $\iota^{{h_j}}(L_j,\theta_{V,j})$ with exactly $2$ horizontal ends.
The third equality follows from Remark \ref{rem:restriction-inclusion-relation}.
Moreover, the term in the sum running over the index $k$ is non-zero only if $k=h_j$.
\end{rem}

\subsection{Algebraic preliminaries and the main idea of the proof of Proposition~\ref{Prop:conedecomp}}
Following \cite{BC-Lef}, we will briefly recall the main idea of the proof along with some algebraic remarks on $A_\infty$-modules.

\subsubsection{$A_\infty$-submodules induce exact triangles}
This section closely follows the remarks by Biran and Cornea \cite{BC-Lef}.\par
Let $\A$ be an $A_\infty$-category, we will denote by $\QQ:=mod(\A)$ the $A_\infty$-category of $A_\infty$-modules over $\A$.
For the definition of $A_\infty$-modules and $A_\infty$-module homomorphisms see \cite[Sect.~(1j)]{seidel-book}.\par
Now, let $\NN$ and $\M$ be two $A_\infty$-modules over $\A$.
An $A_\infty$-module homomorphism $i:\NN\to \M$ is called an \emph{inclusion}, if $i^1:\NN(X)\to\M(X)$ is injective for every object $X$ of $\A$ and all the higher components $i^k$, $k\geq 2$, vanish.
If there exists an inclusion $i:\NN\to\M$, the module $\NN$ is said to be a \emph{submodule} of $\M$.
The structural maps of $\NN$ are given by restriction of the structural maps of $\M$.
Moreover, if $\NN$ is a submodule of $\M$, the quotient module $\M / \NN$ over $\A$ is defined by $\M / \NN (X) = \M(X) / \NN(X)$ together with structural maps induced by those of $\M$.
For an element $b\in\M(X)$, we will denote its equivalence class by $\overline{b}\in\M(X) / \NN(X)$.
One can form the cone $\CC := \mathrm{cone}(i)$ over $i$ which is the $A_\infty$-module (see~\cite[Sect.~(3s)]{seidel-book})
\begin{align*}
  &\CC (X) = \NN(X)[1]\oplus \M(X)\\
  & \mu_\CC^1(b_0,b_1) = (\mu_\NN^1(b_0), \mu_\M^1(b_1) + i^1(b_0))\\
  & \mu_\CC^d((b_0,b_1), a_{d-1},\ldots, a_1) = (\mu_\NN^d(b_0, a_{d-1},\ldots, a_1), \mu_\M^d(b_1, a_{d-1},\ldots, a_1)), \; \text{if}\; d\geq 2.  
\end{align*}

The cone comes along with canonical $A_\infty$-module homomorphisms (cf.~\cite[Sect.~(3f)]{seidel-book}) $\iota\in \hom^0_\QQ(\M,\CC)$ and $\pi\in \hom^1_\QQ(\CC, \NN)$ given by
$$
\iota^1(b_1)=(0,(-1)^{\abs{b_1}}b_1),\quad \pi^1(b_0,b_1)=(-1)^{\abs{b_0}-1}b_0
$$
and vanishing higher order terms.
Here $\abs{\cdot}$ denotes the degree in the graded vector space.
In addition, we have the following $A_\infty$-module homomorphisms
\begin{align}
  &\id_\M\in\hom_\QQ^0(\M,\M), &\text{given by}\quad& \id_\M^1(b)=(-1)^{\abs{b}}b, &\id_\M^d=0,\; d\geq 2,\label{eq:module-homs-id}\\
  &q\in\hom_\QQ^0(\M,\M/ \NN), &\text{given by}\quad& q^1(b)=(-1)^{\abs{b}}\overline{b}, &q^d=0,\; d\geq 2,\label{eq:module-homs-q}\\
  &\vphi\in\hom_\QQ^0(\CC,\M / \NN), &\text{given by}\quad& \vphi^1(b_0,b_1)=(-1)^{\abs{(b_0,b_1)}}\overline{b_1}, &\vphi^d=0,\; d\geq 2.\label{eq:module-homs-vphi}
\end{align}
Notice that $\abs{(b_0,b_1)}=\abs{b_1}=\abs{b_0}-1$. 
The signs occurring in these definitions are needed in order to ensure that these morphisms are actual $A_\infty$-module homomorphisms and not just pre-module homomorphisms (i.e. to ensure that $\mu_\QQ^1(\id_\M), \mu_\QQ^1(q)$ and $\mu_\QQ^1(\vphi)$ vanish).\par
Recall that any $A_\infty$-module $\M\in \QQ$ gives rise to a $H(\A)$-module $H(\M)$ by taking, for each $X\in\A$, the cohomology of $\M(X)$ with respect to the differential $\partial_\M(b)=(-1)^{\abs{b}}\mu_\M^1(b)$, and module structure induced by $b\cdot a = (-1)^{\abs{a}}\mu_\M^2(b,a)$.
Furthermore, any $A_\infty$-module homomorphism $t\in\hom(\NN,\M)$ induces a module homomorphism
$$
H(t):H(\NN)\to H(\M);\quad [b]\mapsto [(-1)^{\abs{b}}t^1(b)]
$$
on the cohomological level.
Note that the signs introduced in \eqref{eq:module-homs-id}-\eqref{eq:module-homs-vphi} cancel on the cohomological level.
In particular $H(\id_\M):H(\M)\to H(\M): [b]\mapsto [(-1)^{\abs{b}}\id_\M^1(b)]=[b]$ is the usual identity morphism of modules.
Now, we have the following commutative diagram in $\QQ$:
\begin{align}\label{diag:quotient-cone-chain-diagram}
   \xymatrix{
   \NN \ar[r]^{i} &\M \ar[r]^q &\M /\NN\\
   \NN \ar[r]^{i}\ar[u]_{\id_\NN} &\M \ar[r]^\iota\ar[u]_{\id_\M} &\CC \ar[u]_\vphi
}
\end{align}

With this preparation we will now prove the following lemma:

\begin{lem}\label{lem:submodule-exact_triangle}
  An inclusion $i:\NN\to \M$ of $A_\infty$-modules induces an exact triangle 
\begin{align*}
   \xymatrix{
  \NN \ar[rr]^{[i]} & & \M \ar[dl]^{[q]} \\
& \M / \NN \ar[ul]^{[j]}&
}
\end{align*}
in $H^0(\QQ)$, where $q\in\hom_\QQ^0(\M,\M/ \NN)$ is given as above.
\end{lem}

\begin{proof}

By \cite[Lemma~3.35]{seidel-book}
\begin{align}
  \xymatrix{
  \NN \ar[rr]^{[i]} & & \M \ar[dl]^{[\iota]} \\
& \CC \ar[ul]^{[\pi]}&
}
\end{align}
is an exact triangle in $H^0(\QQ)$.
Hence, since the higher order terms of the occuring $A_\infty$-module homomorphisms vanish, we get for each object $X$ in $\A$ the following short exact sequences of cochain complexes
\begin{align*}
 & 0\To \NN(X) \overset{i}{\To} \M(X) \overset{q}{\To} \M /\NN (X) \To 0\\
 & 0\To\M(X) \overset{\iota}{\To} \CC(X) \overset{\pi}{\To} \NN[1](X) \To 0
\end{align*}
which induce long exact sequences on cohomology

\begin{align}\label{diag:quotient-cone-cohomology-diagram}
   \resizebox{.9\hsize}{!}{\xymatrix{
  \ldots \ar[r] & H^k(\NN(X)) \ar[r]^{H^k(i)} &H^k(\M(X)) \ar[r]^{H^k(q)} &H^k(\M/\NN(X)) \ar[r]^{\delta} & H^{k+1}(\NN(X)) \ar[r]&\ldots\\
  \ldots \ar[r] & H^k(\NN(X)) \ar[r]^{H^k(i)}\ar[u]^{H^k(\id_\NN)}_\cong &H^k(\M(X)) \ar[r]^{H^k(\iota)}\ar[u]^{H^k(\id_\M)}_\cong &  H^k(\CC(X)) \ar[r]^{H^k(\pi)}\ar[u]^{H^k(\vphi)} & H^{k+1}(\NN(X)) \ar[r]\ar[u]^{H^{k+1}(\id_\NN)}_\cong&\ldots 
}}
\end{align}
Here $\delta$ is (up to sign) the connecting homomorphism given by
$$
\delta [\overline{m}] = \left[ -(i^1|_{\mathrm{im}(i^1)})^{-1}\circ \partial_\M(-1)^{\abs{m}} m \right],
$$
where $(-1)^{\abs{m}}m$ is any lift of $\overline{m}$ with respect to  $q^1$.
Note that if $(n,m)\in \CC(X)$ is a cocycle, i.e. if $\partial_\CC(n,m)=0$, we have that $\mu_M^1(m)=-i^1(n)$ and therefore
\begin{align*}
  \delta\circ H^k(\vphi)[(n,m)] &= \delta[\overline{m}]\\
                                & = \left[ -(i^1|_{\mathrm{im}(i^1)})^{-1}\circ \partial_\M(-1)^{\abs{m}} m \right] \\
                                &= \left[ -(i^1|_{\mathrm{im}(i^1)})^{-1}\circ \mu_\M^1 m \right] \\
                                &= \left[ -(i^1|_{\mathrm{im}(i^1)})^{-1}\circ i^1(-n) \right]\\
                                &=[n]
\end{align*}
and 
\begin{align*}
  H^{k+1}(\id_\NN)\circ H^k(\pi)[(n,m)] = H^{k+1}(\id_\NN) [n] = [n].
\end{align*}
This together with the commutativity of \eqref{diag:quotient-cone-chain-diagram} implies that \eqref{diag:quotient-cone-cohomology-diagram} is a commutative diagram.
By the $5$-Lemma we conclude that $\vphi$ is a quasi-isomorphism.
Since $\vphi$ is a quasi-isomorphism of $A_\infty$-modules it induces an isomorphism in the cohomological category $H^0(\QQ)$ (this is non-trivial, see \cite[Lemma~1.16]{seidel-book}).
Therefore, if we define $[j]\in\hom_{H^0(\QQ)}(\M /\NN,\NN[1])$ by 
$$[j] = [\id_\NN]\circ [\pi]\circ [\vphi]^{-1},$$
 we obtain the following commutative diagram in $H^0(\QQ)$
\begin{align}
    % \resizebox{.9\hsize}{!}{
  \xymatrix{
  \ldots \ar[r] &\NN \ar[r]^{[i]} &\M \ar[r]^{[q]} &\M/\NN \ar[r]^{[j]} & \NN[1] \ar[r]&\ldots\\
  \ldots \ar[r] & \NN \ar[r]^{[i]}\ar[u]_{[\id_\NN]}^\cong &\M \ar[r]^{[\iota]}\ar[u]_{[\id_\M]}^\cong &  \CC \ar[r]^{[\pi]}\ar[u]_{[\vphi]}^\cong & \NN[1] \ar[r]\ar[u]_{[\id_\NN]}^\cong&\ldots 
}%}
\end{align}
Hence we get the desired exact triangle in $H^0(\QQ)$.
\end{proof}

\subsubsection{The main idea of the proof of Proposition~\ref{Prop:conedecomp}}\label{subsect:idea-of-proof-of-conedecomp}

Here we will outline the main idea of the proof of Proposition~\ref{Prop:conedecomp}.
For more details and the full argument we refer to \cite[Prop.~4.3.1]{BC-Lef}.\\ \par

Let $(V,\theta_V):(L_1,\ldots,L_s)\leadsto\emptyset$ be an object of $\fuk(\C\times M)$.
For $1\leq j\leq s$, let $\M_j$ be the submodule of the Yoneda-module $\Y(V,\theta_V)$ in $D\fuk(\C\times M)$ such that for any object $X$ of $\fuk(\C\times M)$, $\M_j(X)$ is generated only by the intersections of $X$ with the lowest $j$ horizontal ends of $V$, i.e. with the horizontal ends of $V$ on heights $h_1=1,h_2\ldots,h_j$.
The structural maps of $\M_j$ are induced by those of $\Y(V,\theta_V)$.
The fact that each $\M_j$ is an $A_\infty$-module over $\fuk(\C\times M)$ is a consequence of how the Fukaya category of cobordisms is set up and moreover we have a sequence of successive $A_\infty$-submodules (see~Step 3 in the proof of \cite[Prop.~4.3.1]{BC-Lef})
\begin{align*}
  0=\M_1\subset \ldots\subset \M_{s-1}\subset \M_s= \Y(V,\theta_V).
\end{align*}
\noindent In addition, for $2\leq j\leq s$, the quotient $\widetilde{L}_j:=\M_j /\M_{j-1}$ is isomorphic to
$$
\Ic^{{h_j}} \Rc_{h_j}(\M_{V,\theta_V})=\Ic^{{h_j}}\Y(L_j,\theta_{V,j})
$$
in $H^0(\fuk(\C\times M))$ (cf.~\cite{BC-Lef} and \cite[Section~4]{Haug-T^2}).
By Lemma \ref{lem:submodule-exact_triangle}, we obtain the following exact triangles in $H^0(mod(\fuk(\C\times M)))$
$$
\widetilde{L}_{j}[-1]\To\M_{j-1}\To\M_{j}\To\widetilde{L}_{j}
$$
for $2\leq j\leq s$.
Rotating these triangles yields
\begin{equation}\label{diag:cone-decomp-1}
\xymatrix@R16pt@C8pt{
        0=\M_1 \ar[rr]  & & \M_2 \ar[dl]\ar[rr]  & & \M_3 \ar[dl]\ar[r] & \cdots  \ar[r] & \M_{s-1} \ar[rr]  & & \M_s\cong V \ar[dl]\\
        & \widetilde{L}_2\ar@{-->}^{[1]}[ul] & & \widetilde{L}_{3}\ar@{-->}^{[1]}[ul] &&& & \widetilde{L}_s\ar@{-->}^{[1]}[ul]
      }
\end{equation}
Note that \eqref{diag:cone-decomp-1} is an equivalent way of writing the iterated cone decomposition \eqref{prop: cone decomp in MxC} from Proposition~\ref{Prop:conedecomp}.
Moreover, notice that for $2\leq j\leq s$, $\wt{L}_j$ is an object of $D\fuk(\C\times M)$, so it follows inductively that each of the modules $\M_j$ is an object of the derived Fukaya category $D\fuk(\C\times M)$ as well.

%%% Local Variables:
%%% mode: latex
%%% TeX-master: "../SCLC-arxiv"
%%% End:

%% file: Theta_map.tex
\section{The Lagrangian cobordism group and $K_0(D\fuk(M))$}\label{chap:theta-map}
%\chapter{The Lagrangian cobordism group and $K_0(D\fuk(M))$}\label{chap:theta-map}
\subsection[Stability conditions and the rel. between $\Om_{Lag}(M)$ and $K_0(D\fuk(M))$]{Stability conditions and the relation between $\Om_{Lag}(M)$ and $K_0(D\fuk(M))$}\label{sec:theta-map}

A central question related to Lagrangian cobordisms and the derived Fukaya category is the following:
\emph{To which extent can algebraic relations in $K_0(D\fuk(M))$ be understood geometrically via Lagrangian cobordisms?}
In this section we will derive formal conditions under which the Lagrangian cobordism group $\Om_{Lag}(M)$ is isomorphic to the Grothendieck group $K_0(D\fuk(M))$.
This is an attempt to demonstrate a link between stability conditions and the Lagrangian cobordism group.
In practice it might not be feasible to check these conditions in many examples.
However, this section provides a general context under which Haug's result on $T^2$ \cite[Thm.~1.1]{Haug-T^2} can be understood and possibly extended to other examples. \\ \par

As before, let $(M^{2n},J,g,\Om)$ be a Calabi-Yau manifold.
Recall from Appendix~\ref{subsubsec:cones} that the \emph{Grothendieck group} (or $K$-group) $K_0(D\fuk(M))$ of the derived Fukaya category of $M$ is defined as the free abelian group $\langle D\fuk(M)\rangle$ generated by the objects of $D\fuk(M)$ modulo the subgroup $R\subset \langle D\fuk(M)\rangle $ generated by expressions of the type $X-Y+Z$ for every exact triangle in $D\fuk(M)$ of the form $X\to Y\to Z\to X[1]$:
\begin{align*}
  K_0(D\fuk(M)):=\langle D\fuk(M)\rangle / R.
\end{align*}

We denote by $\LL$ be the set of all Lagrangian branes $(L,\theta_L,E_L)$ in $M$ equipped with local systems and by $\langle \LL\rangle$ the free abelian group generated by $\LL$.
Let $R_\LL\subset \langle\LL \rangle$ be the subgroup generated by expressions of the following two types

\begin{enumerate}
\item
  \begin{align*}
    \sum_{j=1}^s (L_j,\theta_j,E_j) \in \langle\LL \rangle
  \end{align*}
  whenever there exists a graded Lagrangian cobordism with local system
$$
(V,\theta_V,E_V):((L_1,\theta_1,E_1),\ldots , (L_s,\theta_s,E_s))\leadsto \emptyset
$$
such that the grading $\theta_V$ and the local system $E_V$ restricts to the respective ones on the horizontal ends.

\item
 $ (L,\theta_L,E') - (L,\theta_L,E) + (L,\theta_L,E'')\in  \langle\LL \rangle$ whenever there exists a short exact sequence of local systems $0\to E'\to E\to E''\to 0$ on $(L,\theta_L)$.

\end{enumerate}

 \begin{df}\label{df:Lag.Cob.group}
The \emph{Lagrangian cobordism group of $M$}, denoted by $\Om_{Lag}(M)$, is defined as the quotient of the free abelian group $\langle\LL \rangle$ modulo the subgroup $R_\LL$, that is:
\begin{align*}
  \Om_{Lag}(M):= \langle\LL \rangle / R_\LL.
\end{align*}
\end{df}

Note that if $(V,\theta_V,E_V): ((L_1,\theta_{V,1},E_{V,1}),\ldots , (L_s,\theta_{V,s},E_{V,s}))\leadsto \emptyset$ is a graded Lagrangian cobordism equipped with local system $E_V$ and such that $L_1$ is on height $h_1=1$, then, by \eqref{eq: cone-decomp in M} and Remark~\ref{rem:K-group-cone-sign-relation}, the following relation is satisfied in the Grothendieck group $K_0(D\fuk(M))$:
\begin{align*}
  [(L_1,\theta_{V,1},E_{V,1})]_{K_0}=-\sum_{j=2}^s [(L_j,\theta_{V,j},E_{V,j})]_{K_0} \quad \iff \quad \sum_{j=1}^s [(L_j,\theta_{V,j},E_{V,j})]_{K_0} = 0.
\end{align*}

\begin{rem}\label{rem:shift-cobordism}
  Given $(L,\theta_{L},E_{L})\in\LL$ and $j\in \Z_{>1}$ we can form the cobordism $V:= \iota^j(L,\theta_{L}-1,E_{L}):((L,\theta_{L},E_{L}), (L,\theta_{L}-1,E_{L})) \leadsto \emptyset $ from which we obtain the following relations:
  \begin{align*}
    &[(L,\theta_{L},E_{L})]_{K_0}+[(L,\theta_{L},E_{L})[1]]_{K_0}=0\; \text{in } K_0(D\fuk(M)), \;\text{and}\\
    &[(L,\theta_{L},E_{L})]_{Lag}+[(L,\theta_{L},E_{L})[1]]_{Lag}=0\; \text{in } \Om_{Lag}(M).
  \end{align*}
\end{rem}

Moreover, if $0\to E'\to E\to E''\to 0$ is a short exact sequence of local systems on $(L,\theta_L)$, then, by \cite[Prop.~10.1]{Haug-T^2}, there exists an exact triangle $(L,\theta_L,E') \to (L,\theta_L,E)  \to (L,\theta_L,E'')\to (L,\theta_L,E')[1]$ in $D\fuk(M)$.

This implies, as Biran and Cornea noticed in \cite{BC-CobII}, that there is a well defined surjective homomorphism from the Lagrangian cobordism group to the Grothendieck group of the derived Fukaya category.

\begin{cor}[{\cite[Cor.~1.2.1]{BC-CobII}}]\label{cor:theta-surjective}
The map $\langle\LL \rangle \to K_0(D\fuk(M))$ given by $(L,\theta_L,E_L)\mapsto [(L,\theta_L,E_L)]_{K_0}$ induces a well-defined surjective group homomorphism
\begin{align*}
  \Theta : \Om_{Lag}(M)\To K_0(D\fuk(M)).
\end{align*}
\end{cor}

Note that we identify $(L,\theta_L,E_L)$ with its image under the Yoneda-embedding when viewing it as an object of $D\fuk(M)$.
For the rest of this section we will sometimes suppress the grading and the local system from the notation and simply write $L$ for the Lagrangian brane $(L,\theta_L,E_L)$.

We now want to find certain assumptions under which we can construct a well-defined inverse to the group homomorphism $\Theta : \Om_{Lag}(M)\to K_0(D\fuk(M))$.
Recall that an object of $D\fuk(M)$ is called \emph{geometric} if it is isomorphic to a shift of an object in the image of the Yoneda-embedding.

\begin{assumption*}%[\textbf{S1}]
\begin{enumerate}[label=\textbf{(S\arabic*)}]
 \item \label{assumption:S1}Suppose that there exists a locally-finite stability condition \newline$(Z^M,\P^M)$ on $D\fuk(M)$ such that \emph{all stable objects} are \emph{geometric}.
\end{enumerate}
\end{assumption*}

With the above assumption in place, we can obtain a (finite) \emph{Jordan-H\"older} filtration (see Remark~\ref{df:stable-Jordan-Hoelder}) of every object of $D\fuk(M)$ which is unique up to permutation of stable factors of the same phase.
Therefore, denoting the collection of all \emph{stable} objects by $\SSS$, the inclusion $\langle \SSS\rangle \to \langle D\fuk(M)\rangle$ of free abelian groups descends to a surjective homomorphism $\langle \SSS\rangle \to K_0(D\fuk(M))$.
Put differently, the set $\SSS\subset Ob(D\fuk(M))$ generates $K_0(D\fuk(M))$.
Denoting by $R_\SSS\subset \langle \SSS\rangle$ the kernel of this surjective group homomorphism, we get an isomorphism of groups
\begin{align}\label{eq:stable-relation-K0}
  \varphi:\langle \SSS\rangle / R_\SSS \overset{\cong}{\To} K_0(D\fuk(M)).
\end{align}
The inverse $\varphi^{-1}$ is induced by the surjective homomorphism $\langle D\fuk(M)\rangle \to \langle \SSS\rangle$, that sends a generator $X\in Ob(D\fuk(M))$ to the sum $\sum_{i=1}^r S^X_i$ of its Jordan-H\"older factors $\{S^X_i\}_{i=1}^r$.
Composing $\Theta$ with the inverse of the isomorphism $\varphi$ yields a surjective group homomorphism $\widehat{\Theta}:\Om_{Lag}(M)\To \langle\SSS\rangle / R_\SSS$
\begin{align*}
 \xymatrix{
  \Om_{Lag}(M) \ar[rr]^{\widehat{\Theta}} \ar[rd]_{\Theta}& & \langle\SSS\rangle / R_\SSS\\
  &\langle D\fuk(M)\rangle / R \ar[ru]^\cong_{\varphi^{-1}}&
}\end{align*}

We will need the following two additional assumptions.

\begin{assumptions*}
\begin{enumerate}[label=\textbf{(S\arabic*)}]\setcounter{enumi}{1}
\item\label{assumption:S2}
%[\textbf{(S2)}]
  Every element in $R_\SSS\subset \langle\SSS \rangle$ induces a relation in $\Om_{Lag}(M)$.
  That is, if $\sum_{i=1}^r S_i\in R_\SSS$, then $\sum_{i=1}^r S_i \in R_\LL$ as well, where $S_i\in \LL$ are stable objects for $1\leq i\leq r$.
  In other words, and phrased somewhat imprecisely, any relation in $R_\SSS$ comes from a collection of Lagrangian cobordisms.
\item\label{assumption:S3} For every element $L\in\LL$ we have the relation
%[\textbf{(S3)}]
  \begin{align*}
    [L]_{Lag}=\sum_{j=1}^s[S^L_j]_{Lag}
  \end{align*}
  in $\Om_{Lag}(M)$, where $S^L_1,\ldots,S^L_s$ are the \emph{stable factors} occurring in the Jordan-H\"older filtration of $L\in Ob(D\fuk(M))$.
\end{enumerate}
\end{assumptions*}

\begin{rem}
  Assumptions \ref{assumption:S2} and \ref{assumption:S3} look rather technical.
  However, their advantage is that one only has to check something on the class of \emph{stable} objects.
  The stability condition of \ref{assumption:S1} will take care of the rest as we will see below.
\end{rem}

\begin{prop}
  Under the assumptions \ref{assumption:S1} and \ref{assumption:S2}, the map $\langle\SSS\rangle \to \langle\LL\rangle$ given by $S\mapsto S$
  induces a well-defined group homomorphism
$$
\widehat{\Psi}: \langle\SSS\rangle / R_\SSS \to  \Om_{Lag}(M),
$$
which is a right-inverse of $\widehat{\Theta}$, i.e. $\widehat{\Theta}\circ \widehat{\Psi} = \id_{\langle\SSS\rangle / R_\SSS}$.\par
If in addition we assume \ref{assumption:S3}, then $\widehat{\Psi}$ is also a left-inverse of $\widehat{\Theta}$ and hence
$$
\widehat{\Theta} : \Om_{Lag}(M) \To \langle\SSS\rangle / R_\SSS
$$
is an isomorphism of groups with inverse $\widehat{\Psi}$.
\end{prop}

\begin{proof}
  Well-definedness and the fact that $\widehat{\Psi}$ is a right-inverse of $\widehat{\Theta}$ follows directly from the definition of the maps together with assumptions \ref{assumption:S1} and \ref{assumption:S2}.
  Assumption \ref{assumption:S3} implies that $\widehat{\Psi}$ is surjective and hence also a left-inverse of $\widehat{\Theta}$.
\end{proof}

Since $\varphi:\langle \SSS\rangle / R_\SSS \overset{\cong}{\To} K_0(D\fuk(M))$ is a group isomorphism, we get the following corollary.

\begin{cor}\label{cor:theta-iso}
  Assuming \ref{assumption:S1} and \ref{assumption:S2}, the homomorphism
  \begin{align*}
    \Psi:=\widehat{\Psi}\circ\varphi^{-1} : K_0(D\fuk(M))\to \Om_{Lag}(M)
  \end{align*}
  is a right-inverse of $\Theta$.\par
  If we assume \ref{assumption:S3} as well, then $\Psi$ is also a left-inverse of $\Theta$ and therefore 
  \begin{align*}
    \Theta : \Om_{Lag}(M) \To K_0(D\fuk(M))
  \end{align*}
  is a group isomorphism.
\end{cor}

\subsection{The case of the torus $T^2$}\label{sec:example-T^2}
%\begin{ex}
  In this section we will take a closer look at the case of the $2$-torus $T^2$.
  In \cite{Haug-T^2} Haug proves that the Lagrangian cobordism group $\Om_{Lag}(T^2)$ of the torus $T^2$ is isomorphic to $K_0(D\fuk(T^2))$.
  We bring Haug's result into the context of stability conditions and give some alternative arguments of parts of the proof (see \cite[Sec.~7~and~8]{Haug-T^2}).
 % We briefly summarize Haug's strategy (see \cite[Sec.~8]{Haug-T^2}), 
\begin{notation}
  We start by recalling some notation from \cite{Haug-T^2}.
  The objects of $\fuk^\sharp(T^2)$ are \emph{Lagrangian branes}, i.e. tuples $(L,\theta,P,E)$ consisting of a non-con\-trac\-tible Lagrangian equipped with a grading $\theta$, Pin-structure $P$ and local system $E$ of $\Lambda$-vector spaces (we will usually omit the Pin-structure from the notation).
  For more details, see \cite[Sec.~3]{Haug-T^2}.
  Let $(m,n)$ be a pair of coprime integers. % and $x'\in\R/\Z$.
  If $(m,n)\neq (\pm 1,0)$, we denote by
  \begin{equation*}
    L_{(m,n),x}\subset T^2
  \end{equation*}
  the straight oriented curve of slope $(m,n)\in\Z^2\cong  H_1(T^2;\Z)$, % passing through the point %$(x',0)$, $(x,0)\in T^2$, 
where $x\in \R/\Z$ is the smallest number such that $L_{(m,n),x}$ passes through %both $(x',0)$ and 
$(x,0)\in T^2$.
  If $(m,n)= (\pm 1,0)$ we denote by $L_{(\pm1,0),x}$ the straight oriented horizontal curve in $T^2$, passing through $(0,x)\in T^2$, and oriented such that it represents $(\pm 1, 0)\in H_1(T^2;\Z)$.
  To abbreviate the notation we write $L_{(m,n)}:=L_{(m,n),0}$.
  The Lagrangians $L_{(m,n),x}$ can be viewed as objects of $\fuk^\sharp(T^2)$ by equipping them with their standard brane structure (see \cite{Haug-T^2}).
  Note also that every non-contractible closed curve on $T^2$ is Hamiltonian-isotopic to one of the curves $L_{(m,n),x}$.
\end{notation}
%\par

One of the starting points of Haug's considerations is the following homological mirror symmetry statement, which was proven by Abouzaid-Smith \cite{Abouzaid-Smith-T^4} building on previous work by Polishchuk-Zaslow \cite{Polishchuk-Zaslow} and inspired by Kontsevich \cite{Kontsevich-94}.
In the next theorem $X$ stands for the \emph{Tate curve} which is a specific elliptic curve over $\La$ (see \cite{Abouzaid-Smith-T^4}) and $P_0\in X$ is a base point.
The statement can be found in the following form in \cite[Thm.~7.1]{Haug-T^2}.

\begin{thm}\label{thm:HMS-AS}
  There is an equivalence of triangulated categories
  \begin{equation}\label{eq:mirror-functor}
    \Phi:D^b\mathrm{Coh}(X)\overset{\simeq}{\To} D^\pi\fuk^{\sharp}(T^2)
  \end{equation}
  taking $\OO(nP_0)$ to $L_{(1,-n)}$ for every $n\in \Z$ and the skyscraper sheaf $\Lambda(P_0)$ to $L_{(0,-1),\frac12}$.
\end{thm}
Haug remarks \cite{Haug-T^2} that the Tate curve $X$ can be studied via its analytification
\begin{align}\label{eq:analytic-Tate}
X^{an}=\La^*/\langle \{T^{k}\, |\, k\in \Z\}\rangle \cong S^1\La \times \R/\Z
\end{align}
and in what follows we will not further distinguish between $X$ and $X^{an}$ (see \cite{Fresnel-vdPut, Silverman-Advanced} for more on rigid-analytic geometry and on the Tate curve).
Here $S^1\La\subset \La$ consists of the elements of the Novikov-field of norm $1$ with respect to the non-Archimedean norm
\begin{equation*}
  \left| \sum_{i=0}^\infty c_i T^{a_i} \right|:= e^{-a_0},
\end{equation*}
and $\La^*=\La\setminus\{0\}$.
\begin{rem}
  Note that the Novikov-field $\La$ over the complex numbers is \emph{algebraically closed}.
  A proof of this fact can be found in \cite[App.~A]{FOOO-toric-I}.
\end{rem}
 Next Haug proceeds to partially recover how the functor $\Phi$ acts on certain objects.
 We will give an alternative proof of Haug's description as follows.

 \begin{prop}\label{prop:phi-correspondence-stable}
   \begin{enumerate}
     \item The image under $\Phi$ of a skyscraper sheaf $\La(Q)$ supported at any point $Q\in X$ is isomorphic to a Lagrangian of slope $(0,-1)\in H_1(T^2;\Z)$, equipped with a rank $1$ local system.
     \item The image under $\Phi$ of a stable sheaf of rank $r$ and degree $d$ is isomorphic to a Lagrangian of slope $(r,-d)\in H_1(T^2;\Z)$, equipped with a rank $1$ local system.
   \end{enumerate}
 \end{prop}

 \begin{proof}
   First recall from Theorem~\ref{thm:equivalence-of-semistable-subcat} that the stable coherent sheaves are precisely the indecomposable coherent sheaves of rank $r$ and degree $d$ with $\gcd(r,d)=1$.
   And these are parametrized by $(r,d)\in \Z^2$ and a point $P\in X$ on the curve, so for each coprime pair $(r,d)\in\Z^2$ the curve $X\cong S^1\La\times \R/\Z$ parametrizes the isomorphism classes of stable sheaves of rank $r$ and degree $d$.
   The equivalence $\Phi$ allows us to endow $D^\pi\fuk^{\sharp}(T^2)$ with a stability condition (c.f~Example~\ref{ex:SC-on-DCoh-EllCurve}).
   Note that 
   $$
   \mathrm{End}_{D\fuk(T^2)}((L_{(m,n),x},\theta,E_M))\cong \La
   $$
 since on the chain level, the degree $0$ morphisms are generated by a single element which is in the kernel of the differential (i.e. induces a morphism on the cohomology level).
 Here $E_M$ denotes the rank $1$ local system with monodromy $M\in S^1\La$.
 By Theorems~\ref{thm:semistable-indecomposable-sheaf} and \ref{thm:HMS-AS} this implies that for every $(m,n)\in \Z^2$ with $\gcd(m,n)=1$, every $x\in \R/\Z$ and every $M\in S^1\La$, the object
 $$
 (L_{(m,n),x},\theta,E_M)\in D\fuk^{\sharp}(T^2)\subset D^\pi\fuk^{\sharp}(T^2)
 $$
 is \emph{stable}.
   Also, note that
   $$
   \Hom_{D^\pi\fuk(T^2)}((L_{(m,n),x},\theta,E_M),(L_{(m,n),x},\theta,E_{M'})) =0
   $$
   if the monodromies $M,M'\in S^1\La$ of the rank $1$ local systems $E_M$ and $E_{M'}$ differ, since then the generator of the degree zero morphisms on the chain level does not survive to the cohomology level because the monodromies do not cancel each other; for the precise definition of the morphisms with local systems see \cite[Sect.~3.2]{Haug-T^2}.
   Clearly we also have that 
   $$
   \Hom_{D^\pi\fuk(T^2)}((L_{(m,n),x},\theta,E_M),(L_{(m,n),x'},\theta,E_{M'})) =0
   $$
   for $x\neq x'\in \R/\Z$, since there are no intersection points.
   Therefore we have shown that
\begin{equation*} \text{\emph{all stable objects of $D^\pi\fuk^{\sharp}(T^2)$ are shifts of objects of the form} } (L_{(m,n),x},\theta,E_M)\end{equation*}
with coprime $(m,n)\in\Z^2$, $x\in \R/\Z$ and $E_M$ of rank $1$ with monodromy $M\in S^1\La$ (recall that the grading $\theta$ accounts for shifts, i.e. $(L_{(m,n),x},\theta,E_M)[1] = (L_{(-m,-n),x},\theta-1,E_M)$; see Def.~\ref{df:shift}).\par

Next, let $(r,-d)\in \Z^2$ be a fixed pair of coprime integers with $r>0$ (by shifting $[\cdot]$ it is sufficient to consider the case $r>0$).
We can then iteratively form the Lagrangian surgery (see \cite[Sec.~5]{Haug-T^2}) of $L_{(1,-d)}$ with $L_{(1,0)}$ and obtain, by Corollary~\ref{cor:theta-surjective}, the following equation in $K_0(D^\pi\fuk^{\sharp}(T^2))$:
\begin{equation*}
  [L_{(r,-d),x}]_{K_0}=[L_{(1,-d)}]_{K_0} + (r-1)[L_{(1,0)}]_{K_0},
\end{equation*}
where $L_{(r,-d),x}$ is equipped with the standard brane structure.
Therefore, under the equivalence $\Phi$, the object $L_{(r,-d),x}$ corresponds, up to isomorphism, to an indecomposable object $F\in D^b\mathrm{Coh}(X)$ satisfying
\begin{equation*}
  [F]_{K_0}=[\OO(dP_0)]_{K_0} + (r-1)[\OO_X]_{K_0}.
\end{equation*}
By forming the iterated cone over $\OO(dP_0)$ and $(r-1)$-copies of $\OO_X$ (corresponding to the sequence of Lagrangian surgeries under $\Phi$) we obtain a coherent sheaf $F'$ of rank $r$ and degree $d$ with $\gcd(r,d)=1$ and $[F']_{K_0}=[F]_{K_0}$.
Hence $F'$ is \emph{stable} and since stable sheaves generate the $K_0$-group we must have that $F'\cong F[2k]$ for some integer $k$.
However, since we already know that $\Phi(\OO(dP_0))\cong L_{(1,-d)}$ and $\Phi(\OO_X)\cong L_{(1,0)}$ and since the equivalence $\Phi$ is an equivalence of triangulated categories, it follows that $k=0$ and hence $F\cong F'$.
That is, we have shown that up to isomorphism $L_{(r,-d),x}$ corresponds to a stable sheaf $F$ of rank $r$ and degree $d$ under the equivalence $\Phi$.

%NEXT: analyse the Hom's to see that all the (L_{(r,-d),x'},E_M) correspond to stable sheaves of rank r and degree d.
Next, note that if $F\in \mathrm{Coh}(X)$ is a stable coherent sheaf not isomorphic to the skyscraper sheaf $\La(Q)$, for some $Q\in X$, then we have
\begin{equation}\label{eq:dimHom-skyscraper}
  \dim\Hom_{D^b\mathrm{Coh}(X)}(F,\La(Q)) = \rk(F)
\end{equation}
and in particular $\Hom_{D^b\mathrm{Coh}(X)}(\La(P_0),\La(Q)) = 0 $ for $Q\in X\setminus \{P_0\}$.
Recall from Theorem~\ref{thm:Aut-to-SL-surj} that there is a surjective homomorphism 
$$
\zeta:\mathrm{Aut}(D^b\mathrm{Coh}(X))\to SL(2,\Z).
$$
The kernel of this homomorphism is generated by $\mathrm{Pic}^0(X)$, $\mathrm{Aut}(X)$ and even shifts, that is by taking the tensor product with degree $0$ line bundles, pulling back along automorphisms of the curve $X$ and taking even shifts (see \cite{Burban-Kreussler-06}).
Every element of $\ker (\zeta)$ preserves the rank and degree of coherent sheaves and the group of automorphisms $\mathrm{Aut}(X)$ of $X$ acts transitively on $X$ (see \cite[Chap.~IV, Cor.~4.3]{Hartshorne-AG}).
Therefore, given any point $Q\in X$ there exists $\psi_Q\in \ker (\zeta)$ such that $\psi_Q(\La(P_0)) = \La(Q)$.
Furthermore, any auto-equivalence $\psi\in \mathrm{Aut}(D^b\mathrm{Coh}(X))$ satisfies 
$$
\Hom_{D^b\mathrm{Coh}(X)}(A,B)\cong \Hom_{D^b\mathrm{Coh}(X)}(\psi(A),\psi(B))
$$
for any pair of objects $A,B$ of $D^b\mathrm{Coh}(X)$.\par
Given any two connected special Lagrangian branes 
\begin{align*}
(L_{(m,n),x},\theta,E_M)\quad \text{and}\quad (L_{(m',n'),x'},\theta',E_{M'})   
\end{align*}
of different slopes $(m,n)\neq(m',n')$ and with rank $1$ local systems, all their intersection points are of the same index.
Therefore, if $(m,n)\neq(m',n')$ and if the degree of their intersection points is $0$, we have
\begin{equation*}
  \dim\Hom_{D\fuk(T^2)}((L_{(m,n),x},\theta,E_M),(L_{(m',n'),x'},\theta',E_{M'})) = |mn' - nm'|
\end{equation*}
i.e. (the absolute value of) their homological intersection number and 
\begin{equation*}
\Hom_{D\fuk(T^2)}((L_{(m,n),x},\theta,E_M),(L_{(m',n'),x'},\theta',E_{M'}))=0
\end{equation*}
otherwise (this depends on the difference in their gradings $\theta$, $\theta'$).
Hence, by \eqref{eq:dimHom-skyscraper} and Theorem~\ref{thm:HMS-AS} we conclude that $\psi_Q(\La(P_0)) = \La(Q)$ corresponds, up to isomorphism, to a Lagrangian of the form $(L_{(0,-1),x},\theta,E_M)$, for some $x\in\R/\Z$, equipped with its standard grading and a local system of rank $1$ and monodromy $M\in S^1\La$.
This shows the first part of the proposition.\par
As we have seen above, $L_{(r,-d),x}$ corresponds, up to isomorphism, to a stable sheaf $F$ of rank $r$ and degree $d$ under the equivalence $\Phi$.
Hence, the second part follows from the first part together with Theorem~\ref{thm:equivalence-of-semistable-subcat}.
   
 \end{proof}

 \begin{rem}\label{rem:S1-satisfied}
   Recall that $D^b\mathrm{Coh}(X)$ admits a locally-finite stability condition (see Example~\ref{ex:SC-on-DCoh-EllCurve}) and hence, we obtain a locally-finite stability condition $(Z^{T^2},\P^{T^2})$ on $D^\pi\fuk^\sharp(T^2)$ via the equivalence $\Phi$.
   As we can see from Proposition~\ref{prop:phi-correspondence-stable}, all the stable objects are \emph{geometric} and therefore Assumption~\ref{assumption:S1} is satisfied.
 \end{rem}
 \begin{rem}
   By Proposition~\ref{prop:phi-correspondence-stable} all the stable objects of $D^\pi\fuk^\sharp(T^2)$ are in fact objects of $D\fuk^\sharp(T^2)$.
   Therefore $D\fuk^\sharp(T^2)$ admits a stability condition itself and hence is \emph{split closed} (cf.~Section~\ref{sec:consequence-example}) and equivalent to $D^\pi\fuk^\sharp(T^2)$.
   This allows us to simply write $D\fuk^\sharp(T^2)$ for $D^\pi\fuk^\sharp(T^2)$ from now on.
 \end{rem}

As we have seen in Theorem~\ref{thm:semistable-indecomposable-sheaf}, all the Jordan-H\"older factors of an indecomposable semistable object are isomorphic.
Translating this to $D\fuk(T^2)$, this means that all the indecomposable semistable objects are non-trivial extensions of the local systems on stable objects $(L_{(m,n),x},\theta,E_M)$.
This leads to the following corollary (cf.~\cite[Prop.~7.3]{Haug-T^2}).

\begin{cor}\label{cor:phi-correspondence-semistable}
     \begin{enumerate}
     \item The image under $\Phi$ of an indecomposable coherent torsion sheaf on $X$ of degree $d$ is isomorphic to a Lagrangian of slope $(0,-1)\in H_1(T^2;\Z)$, equipped with an indecomposable local system of rank $d$.
     \item The image under $\Phi$ of an indecomposable, coherent sheaf of rank $r>0$ and degree $d$ is isomorphic to a Lagrangian of slope $\frac{1}{h}(r,-d)\in H_1(T^2;\Z)$, equipped with an indecomposable local system of rank $h$, where $h:=\gcd(r,d)$.
   \end{enumerate}
\end{cor}

\begin{rem}\label{rem:S3-satisfied}
  Short exact sequences of local systems on the same underlying Lagrangian induce relations in the Lagrangian cobordism group (cf.~Def.~\ref{df:Lag.Cob.group}).
  Therefore, for every Lagrangian brane $(L_{(m,n),x},\theta,E)$  with local system of rank $h$ we have the relation
  \begin{equation*}
    [(L_{(m,n),x},\theta,E)]_{Lag} = \sum_{j=1}^h [(L_{(m,n),x},\theta,E_{M_j})]_{Lag}
  \end{equation*}
  in $\Om_{Lag}(M)$, where $(L_{(m,n),x},\theta,E_{M_j})$ are the stable factors equipped with local systems of rank $1$ with possibly different monodromies $M_j\in S^1\La$.
  This means that $(Z^{T^2},\P^{T^2})$ also satisfies Assumption~\ref{assumption:S3}.
\end{rem}

\begin{rem}
  We will use a slightly different definition of the \emph{standard grading} than Haug \cite[Sec.~3.1.1]{Haug-T^2}, the difference is just a factor of $\frac12$.
  The \emph{standard grading} of a Lagrangian in $T^2$ of slope $(m,n)\in H_1(T^2;\Z)$ is defined as the unique number $\theta\in [-\frac12,\frac12)$ satisfying
  \begin{equation*}
    e^{2\pi i\theta} = \frac{m + i n}{\sqrt{m^2 + n^2}}.
  \end{equation*}
  Recall from Example~\ref{ex:SC-on-DCoh-EllCurve} that
  $$
  Z:K_0(D^b\mathrm{Coh}(X))\To \C
  $$
  was given by $Z(F) = -\deg(F) + i\cdot \rk (F)$ for a coherent sheaf $F$ on $X$.
  Hence, by Corollary~\ref{cor:phi-correspondence-semistable} we can see that
   $$
  Z^{T^2}:K_0(D\fuk^\sharp(T^2))\To \C
  $$
  is given by
  $$
  Z^{T^2}(L_{(m,n),x},\theta,E) = \rk(E)\cdot (n + i\cdot m )
  $$
  for Lagrangian branes $(L_{(m,n),x},\theta,E)$ equipped with their standard grading.
  Therefore, we get the following relation between the \emph{phase} and the standard grading
  $$
  \phi((L_{(m,n),x},\theta,E))= \frac12 - \theta \in (0,1].
  $$
  This also fits with the relation
$$
(L_{(m,n),x},\theta,E)[1] = (L_{(-m,-n),x},\theta-1,E) \in \P^{T^2}((L_{(m,n),x},\theta,E))[1].
$$
  That is, if the \emph{grading} decreases by $1$, then the \emph{phase} increases by $1$ as expected (cf.~Def.~\ref{df:shift}), and vice versa.
\end{rem}

Haug proceeds to show %, using a classification result of Atiyah (cf.~\cite[Thm.~3]{Atiyah-VBEC} and \cite[Thm.~11.1]{Haug-T^2}), 
that all the $K_0$-relations among indecomposable coherent sheaves, are induced by relations in the cobordism group $\Om_{Lag}(T^2)$ under the equivalence $\Phi$ (for more details, see \cite[Sect.~8.1.1]{Haug-T^2}).
In order to show this, Haug uses the following classification result of Atiyah, see \cite[Thm.~3]{Atiyah-VBEC} and \cite[Thm.~11.1]{Haug-T^2}) for this particular formulation:
\begin{thm}\label{thm:Atiyah-1}
  There exists an integer $N(r,d)$ such that for every $n\geq N(r,d)$, every indecomposable sheaf $E$ of rank $r>0$ and degree $d$ fits, up to isomorphism, into the short exact sequence
  \begin{align*}
    0 \To \left( \OO_X^{\oplus r-1}\right) (-n) \To E \To (\det E) ((r-1)n) \To 0.
  \end{align*}
\end{thm}
\noindent In the above theorem $\bullet (n)$ denotes tensoring by the $n^{th}$ power of a hyperplane bundle.
A hyperplane bundle is a line bundle corresponding to a hyperplane section (see \cite{Atiyah-VBEC}).\par
More precisely, Haug proves the following (see \cite[Prop.~8.1]{Haug-T^2})
\begin{prop}\label{prop:generators-relation-stable}
  The set $R^X_\SSS$ of all $K_0$-relations among stable coherent sheaves on $X$ is generated by short exact sequences of the form
  \begin{enumerate}[ref=(\roman*)]
    \item \label{prop:generators-relation-stable-type1} $0\to F\to E\to 0$ with $F,E\in \mathrm{Coh}(X)$ both stable,
    \item \label{prop:generators-relation-stable-type2} $0\to \OO(D-Q) \to \OO(D) \to \La(Q)\to 0$ with $D$ a divisor and $Q\in X$,
    \item \label{prop:generators-relation-stable-type3} $0\to\left( \OO_X^{\oplus r-1}\right) (-n) \to E \to (\det E) ((r-1)n) \to 0$, with a stable sheaf $E\in \mathrm{Coh}(X)$ of rank $r>0$ and $n\in \Z$ as in Theorem~\ref{thm:Atiyah-1}.
  \end{enumerate}
\end{prop}
For a proof of this proposition see \cite[Sect.~8.1.1]{Haug-T^2}, an important ingredient is the fact that 
\begin{align}\label{eq:rk-det-iso-Pic}
  (\det,\rk) : K_0(X) \To \mathrm{Pic}(X)\oplus \Z
\end{align}
is an isomorphism (see~\cite{Hartshorne-AG} or \cite[p.34]{LePotier-VB}).
In fact Haug proves a slightly stronger result involving all relations among indecomposable coherent sheaves.
However, by \eqref{eq:stable-relation-K0} it suffices to consider the relations among stable coherent sheaves.
Haug then shows that the relations of the three types in Proposition~\ref{prop:generators-relation-stable} are induced by relations in the Lagrangian cobordism group.
The relations of type \ref{prop:generators-relation-stable-type1} identify isomorphic stable sheaves which correspond to Hamiltonian isotopic Lagrangian branes which, by Lagrangian suspension, are cobordant.
If $D$ is a divisor of degree $d$, then $\OO(D-Q)$, resp. $\La(Q)$ correspond under \eqref{eq:mirror-functor} to Lagrangian branes of slopes $(1,-(d-1))$ resp. $(0,-1)$.
One can perform Lagrangian surgery of these two Lagrangians and obtain a cobordism between them and a Lagrangian brane of slope $(1,-d)$.
As $\dim \mathrm{Ext}^1(\La(Q),\OO(D-Q)) = 1$, $\OO(D)$ must correspond via \eqref{eq:mirror-functor} to the Lagrangian brane of slope $(1,-d)$ obtained in this way.
This shows that the relations of type \ref{prop:generators-relation-stable-type2} are induced by Lagrangian cobordisms (see also \cite[Sec.~7.5]{Haug-T^2}).
For the relations of type \ref{prop:generators-relation-stable-type3}, one can construct a relation in $\Om_{Lag}(T^2)$ by iteratively surgering the Lagrangian brane of slope $(1,-d-3n(r-1))$, corresponding to the line bundle $(\det E) ((r-1)n)$, with $(r-1)$-copies of the Lagrangian brane of slope $(1,3n)$, corresponding to $\OO_X(-n)$.
Using \eqref{eq:rk-det-iso-Pic} Haug \cite[Sect.~7.5]{Haug-T^2} then shows that the Lagrangian obtained after these surgeries corresponds to $E$ under \eqref{eq:mirror-functor}.
Therefore, relations of type \ref{prop:generators-relation-stable-type3} are induced by relations in $\Om_{Lag}(T^2)$ as well.
Moreover, by Remark~\ref{rem:shift-cobordism}, relations in the $K_0$-group coming from shifts are also induced by relations in $\Om_{Lag}(T^2)$.

This implies that all the $K_0$-relations between stable objects of $D^b\mathrm{Coh}(X)$ are induced by relations in $\Om_{Lag}(T^2)$ and therefore Assumption~\ref{assumption:S2} is satisfied as well.

In conclusion, combining this observation with Remarks~\ref{rem:S1-satisfied}, \ref{rem:S3-satisfied} and Corollary~\ref{cor:theta-iso} we obtain the following result (cf.~\cite[Thm.~1.3]{Haug-T^2}).
\begin{thm}\label{thm:Haug-isomorphism}
  $D\fuk^\sharp(T^2)$ admits a locally-finite stability condition satisfying Assumptions~\ref{assumption:S1}, \ref{assumption:S2} and \ref{assumption:S3}, and 
  \begin{align*}
    \Theta : \Om_{Lag}(T^2) \To K_0(D\fuk^\sharp(T^2))
  \end{align*}
  is a group isomorphism.
\end{thm}

In some sense Corollary~\ref{cor:theta-iso} demystifies the connection between the Lagrangian cobordism group and the $K_0$-group.
Theorem~\ref{thm:Haug-isomorphism} yields a new perspective on Haug's result and hints at a deeper reason behind why Haug's arguments worked.\par\bigskip

In order to describe the Lagrangian cobordism group of $T^2$ explicitly we include the following corollary. % and sketch its proof below.

\begin{cor}
 $\Om_{Lag}(T^2)\cong (S^1\Lambda \times \R / \Z) \oplus \Z^2.$
\end{cor}
\begin{proof}
  By Theorems~\ref{thm:Haug-isomorphism} and~\ref{thm:HMS-AS} we have isomorphisms
$$
\Om_{Lag}(T^2) \cong K_0(D\fuk^\sharp(T^2))\cong K_0(D^b\mathrm{Coh}(X)).
$$
We can view $\mathrm{Coh}(X)$ as the subcategory of $D^b\mathrm{Coh}(X)$ given by complexes concentrated in degree $0$ and we denote the inclusion by $\iota:\mathrm{Coh}(X)\hookrightarrow D^b\mathrm{Coh}(X)$.
 Short exact sequences $0\to A \to B \to C\to 0$ in $\mathrm{Coh}(X)$ are in one to one correspondence with exact triangles $A\to B\to C\to A[1]$ in $D^b\mathrm{Coh}(X)$ with $A,B,C\in \mathrm{Coh}(X)\subset D^b\mathrm{Coh}(X)$ (this is true in general if $\mathrm{Coh}(X)$ is replaced by any abelian category; see~\cite[Ex.~2.27]{Huybrechts-Fourier-Mukai}).
  Denote by $K_0(X)$ the Grothendieck group of $\mathrm{Coh}(X)$, which is the free abelian group generated by objects of $\mathrm{Coh}(X)$ modulo the relations $[B]=[A]+[C]$ for every short exact sequence $0\to A\to B \to C \to 0$ in $\mathrm{Coh}(X)$.
  Note that every exact triangle $A^\bullet\to B^\bullet \to C^\bullet \to A^\bullet[1]$ in $D^b\mathrm{Coh}(X)$ induces a long exact sequence in cohomology (see~\cite[IV.1.6]{Gelfand-Manin})
$$
\ldots\to H^i(A^\bullet)\to H^i(B^\bullet)\to H^i(C^\bullet)\to H^{i+1}(A^\bullet)\to \ldots.
$$
 This, together with \eqref{eq:direct-sum-of-cohomology-objects}, yields a well defined group homomorphism
$$
h: K_0(D^b\mathrm{Coh}(X))\to K_0(X);\quad [E^\bullet]\mapsto \sum_{i\in\Z} (-1)^i[H^i(E^\bullet)].
$$
The inclusion by $\iota:\mathrm{Coh}(X)\hookrightarrow D^b\mathrm{Coh}(X)$ induces a homomorphism
$$
k : K_0(X) \to K_0(D^b\mathrm{Coh}(X)); \quad [E]\mapsto [\iota(E)].
$$
We clearly have $h\circ k = \id_{K_0(X)}$.
By \eqref{eq:truncation-triangles} it also follows that $k\circ h = \id_{K_0(D^b\mathrm{Coh}(X))}$.
Therefore we have shown that we have an isomorphism
$$
K_0(D^b\mathrm{Coh}(X))\cong K_0(X).
$$
Recall from \eqref{eq:rk-det-iso-Pic} that there is an isomorphism $K_0(X)\cong\mathrm{Pic}(X)\oplus \Z$. % (see~\cite{Hartshorne-AG}).
In addition we have a short exact sequence $0\to \mathrm{Pic}^0(X)\to \mathrm{Pic}(X) \overset{\mathrm{deg}}{\to} \Z\to 0$ (see~\cite[Prop.~2.6]{Silverman-Advanced})
and an isomorphism $\mathrm{Pic}^0(X)\cong X$ (see~\cite[III.~Prop.3.4]{Silverman-AEC}).
Recalling from \eqref{eq:analytic-Tate} that $X\cong S^1\Lambda \times \R/\Z$, and putting everything together, we conclude that
\begin{align*}
  \Om_{Lag}(T^2)\cong K_0(X) \cong \mathrm{Pic}(X)\oplus \Z \cong (S^1\Lambda\times \R/\Z)\oplus \Z^2
\end{align*}
holds.
\end{proof}

\begin{rem}
  Note that $K_0(D^b\mathrm{Coh}(X))\cong K_0(X)$ also follows from the general fact that $K_0(\D)\cong K_0(\A)$, whenever $\A$ is the heart of a bounded $t$-structure on a triangulated category $\D$ (see \cite[Sect.~2.2]{huybrechts-introSC}).
\end{rem}

%% file: Proof_of_main_result.tex
\section{Proof of Theorem~\ref{thm:(Z,P)-is-SC}}\label{Sec:proof_of_main_result}
%\chapter{Proof of Theorem~\ref{thm:(Z,P)-is-SC}}\label{Sec:proof_of_main_result}

With our preparations in place we may now proceed with the proof of Theorem~\ref{thm:(Z,P)-is-SC}.
We will show that the central charge $Z:K_0(D\fuk(\C\times M)\to \C$ as in \eqref{eq:def_central_charge} is a well defined homomorphism and that $(Z,\P)$ satisfies all the axioms of Definition~\ref{df:SC} and is locally-finite.

\begin{proof}[Proof of Theorem~\ref{thm:(Z,P)-is-SC}]
  
First, we will show that $Z$ is a well-defined group homomorphism.

Suppose that $V\to W \to U$ is an exact triangle in $D\fuk(\C\times M)$ and hence $[V]_{K_0}-[W]_{K_0}+[U]_{K_0}=0$ in $K_0(D\fuk(\C\times M))$.
We need to check that $$Z[V]_{K_0}-Z[W]_{K_0}+Z[U]_{K_0}=0$$ is satisfied.\par
Recall that for each $j\in\mathbb{Z}_{\geq 1}$ we have a triangulated restriction functor $$\Rc_j:D\fuk(\C\times M)\to D\fuk(M).$$
Therefore, for each $j\in\mathbb{Z}_{\geq 1}$, we get the relation $[\Rc_jV]_{K_0}-[\Rc_jW]_{K_0}+[\Rc_jU]_{K_0}=0$ in $K_0(D\fuk(M))$ which implies the following:
\begin{align*}
  Z[V]_{K_0}-Z[W]_{K_0}+Z[U]_{K_0} &=\sum_{j\geq 2}Z^M[\Rc_jV]_{K_0}-Z^M[\Rc_jW]_{K_0}+Z^M[\Rc_jU]_{K_0}\\
                 &=\sum_{j\geq 2}Z^M([\Rc_jV]_{K_0}-[\Rc_jW]_{K_0}+[\Rc_jU]_{K_0})\\
                 &=0.
\end{align*}
Therefore the central charge $Z$ is a well-defined homomorphism.

It remains to check that $(Z,\P)$ satisfies \ref{axiom-1}-\ref{axiom-4} in Definition \ref{df:SC}.

\begin{enumerate}
\item[\ref{axiom-1}]
  Let $\Ic^{h}X[r]\in\P(\phi)$ with $X\in \P^M(\phi-r+\kappa h)$.
  Then, by Remark \ref{rem:restriction-inclusion-relation}
  \begin{align*}
    Z[\Ic^{h}X[r]]_{K_0}&=Z^M[\Rc_{h}\Ic^{h}X[r]]_{K_0} \\
                                       &= Z^M[X[r]]_{K_0}\\
                                       &=\underbrace{m(X)}_{>0}\exp (i\pi(\phi+\kappa h))\\
                                       &=m(X)\exp (i\pi\phi),
  \end{align*}
since $\kappa$ is even and  $X[r]\in \P^M(\phi-r+\kappa h)[r]=\P^M(\phi+\kappa h)$.
\item[\ref{axiom-2}]
  \begin{align*}
    \P(\phi)[1]&=\left\langle\{ \Ic^{h}X[r+1] \, |\, r\in \Z, \, h\in\Z_{>1},\, X\in\P^M(\phi-r+\kappa h) \}\right\rangle \\
               &=\left\langle \{ \Ic^{h}X [r]\, |\, r\in \Z, \, h\in\Z_{>1},\, X\in\P^M(\phi+1-r+\kappa h) \}\right\rangle \\
               &= \P(\phi+1).
  \end{align*}
  Here, the brackets $\langle\ldots\rangle$ signify that we take the additive closure in $D\fuk(\C\times M)$.

\item[\ref{axiom-3}]

  In order to show axiom \ref{axiom-3}, we will prove the following lemma.
  \begin{lem}\label{lem:hom-inclusion}
    Given $h,h'\in\Z$ and any two objects $X,X'$ of $D\fuk(M)$ we have
    $$
    \Hom_{D\fuk(\C\times M)}(\Ic^{{h'}}X',\Ic^{h}X)\cong \begin{cases} \Hom_{D\fuk(M)}(X',X), &\text{if } h'\geq h\\
    0, &\text{if } h'<h,\end{cases}
    $$
    where the isomorphism is induced by the restriction $\Rc_1$.
  \end{lem}

Before proving Lemma~\ref{lem:hom-inclusion}, we will deduce \ref{axiom-3} from it.\par
Suppose that $\phi'>\phi$, $\Ic^{{h'}}X'[r']\in\P(\phi')$ and $\Ic^{h}X[r]\in\P(\phi)$.

\begin{enumerate}
  \item If $h'<h$, then  $\Hom_{D\fuk(\C\times M)}(\Ic^{{h'}}X'[r'],\Ic^{h}X[r])=0$ by Lemma~\ref{lem:hom-inclusion}.
  \item If $h'\geq h$, then
    $$
    \Hom_{D\fuk(\C\times M)}(\Ic^{{h'}}X'[r'],\Ic^{h}X[r])\cong  \Hom_{D\fuk(M)}(X'[r'],X[r])=0
    $$
    since $(Z^M,\P^M)$ satisfies \ref{axiom-3} and  $X'[r']\in\P^M(\phi'+\kappa h')$, $X[r]\in\P^M(\phi+\kappa h)$ where $\phi'+\kappa h'>\phi+\kappa h$.
\end{enumerate}
Thus $(Z,\P)$ satisfies \ref{axiom-3} as well.

  \begin{proof}[Proof of Lemma~\ref{lem:hom-inclusion}]
    \textbf{Step 1:} Suppose that $X$ and $X'$ are geometric objects, i.e. $X\cong L[r]$ and $X'\cong L'[r']$ where $L, L'\in \fuk(M)$ are graded Lagrangians in $M$.

  In the setup of the Floer complex $CF(V,W)$ of two cobordisms, the perturbation datum is chosen in such a way that the horizontal ends of the \emph{second} cobordism $V$ will be slightly perturbed downwards, i.e. in the $-i$ direction in $\C$. In particular, this means that if $h'<h$, we get that $HF(\Ic^{{h'}}L'[r'],\Ic^{{h}}L[r])=0$ since after a Hamiltonian isotopy that preserves horizontal ends, the projections of $\Ic^{{h'}}L'$ and $\Ic^{{h}}L$ to $\C$ will be disjoint. Therefore, in this case there are no non-trivial morphisms, i.e. $\Hom_{D\fuk(\C\times M)}(\Ic^{{h'}}X',\Ic^{{h}}X)=0$.\par
  Suppose now that $h'\geq h$.
  In this case, instead of having disjoint projections to $\C$, we can arrange that after suitable Hamiltonian perturbation, the projections of $\Ic^{{h'}}L'$ and $\Ic^{{h}}L$ to $\C$ intersect in exactly one point $z\in\C$ (see Figure~\ref{fig: Lemma_A3}).

	\begin{figure}[ht]
\def\svgwidth{0.4\columnwidth}
\centering

                         \includegraphics[width=0.4\columnwidth]{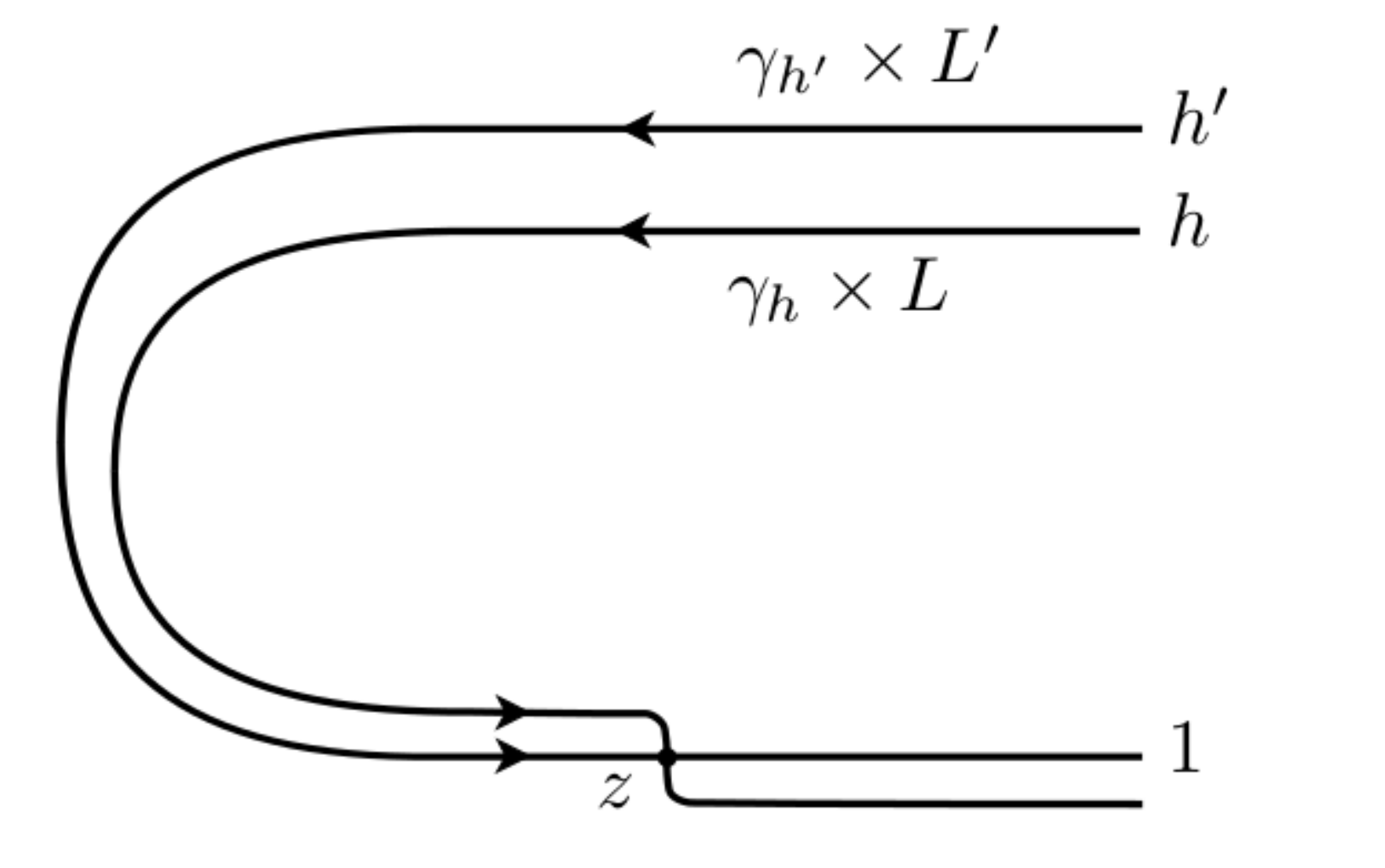}
			\caption{The projections of $\Ic^{{h}}L$ after a slight downward perturbation and $\Ic^{{h'}}L'$ to $\C$, in the case $h'\geq h$.}
			\label{fig: Lemma_A3}
\end{figure}

  Thus, all the intersection points of $\Ic^{{h'}}L'\pitchfork \Ic^{{h}}L$ are of the form $\xi=(z,p)\in\C \times M$ for some $p\in L'\pitchfork L$.
  Using the observation of Remark \ref{rem:curve-grading-increase}, and denoting by $\phi'_1,\ldots ,\phi'_n$ the K\"ahler angles between $L'$ and $L$, we calculate the degree of $\xi$:

  \begin{align*}
   &\deg_\xi(\Ic^{{h'}}L'[r'],\Ic^{{h}}L[r])\\  & = (n+1) + \left(\theta_{L}-r+\frac12\right) - (\theta_{L'}-r' + 1) -\frac{1}{\pi}\left(\frac{\pi}{2} + \sum_{i=1}^n\phi'_i\right)\\
                      & = n +\theta_{L} - \theta_{L'} + r' - r - \frac{1}{\pi}\sum_{i=1}^n\phi'_i\\
                      & = \underbrace{\deg_z(\gamma_{h'},\gamma_{h})}_{=0} + \deg_p((L',\theta_{L'})[r'],(L,\theta_{L})[r])\\
                      & = \deg_p((L',\theta_{L'})[r'],(L,\theta_{L})[r]).
  \end{align*}

Hence, we get $\Hom_{D\fuk(\C\times M)}(\Ic^{{h'}}X',\Ic^{{h}}X)\cong \Hom_{D\fuk(M)}(L'[r'],L[r]).$\vspace{2pt}

\textbf{Step 2:}
 Let $\mathcal{Q}$ be the collection of pairs of objects of $D\fuk(M)$ which satisfy the conclusion of the lemma.
 Note that by Step 1 every pair of geometric objects lies in $\mathcal{Q}$.\par
 Suppose now that $(X',X)\in\mathcal{Q}$ and $(X',Y)\in \mathcal{Q}$ and let $C$ be an object in the isomorphism class of the cone $(X\overset{\varphi}\to Y)$ over some morphism $\varphi\in\Hom_{D\fuk(M)}(X,Y)$.
 We want to show that the pair $(X',C)$ lies in $\mathcal{Q}$ as well.\par
 Recall that the restriction functor at height 1 is a triangulated functor $\Rc_1:D\fuk(\C\times M) \to D\fuk(M)$ which satisfies $\Rc_1\Ic^{h}\M = \M[-1]$ for $\M\in D\fuk(M)$.
 We have the following commutative (up to signs) diagram of exact triangles:
$$ \xymatrix{
   \ldots\ar[r] & \Ic^{h}X \ar[r]^{\Ic^{h}\varphi} \ar[d]^{\Rc_1} &\Ic^{h}Y \ar[r]\ar[d]^{\Rc_1} &\Ic^{h}C \ar[r]\ar[d]^{\Rc_1} & \Ic^{h}X[1]\ar[r]\ar[d]^{\Rc_1} &\ldots  \\
   \ldots\ar[r] & X[-1] \ar[r]^{-\varphi[-1]} & Y[-1]  \ar[r] & C[-1] \ar[r]  & X \ar[r]& \ldots 
 }$$
 
 Given any object $\wt{X'}\in D\fuk(\C\times M)$ we apply the cohomological functor $\Hom_{D\fuk(\C\times M)}(\wt{X'},\,\cdot\,)$ (see \cite[Chap.~2~Prop.~1.2.1]{Verdier96-PhD}) to the above diagram in order to obtain a commutative (up to signs) diagram of long exact sequences.
 \begin{equation*}
   \resizebox{.9\hsize}{!}{ \xymatrix@C=0.9em{
%   \ldots
\ar[r]& \Hom(\wt{X'},\Ic^{h}X) \ar[r]^{} \ar[d]^{\Rc_1} &\Hom(\wt{X'},\Ic^{h}Y) \ar[r]\ar[d]^{\Rc_1} & \Hom(\wt{X'},\Ic^{h}C) \ar[r]\ar[d]^{\Rc_1} & \Hom(\wt{X'},\Ic^{h}X[1])\ar[r]\ar[d]^{\Rc_1}& \\% \ldots  \\
%   \ldots\ar[r] 
\ar[r]& \Hom(\Rc_1\wt{X'},X[-1]) \ar[r]^{} & \Hom(\Rc_1\wt{X'},Y[-1])  \ar[r] & \Hom(\Rc_1\wt{X'},C[-1]) \ar[r]  & \Hom(\Rc_1\wt{X'},X) \ar[r] &% \ldots 
 }}
 \end{equation*}

 Specializing to $\wt{X'}:=\Ic^{{h'}}X'$ yields
    \begin{equation*}
      \resizebox{.9\hsize}{!}{\xymatrix@C=0.9em{
%   \ldots
\ar[r]& \Hom(\Ic^{{h'}}X',\Ic^{h}X) \ar[r]%^{\Ic^{h}\varphi}
 \ar[d]^{\Rc_1} &\Hom(\Ic^{{h'}}X',\Ic^{h}Y) \ar[r]\ar[d]^{\Rc_1} & \Hom(\Ic^{{h'}}X',\Ic^{h}C) \ar[r]\ar[d]^{\Rc_1} & \Hom(\Ic^{{h'}}X',\Ic^{h}X[1])\ar[r]\ar[d]^{\Rc_1}& \\% \ldots  \\
%   \ldots\ar[r] 
\ar[r]& \Hom(X'[-1],X[-1]) \ar[r]%^{\varphi[-1]}
 & \Hom(X'[-1],Y[-1])  \ar[r] & \Hom(X'[-1],C[-1]) \ar[r]  & \Hom(X'[-1],X) \ar[r] &% \ldots 
 } }
\end{equation*}

\begin{enumerate}
\item If $h'<h$, then the top long exact sequence takes the form
  
\begin{equation*}
  \ldots 0 \To 0 \To \Hom(\Ic^{{h'}}X',\Ic^{h}C)\to 0 \To\ldots
\end{equation*}

 and hence $\Hom(\Ic^{{h'}}X',\Ic^{h}C) = 0$.

 \item If $h'\geq h$, then 
\begin{equation*}
  \mkern-20mu\resizebox{.9\hsize}{!}{\xymatrix@C=0.9em{
%   \ldots
\ar[r]& \Hom(\Ic^{{h'}}X',\Ic^{h}X) \ar[r]%^{\Ic^{h}\varphi}
 \ar[d]^{\Rc_1}_{\cong} &\Hom(\Ic^{{h'}}X',\Ic^{h}Y) \ar[r]\ar[d]^{\Rc_1}_{\cong} & \Hom(\Ic^{{h'}}X',\Ic^{h}C) \ar[r]\ar[d]^{\Rc_1} & \Hom(\Ic^{{h'}}X',\Ic^{h}X[1])\ar[r]\ar[d]^{\Rc_1}_{\cong}& \\% \ldots  \\
%   \ldots\ar[r] 
\ar[r]& \underset{\cong\Hom(X',X)}{ \Hom(X'[-1],X[-1])} \ar[r]%^{\varphi[-1]}
 & \underset{\cong\Hom(X',Y)}{\Hom(X'[-1],Y[-1])}  \ar[r] & \underset{\cong\Hom(X',C)}{\Hom(X'[-1],C[-1])} \ar[r]  & \underset{\cong\Hom(X',X)[1]}{\Hom(X'[-1],X)} \ar[r] &% \ldots 
 } }
\end{equation*}

and by the $5$-Lemma we get $\Hom(\Ic^{{h'}}X',\Ic^{h}C)\cong\Hom(X',C) $.
\end{enumerate}
Therefore the pair $(X',C)$ lies in $\mathcal{Q}$ as well.\par

With a similar argument, using instead the cohomological functor 
$$
\Hom_{D\fuk(\C\times M)}(\,\cdot\, ,\wt{X}),
$$
one can show that if $(X',X)\in\mathcal{Q}$ and $(Y',X)\in\mathcal{Q}$, then $(C',X)\in\mathcal{Q}$ as well, where $C':=(X'\overset{\varphi'}{\to}Y')$.\par
To conclude the proof of the lemma we note that any object of $D\fuk(M)$ is isomorphic to an iterated cone over geometric objects, and hence the lemma holds for every pair of objects.
  \end{proof}

\end{enumerate}

\begin{enumerate}[resume]
\item[\ref{axiom-4}]
\textbf{Step 1:} We will first check \ref{axiom-4} of Definition \ref{df:SC} for Yoneda-modules $(V,\theta_V)\in D\fuk(\C \times M)$.\par
Let $(V,\theta_V):(L_1,\ldots,L_s)\leadsto\emptyset$ be an object of $\fuk(\C\times M)$.
Recall the iterated cone decomposition of $(V,\theta_V)$ from Proposition~\ref{Prop:conedecomp} in its particular form \eqref{diag:cone-decomp-1}
\begin{equation*}%\label{diag:cone-decomp-1}
\xymatrix@R16pt@C8pt{
        0=\M_1 \ar[rr]  & & \M_2 \ar[dl]\ar[rr]  & & \M_3 \ar[dl]\ar[r] & \cdots  \ar[r] & \M_{s-1} \ar[rr]  & & \M_s\cong V \ar[dl]\\
        & \widetilde{L}_2\ar@{-->}^{[1]}[ul] & & \widetilde{L}_{3}\ar@{-->}^{[1]}[ul] &&& & \widetilde{L}_s\ar@{-->}^{[1]}[ul]
      }
\end{equation*}

%%%%%%%%%%%%%%%%%%%%%%%%%%%% NEW VERSION %%%%%%%%%%%%%%%%%%%%%%%%%%%%%%%%%%%%%%%%%
For $2\leq j\leq s$ we write the HN-filtration of $\Rc_{h_j}\wt{L}_j \cong \Rc_{h_j}\Ic^{{h_j}}(L_j,\theta_{V,j}) \cong (L_j,\theta_{V,j})$
in $D\fuk(M)$ as
\begin{equation}\label{eq:HN_filtration-of-restriction}
  \Rc_{h_j}\wt{L}_j \cong ( X_{j,\ell(j)}[-1]\to \ldots \to X_{j,1}[-1]\to 0 )
\end{equation}
where $X_{j,k}\in \P^M(\phi^M_{j,k})$ is semistable of phase $\phi^M_{j,k}$ for each $1\leq k \leq \ell(j)$.
Note that if $\wt{L}_j$ is semistable to begin with, then $\ell(j)=1$ and $X_{j,1}\cong (L_j,\theta_{V,j})$.
By applying the inclusion functor $\Ic^{{h_j}}$ to the HN-filtration of $\Rc_{h_j}\wt{L}_j$ we obtain the following iterated cone decomposition of $\wt{L}_j$:
\begin{equation}\label{eq:semistable-cone-refinement-1}
  \wt{L}_j%\cong \Ic^{{h_j}}\Rc_{h_j}\Ic^{{h_j}}(L_j,\theta_{V,j})[\kappa h_j] 
           \cong ( \wt{X}_{j,\ell(j)}[-1]\to \ldots \to \wt{X}_{j,1}[-1]\to 0 ),
\end{equation}
where $\wt{X}_{j,k} := \Ic^{{h_j}} X_{j,k}$ for $1\leq k \leq \ell(j)$ and $\wt{X}_{j,k}$ is of \emph{height} $h_j$.
As $X_{j,k}\in \P^M(\phi^M_{j,k})$, it follows from \eqref{eq:inclusion-phase-relation} that $\wt{X}_{j,k}$ is semistable of phase
$$
\phi_{j,k}:=\phi(\wt{X}_{j,k}) = \phi^M_{j,k}-\kappa h_j.
$$
By Remark~\ref{rem:nested-cones} and \eqref{eq:semistable-cone-refinement-1} the cone decomposition \eqref{eq:cone-decomp-wt(L)} admits the refinement
\begin{equation}
  \label{eq:semistable-cone-refinement-2}
  V \cong (\wt{X}_{s,\ell(s)}[-1]\to\ldots\to \wt{X}_{3,1}[-1]\to \wt{X}_{2,\ell(2)}[-1]\to\ldots \to \wt{X}_{2,1}[-1]\to 0)
\end{equation}
In an effort to render the notation more transparent we will relabel the factors of \eqref{eq:semistable-cone-refinement-2} as $\wt{X}_q := \wt{X}_{j,k}$ (resp.~ $X_q:=X_{j,k}$) where $2\leq j\leq s$, $1\leq k\leq \ell(j)$ and 
$q:=k+ \sum_{i=2}^{j-1}\ell(i)$.
We also relabel the phases in the same fashion.
After this relabeling \eqref{eq:semistable-cone-refinement-2} takes the form
\begin{equation}\label{eq:semistable-cone-refinement-3}
  V \cong (\wt{X}_{s'}[-1]\to\ldots\to \wt{X}_1[-1]\to 0)
\end{equation}
where $s'=\sum_{i=2}^s\ell(i)$,
\begin{equation}\label{eq:equal-phases}
  X_j\in \P^M(\phi^M_j),\quad \wt{X}_j\in \P(\phi_j),\quad \phi_j=\phi^M_j-\kappa \wt{h}_j
\end{equation}
and $\wt{h}_j$ denotes the \emph{height} of $\wt{X}_j$.
Notice also that if $\sum_{i=2}^{j-1}\ell(i) < p,q \leq \sum_{i=2}^j\ell(i)$, then $\wt{X}_p$ and $\wt{X}_q$ are on the \emph{same} height 
$$
\wt{h}_p=\wt{h}_q=h_j
$$
for $2\leq j\leq s$.
Conversely, if $\wt{X}_p$ respectively $\wt{X}_q$ are semistable objects occurring in the iterated cone decomposition of $\wt{L}_j$ respectively $\wt{L}_i$, with $i\neq j$, then their heights are different, i.e.~$\wt{h}_p=h_j\neq h_i=\wt{h}_q$.\par

We need to check that the phases in the iterated cone decomposition are descending, i.e. that $\phi_2>\ldots >\phi_{s'}$ is satisfied.
By \eqref{eq:equal-phases} this amounts to checking that $\phi^M_2-\kappa \wt{h}_2>\ldots >\phi^M_{s'}-\kappa \wt{h}_{s'}$ holds.\par
 
First note that if $\wt{h}_j=\wt{h}_{j+1}$, that is if we are within the cone decomposition \eqref{eq:semistable-cone-refinement-1}, then we know by the HN-filtration \eqref{eq:HN_filtration-of-restriction} that the inequality $\phi^M_j> \phi^M_{j+1}$ holds and hence the desired inequality $\phi_j> \phi_{j+1}$ follows in these cases.\par
It remains to check the cases where $\wt{h}_j<\wt{h}_{j+1}$ holds.
In this case we are at the passage from $\wt{L}_{j+1}$ to $\wt{L}_j$ in the cone decomposition.

By \eqref{eq: cone-decomp in M}, the triangulated restriction functor $\Rc_1$ applied to \eqref{eq:semistable-cone-refinement-3} yields the following iterated cone in $D\fuk(M)$:
\begin{equation}\label{diag:cone-decomp-2-in-M}
  (L_1,\theta_{V,1})\cong (X_{s'}[-2]\to\ldots\to X_1[-2]\to 0).
\end{equation}

\textbf{Now, let $1\leq j\leq s'-1$ be an index such that $\wt{h}_j<\wt{h}_{j+1}$.}\par

Using Remark~\ref{rem:shifting-iterated-cones} we may reorder the brackets in \eqref{eq:semistable-cone-refinement-3} and \eqref{diag:cone-decomp-2-in-M} as follows
\begin{align}
  &(\wt{X}_{s'}[-1]\to\ldots\to(\wt{X}_{j+1}[-2]\overset{\wt{\alpha}_j}{\to}\wt{X}_j[-1])\to\ldots\to 0)\quad\text{and}\label{diag:cone-decomp-3}\\
  &(X_{s'}[-2]\to\ldots\to (X_{j+1}[-3]\overset{\alpha_j}{\to} X_j[-2])\to\ldots\to 0),\label{diag:cone-decomp-3-in-M}
\end{align}
where $\alpha_j:=\Rc_1\wt{\alpha}_j$.
There are two cases to examine.\par\bigskip
%\begin{itemize}
 \textbf{Case 1:} % \item[\textbf{Case 1:}] 
    If $\wt{\alpha}_j$ is not the zero-morphism.\par\bigskip
                          By Lemma~\ref{lem:hom-inclusion} we obtain
                          \begin{align*}
                            &\Hom_{D\fuk(\C\times M)}(\wt{X}_{j+1}[-2],\wt{X}_j[-1])\\
                            \cong& \Hom_{D\fuk(M)}(\Rc_1\wt{X}_{j+1}[-2],\Rc_1\wt{X}_j[-1])\\
                            \cong& \Hom_{D\fuk(M)}(X_{j+1}[-3], X_j[-2]).
                          \end{align*}
                          Under this identification, $\wt{\alpha}_j$ corresponds to $\alpha_j$.
                          Therefore, if $\wt{\alpha}_j\neq 0$, then $\alpha_j\neq 0$ as well.
                          Thus, by Definition \ref{df:SC} \ref{axiom-3} and the decomposition \eqref{diag:cone-decomp-3-in-M}, we have the phase inequality 
                          $$
                          \phi^M_{j+1}-1 \leq \phi^M_j
                          $$
                          which, since $\kappa > 2$ and $\wt{h}_j<\wt{h}_{j+1}$, implies that 
                          \begin{align}\label{eq:height-difference-phase-difference-bigger-than-1}
                          \mkern-10mu\phi_{j+1}=\phi^M_{j+1}-\kappa \wt{h}_{j+1} \leq \phi^M_j+1-\kappa \wt{h}_{j+1} < \phi^M_j -\kappa \wt{h}_j-1 =\phi_j-1 <\phi_j
                          \end{align}
                          is indeed satisfied.\par\bigskip
    
 \textbf{Case 2:} % \item[\textbf{Case 2:}]
    Suppose that the morphism $\wt{\alpha}_j$ is the zero-morphism and that the phase inequality is unfavorable, i.e. $\phi_j\leq\phi_{j+1}$.\par\bigskip
    \begin{enumerate}[(i)]
      \item \label{A4-Step-1-Case-2-i}
        If $\phi_j=\phi_{j+1}$:\par
        In this case we have that $\mathrm{cone}(\wt{\alpha}_j)\cong \wt{X}_{j+1}[-1]\oplus \wt{X}_j[-1]\in\P(\phi_j-1)$ is itself semistable of phase $\phi_j-1$ since by definition $\P(\phi_j-1)$ is \emph{additive}.
        So we may rewrite \eqref{eq:semistable-cone-refinement-3} as
        \begin{align*}
          V \cong (\wt{X}_{s'}[-1]\to\ldots\to\wt{X}_{j+1}[-1]\to\mathrm{cone}(\wt{\alpha}_j)\to \wt{X}_{j-1}[-1]&\to\ldots\\ \ldots\to\wt{X}_1[-1]&\to 0).
        \end{align*}

      \item \label{A4-Step-1-Case-2-ii}
                  If $\phi_j<\phi_{j+1}$:\par
                          In this case, by Remark \ref{rem:cone-direct-sum}, we have
                          $$
                          (\wt{X}_{j+1}[-2]\overset{0}{\to} \wt{X}_j[-1]) \cong (\wt{X}_j[-2]\overset{0}{\to} \wt{X}_{j+1}[-1])
                          $$
                          and so we may switch these two objects in the iterated cone to get the desired phase inequality
                      \begin{align*}
                        &(\wt{X}_s[-1]\to\ldots(\wt{X}_{j+1}[-2]\overset{\wt{\alpha}_j=0}{\To}\wt{X}_j[-1])\to\ldots\to 0)\\
                       \cong & (\wt{X}_{s'}[-1]\to\ldots(\wt{X}_j[-2]\overset{0}{\to}\wt{X}_{j+1}[-1])\to\ldots\to 0)\\
                        \cong &(\wt{X}_{s'}[-1]\to\ldots \wt{X}_j[-1]\to\wt{X}_{j+1}[-1]\to\ldots\to\wt{X}_2[-1]\to 0).
                      \end{align*}
                      Switching these objects may create problems with the phase inequalities of the adjacent objects, so we iteratively repeat this switch until we end up with
                      \begin{align}\label{A4-Step-1-Case-2-ii-switched}
                        V\cong (\wt{X}_{s'}[-1]&\to\ldots\to \wt{X}_{k+1}[-1]\to\wt{X}_j[-1]\to\wt{X}_k[-1]\to\ldots \\
                        \ldots&\to\wt{X}_r[-1]\to\wt{X}_{j+1}[-1]\to\wt{X}_{r-1}[-1]\to\ldots\to\wt{X}_1[-1]\to 0),\nonumber
                      \end{align}
                      where $j+1\leq k\leq s'$ and $1\leq r\leq j$ are such that
                      \begin{align*}
                       & \phi_k > \phi_j, \quad \phi_{k+1}\leq \phi_j\\
                      \text{and}\quad & \phi_r<\phi_{j+1}, \quad \phi_{r-1}\geq \phi_{j+1}.
                      \end{align*}
\end{enumerate}
%\end{itemize}
We now iteratively repeat the argument of Case 2~\ref{A4-Step-1-Case-2-i} and \ref{A4-Step-1-Case-2-ii} for \eqref{A4-Step-1-Case-2-ii-switched} until, after finitely many steps, we end up with the desired HN-filtration of $V$.
This shows that Yoneda-modules $(V,\theta_V)\in D\fuk(\C \times M)$ satisfy \ref{axiom-4} of Definition \ref{df:SC}.\\ \par

\textbf{Step 2:} It remains to check \ref{axiom-4} for general modules.\par

 We will first check that any cone $W:=((U,\theta_U)\to (V,\theta_V))$ over two Yoneda-modules $(U,\theta_U),(V,\theta_V)\in D\fuk(\C\times M )$ satisfies \ref{axiom-4}.

By Step 1, we know that both $U$ and $V$ admit a HN-filtration
\begin{equation}\label{diag:cone-decomp-step2-U}
\xymatrix@R16pt@C8pt{
        0=\M^U_1 \ar[rr]  & & \M^U_2 \ar[dl]\ar[rr]  & & \M^U_3 \ar[dl]\ar[r] & \cdots  \ar[r] & \M^U_{r-1} \ar[rr]  & & \M^U_r\cong U, \ar[dl]\\
        & \wt{X}_1\ar@{-->}^{[1]}[ul] & & \wt{X}_{3}\ar@{-->}^{[1]}[ul] &&& & \wt{X}_r\ar@{-->}^{[1]}[ul]
      }
\end{equation}
with $\phi^U_1 > \ldots > \phi^U_r$, where $\phi^U_j:=\phi(\wt{X}_j)$, and 
\begin{equation}\label{diag:cone-decomp-step2-V}
\xymatrix@R16pt@C8pt{
        0=\M^V_1 \ar[rr]  & & \M^V_2 \ar[dl]\ar[rr]  & & \M^V_3 \ar[dl]\ar[r] & \cdots  \ar[r] & \M^V_{s-1} \ar[rr]  & & \M^V_s\cong V, \ar[dl]\\
        & \wt{Y}_1\ar@{-->}^{[1]}[ul] & & \wt{Y}_{3}\ar@{-->}^{[1]}[ul] &&& & \wt{Y}_s\ar@{-->}^{[1]}[ul]
      }
\end{equation}
with $\phi^V_1 > \ldots > \phi^V_s$, where $\phi^V_i:=\phi(\wt{Y}_i)$.

We will assume for now, that each semistable factor $\wt{X}_j$ (resp.~$\wt{Y}_i$) is isomorphic to an object in the image of some inclusion functor $\Ic^{h_j}$ as opposed to a direct sum of such object of equal phase (cf.~Step~1, Case~2 \ref{A4-Step-1-Case-2-i}).
Below we will see that this assumption can be made without loss of generality.
With this assumption in place, each semistable factor $\wt{X}_j$ (resp.~$\wt{Y}_i$) has a well defined height which we will denote by $h^U_j$ (resp.~$h^V_i$).
More precisely, from Step 1 we know that $\wt{X}_j \cong \Ic^{h^U_j}X_j$ and $\wt{Y}_i \cong \Ic^{h^V_i}Y_i$, where $X_j\in\P^M(\phi^{M,U}_j)$  and $Y_i\in\P^M(\phi^{M,V}_i)$.
Moreover, as in \eqref{eq:equal-phases}, the phases satisfy the following relations
\begin{equation}\label{eq:phase-relation-A4-step2}
  \phi^U_j = \phi^{M,U}_j-\kappa h^U_j \quad\text{and}\quad \phi^V_i = \phi^{M,V}_i - \kappa h^V_i.
\end{equation}

We may rewrite \eqref{diag:cone-decomp-step2-U} and \eqref{diag:cone-decomp-step2-V} as
\begin{align*}
 U\cong (\wt{X}_r[-1]\to\ldots \wt{X}_1[-1]\to 0) \quad \text{and}\quad V\cong (\wt{Y}_s[-1]\to\ldots \wt{Y}_1[-1]\to 0) 
\end{align*}
% \begin{align*}
%   (V,\theta_V)\cong (\wt{Y}_s[-1]\to\ldots \wt{Y}_1[-1]\to 0)
% \end{align*}
and hence
\begin{align}
  W&\cong ((\wt{X}_r[-1]\to\ldots \to\wt{X}_1[-1]\to 0) \to (\wt{Y}_s[-1]\to\ldots\to \wt{Y}_1[-1]\to 0))\nonumber\\
   &\cong (\wt{X}_r\to\ldots \to\wt{X}_1 \to\wt{Y}_s[-1]\to\ldots \to\wt{Y}_1[-1]\to 0)\label{eq:W-long-conedecomp-1}\\
   &\cong (\wt{X}_r\to\ldots \to(\wt{X}_1[-1] \overset{\beta}{\To}\wt{Y}_s[-1])\to\ldots \to\wt{Y}_1[-1]\to 0)\label{eq:(A4)-cone-decomp}
\end{align}
If $\phi^V_s-1>\phi^U_1$ then there is nothing to show.
So we assume that $\phi^V_s-1\leq\phi^U_1$ holds.
Moreover, we will also assume that $\beta$ is not an isomorphism, as otherwise the cone over $\beta$ is the zero-object.
\par\bigskip
%\begin{itemize}
\textbf{Case 1:} %\item [\textbf{Case 1:}]
\boldmath$\phi^V_s<\phi^U_1$\unboldmath.\par%\bigskip
  In this case, by \ref{axiom-3}, $\beta=0$ and therefore
$$
(\wt{X}_1[-1] \overset{0}{\To}\wt{Y}_s[-1])\cong (\wt{Y}_s[-2] \overset{0}{\To}\wt{X}_1)
$$
which leads to
  \begin{align*}
  W&\cong (\wt{X}_r\to\ldots \to(\wt{Y}_s[-2] \overset{0}{\To}\wt{X}_1)\to\ldots \to\wt{Y}_1[-1]\to 0)\\
   &\cong (\wt{X}_r\to\ldots \to \wt{Y}_s[-1] \to\wt{X}_1 \to\ldots \to\wt{Y}_1[-1]\to 0)
  \end{align*}
  where $\phi^V_s-1 < \phi^U_1$.
  Let $1\leq j\leq r$ and $1\leq i \leq s$ be such that
  \begin{align*}
   & \phi^V_s< \phi^U_j,\quad \phi^V_s\geq \phi^U_{j+1}\\
    \text{and}\quad & \phi^V_i< \phi^U_1,\quad \phi^V_{i-1}\geq \phi^U_1
  \end{align*}
  are satisfied.
  Then, by the above argument we may rewrite \eqref{eq:W-long-conedecomp-1} as
  \begin{align*}\label{eq:W-conedecomp-shuffle-1}
    W  \cong (\wt{X}_r\to&\ldots \to \wt{X}_{j+1}\to\wt{Y}_s[-1]\to\wt{X}_j\to \ldots \to\wt{X}_2 \to\wt{Y}_{s-1}[-1]\to\ldots\\
    &\ldots\to\wt{Y}_i[-1]\to\wt{X}_1\to\wt{Y}_{i-1}[-1]\to\ldots\to\wt{Y}_1[-1]\to 0).\nonumber
  \end{align*}
If $\phi^V_{s-1}<\phi^U_2$, we repeat the above argument again and again, until, after a finite number of iterations, we arrive at an iterated cone decomposition of $W$ in which the $\wt{X}_j$'s and $\wt{Y}_i$'s have been reshuffled (maintaining their internal order) in such a way that at every occurrence of
\begin{equation}\label{eq:occurrence}
\ldots\to\wt{X}_j\to\wt{Y}_i[-1]\to\ldots  
\end{equation}

 we have $\phi^U_j\leq\phi^V_i$, i.e. Case 1 does not occur anymore.
Thus, at each occurrence \eqref{eq:occurrence} in this reshuffled cone decomposition, the phase inequality is either correct already (i.e. $\phi^U_j<\phi^V_i-1$), in which case we don't do anything, or we have $\phi^U_j-1\leq \phi^V_i-1\leq \phi^U_j$, in which case  we proceed with Case 2 at this point in the iterated cone decomposition.\par
In what follows we only exemplify Case 2 at the occurrence of $\ldots \to\wt{X}_1 \to\wt{Y}_s[-1]\to\ldots$ in the original iterated cone decomposition \eqref{eq:W-long-conedecomp-1}, if however, reshuffling as described above is required, one has to apply Case 2 a finite number of times.
Namely at each occurrence \eqref{eq:occurrence} in the reshuffled iterated cone decomposition where the phase inequality $\phi^U_j-1\leq \phi^V_i-1\leq \phi^U_j$ holds.

\par\bigskip
\textbf{Case 2:} %\item[\textbf{Case 2:}] 
 \boldmath$\phi^U_1-1\leq\phi^V_s-1\leq \phi^U_1$.\unboldmath\par%\bigskip
  Notice that if $\beta=0$ and $\phi^V_s-1 < \phi^U_1$ we may go back to the argument of the first case, which we iteratively repeat a finite number of times as in Step 1 Case 2 \ref{A4-Step-1-Case-2-ii}.\\
  If $\beta=0$ and $\phi^V_s-1 = \phi^U_1$ we replace $(\wt{X}_1[-1] \overset{0}{\To}\wt{Y}_s[-1])$ by $\wt{X}_1\oplus \wt{Y}_s[-1]$, which is itself an object in the additive subcategory $\P(\phi^U_1)$, in the iterated cone decomposition of $W$.
  Note that we perform this replacement as a last step, i.e. after \ref{it:case2-step2-(i)}-\ref{it:case2-step2-(iii)} below (cf. the discussion after \ref{it:case2-step2-(iii)} below).
  
  \begin{enumerate}[label={(\roman*)}]
    \item\label{it:case2-step2-(i)} If $h^U_1<h^V_s$, then, as we have seen in Lemma~\ref{lem:hom-inclusion}, $\beta=0$.
          Hence we are in a situation that we have already discussed above.
    \item\label{it:case2-step2-(ii)} Suppose that $\beta \neq 0$ and $h^U_1>h^V_s$.\par
      We apply the restriction functor $\Rc_1$ to the decomposition \eqref{eq:(A4)-cone-decomp}.
      Note that 
      $$
       \Rc_1\wt{X}_1[-1]\cong X_1[-2] \quad \text{and}\quad \Rc_1\wt{Y}_s[-1]\cong Y_s[-2].
      $$
     By Lemma~\ref{lem:hom-inclusion} we obtain the identification
      \begin{align*}
        &\Hom_{D\fuk(\C\times M)}(\wt{X}_1[-1],\wt{Y}_s[-1])\\
        \cong & \Hom_{D\fuk(M)}(\Rc_1\wt{X}_1[-1],\Rc_1\wt{Y}_s[-1])\\
        \cong &\Hom_{D\fuk(M)}(X_1[-2],Y_s[-2])        
      \end{align*}
      under which $\beta$ corresponds to $\Rc_1\beta$.
      The assumption $\beta\neq 0$ thus implies that $\Rc_1\beta\neq 0$ as well.\par
    
      By \eqref{eq:phase-relation-A4-step2} we have $\phi^{M,U}_1:=\phi(X_1) = \phi^U_1+\kappa h^U_1$ and $\phi^{M,V}_s:= \phi(Y_s) = \phi^V_s + \kappa h^V_s$.
      Since $\Rc_1\beta\neq 0$ we get the phase inequality $\phi^{M,U}_1\leq \phi^{M,V}_s$ and hence (since $h^U_1 > h^V_s$ and $\kappa > 2$)
      \begin{align}\label{eq:height-difference-phase-difference-bigger-than-1-2}
       \mkern-10mu\phi^U_1=\phi^{M,U}_1 - \kappa h^U_1 \leq \phi^{M,V}_s - \underbrace{\kappa h^U_1}_{>\kappa h^V_s + 1} < \phi^{M,V}_s - \kappa h^V_s -1 = \phi^V_s - 1
      \end{align}
      which contradicts our assumption $\phi^V_s-1 \leq \phi^U_1$.

     \item\label{it:case2-step2-(iii)}
       Suppose that $\beta \neq 0$ and $h:=h^U_1=h^V_s$.\par
       Recall that we are still assuming that
       \begin{align}\label{eq:step2-case2-1}
         \phi^U_1-1\leq \phi^V_s-1\leq \phi^U_1
       \end{align}
       holds.\par
       Let $i\in\{1,\ldots,r\}$ and $j\in\{1,\ldots,s\}$ be such that
       \begin{align}
        \mkern-8mu \phi^U_1-1 \leq \boldsymbol{\phi^V_s-1 \leq \phi^U_i} \leq \phi^U_1 &\quad \text{and}\quad \phi^V_s-1 \leq \boldsymbol{\phi^V_j-1 \leq \phi^V_s}\label{eq:step2-case2-2}\\
         h=h^U_1=h^U_2=\dots=h^U_i &\quad\text{and}\quad h=h^V_s=h^V_{s-1}=\dots=h^V_j\label{eq:step2-case2-heights}\\
         \phi^U_{i+1} < \phi^V_s-1 &\quad\text{and}\quad \phi^V_s < \phi^V_{j-1}-1.\label{eq:step2-case2-3}
       \end{align}
       Note that since $\kappa > 2$ we have that if \eqref{eq:step2-case2-2} is satisfied, then \eqref{eq:step2-case2-heights} is satisfied as well (see \eqref{eq:height-difference-phase-difference-bigger-than-1} or \eqref{eq:height-difference-phase-difference-bigger-than-1-2}).
       Furthermore, note that the inequalities $\phi^U_i\leq \phi^U_1$ respectively $\phi^V_s-1\leq\phi^V_j-1$ in \eqref{eq:step2-case2-2} are in fact strict inequalities whenever $i\neq 1$ respectively $j\neq s$.
       We denote by
       \begin{align*}
         &\psi^U_k:=\phi(\Rc_h\wt{X}_k)= \phi^{M,U}_k = \phi^U_k+\kappa h \quad \text{for}\quad 1\leq k\leq i,\quad \text{and}\\ & \psi^V_n:=\phi(\Rc_h\wt{Y}_n)=\phi^{M,V}_n = \phi^V_n+\kappa h \quad \text{for}\quad j\leq n\leq s
       \end{align*}
       the phases of the semistable objects in $D\fuk(M)$ obtained by applying the restriction functor $\Rc_h$.
       Note that we have $\psi^U_i<\ldots <\psi^U_1$ and $\psi^V_s < \ldots <\psi^V_j$ and moreover analogues of \eqref{eq:step2-case2-1} and \eqref{eq:step2-case2-2} also hold for these phases.
       Remark \ref{rem:nested-cones} allows us to rewrite $W$ as 
       \begin{align*}
         W \cong &(\wt{X}_r\to\ldots \to\wt{X}_{i+1}\\
         \to&\underbrace{(\wt{X}_i[-1]\to \ldots\to \wt{X}_1[-1] \to\wt{Y}_s[-2]\to\ldots \to\wt{Y}_{j+1}[-2]\to\wt{Y}_j[-1])}_{\textstyle{=:\wt{Z}}}\\
         \to &\wt{Y}_{j-1}[-1]\to \ldots\to \wt{Y}_1[-1]\to 0).
       \end{align*}
       By the above inequalities and by repeatedly applying Remark~\ref{rem:extension-closed} we can determine the range of phases of the semistable factors of the restriction $\Rc_h\wt{Z}$:
       \begin{align*}
         &C_1:=(\underbrace{\Rc_h\wt{Y}_{j+1}[-2]}_{\in \P^M(\psi^V_{j+1}-2)}\to \underbrace{\Rc_h\wt{Y}_j[-1]}_{\in\P^M(\psi^V_j-1)}) &\implies& C_1\in\P^M([\psi^V_{j+1}-1,\psi^V_j-1])\\
         &\psi^V_s-1 < \psi^V_{j+1}-1 <\psi^V_j-1\leq\psi^V_s &\implies& C_1\in\P^M([\psi^V_{s}-1,\psi^V_s])\\
         &C_2:=(\underbrace{\Rc_h\wt{Y}_{j+2}[-2]}_{\in \P^M(\psi^V_{j+2}-2)}\to\mkern-20mu \underbrace{C_1}_{\in\P^M([\psi^V_{s}-1,\psi^V_s])}\mkern-36mu) &\implies& C_2\in\P^M([\psi^V_{s}-1,\psi^V_s])\\
         &\vdots&&\\
         &C_{s-j}:=(\underbrace{\Rc_h\wt{Y}_{s}[-2]}_{\in \P^M(\psi^V_{s}-2)}\to \mkern-16mu\underbrace{C_{s-(j+1)}}_{\in\P^M([\psi^V_{s}-1,\psi^V_s])}\mkern-19mu) &\implies& C_{s-j}\in\P^M([\psi^V_{s}-1,\psi^V_s])\\
         &C_{s-j+1}:=(\underbrace{\Rc_h\wt{X}_{1}[-1]}_{\in \P^M(\psi^U_{1}-1)}\to \mkern-20mu\underbrace{C_{s-j}}_{\in\P^M([\psi^V_{s}-1,\psi^V_s])}\mkern-33mu)& \overset{\eqref{eq:step2-case2-1}}{\implies}& C_{s-j+1}\in\P^M([\psi^V_{s}-1,\psi^V_s])\\
         %&\text{and}\quad \psi^V_s-1\leq\psi^U_1\leq \psi^V_s
         &C_{s-j+2}:=(\underbrace{\Rc_h\wt{X}_2[-1]}_{\in \P^M(\psi^U_2-1)}\to\mkern-20mu \underbrace{C_{s-j+1}}_{\in\P^M([\psi^V_{s}-1,\psi^V_s])}\mkern-25mu)& \overset{\eqref{eq:step2-case2-2}}{\implies}& C_{s-j+2}\in\P^M([\psi^V_{s}-1,\psi^V_s])\\
         &\vdots&&\\
         &C_{s-j+i}:=(\underbrace{\Rc_h\wt{X}_{i}[-1]}_{\in \P^M(\psi^U_{i}-1)}\to\mkern-20mu \underbrace{C_{s-j+i-1}}_{\in\P^M([\psi^V_{s}-1,\psi^V_s])}\mkern-18mu)& \overset{\eqref{eq:step2-case2-2}}{\implies}& C_{s-j+i}\in\P^M([\psi^V_{s}-1,\psi^V_s]). %\\
       \end{align*}
       In conclusion, we get
       \begin{align*}
         \Rc_h\wt{Z}\cong C_{s-j+i}\in\P^M([\psi^V_{s}-1,\psi^V_s]),
       \end{align*}
       which allows us to express $\Rc_h\wt{Z}$ as
       \begin{align}\label{eq:cone-decomp-of-RZ}
          \Rc_h\wt{Z} \cong (E_m[-1]\to E_{m-1}[-1]\to\ldots\to E_1[-1]\to 0)
       \end{align}
       where $E_\ell\in\P^M(\psi_{E_\ell})$, for $1\leq \ell\leq m$, are the semistable factors whose phases satisfy
       \begin{align*}
         \psi^V_s-1\leq\psi_{E_m} < \psi_{E_{m-1}}<\ldots < \psi_{E_1}\leq \psi^V_s.
       \end{align*}
       Applying the height $h$ inclusion functor $\Ic^h$ to \eqref{eq:cone-decomp-of-RZ}, and setting $\wt{E}_\ell:=\Ic^h E_\ell$ and $\phi_{\wt{E}_\ell}:=\psi_{E_\ell}-\kappa h$, we obtain the following iterated cone decomposition of $\wt{Z}$ in $D\fuk(\C\times M)$
       \begin{align*}
         \wt{Z}\cong \Ic^h\Rc_h\wt{Z}\cong (\wt{E}_m[-1]\to \wt{E}_{m-1}[-1]\to\ldots\to \wt{E}_1[-1]\to 0)
       \end{align*}
       where $\wt{E}_\ell\in\P(\phi_{\wt{E}_\ell})$ and the phases satisfy
       \begin{align*}
         \phi^V_s-1\leq \phi_{\wt{E}_m} < \phi_{\wt{E}_{m-1}}<\ldots < \phi_{\wt{E}_1}\leq \phi^V_s.
       \end{align*}
       Now, applying Remark \ref{rem:nested-cones} again, we replace $\wt{Z}$ in the iterated cone decomposition of $W$ as follows
       \begin{align}\label{eq:Step-2-Case-2-iterated-cone-decomp-descending-phases}
         W\cong  (\wt{X}_r\to & \ldots \to  \wt{X}_{i+1} \to \wt{E}_m\to\ldots \\ \nonumber
         &\ldots\to \wt{E}_1 \to \wt{Y}_{j-1}[-1]\to \ldots\to \wt{Y}_1[-1]\to 0).
       \end{align}
       Here, the phases satisfy the following inequalities:
       \begin{align*}
         &\phi^U_r+1 < \ldots <\phi^U_{i+1}+1,\\
         &\phi^V_s\leq \phi_{\wt{E}_m}+1 < \phi_{\wt{E}_{m-1}}+1<\ldots < \phi_{\wt{E}_1}+1\leq \phi^V_s+1,\\
         & \phi^V_{j-1}< \ldots <\phi^V_1,
       \end{align*}
       and moreover by \eqref{eq:step2-case2-3}
       \begin{align*}
         \phi^U_{i+1}+1<\phi^V_s \quad\text{and}\quad \phi^V_s+1 <\phi^V_{j-1}
       \end{align*}
       holds.
       The upshot is, that $\phi^U_{i+1}+1<\phi^V_s\leq\phi_{\wt{E}_m}+1$ and $\phi_{\wt{E}_1}+1\leq\phi^V_s+1 <\phi^V_{j-1}$ are indeed satisfied and we have thus found an iterated cone decomposition of $W$ with semistable factors of strictly descending phases, i.e. its Harder-Narasimhan filtration.
      
  \end{enumerate}

%\end{itemize}

In Step~2 we made the assumption that each semistable factor $\wt{X}_j$ (resp.~$\wt{Y}_i$) is isomorphic to an object in the image of some inclusion functor.
If we omit this assumption, then %by Step~1, Case~2 \ref{A4-Step-1-Case-2-i},
$\wt{X}_j$ may be isomorphic to a finite direct sum of the form $\wt{X}_j\cong \bigoplus_k\Ic^{h_k}X_{j,k}$, where each summand is semistable of phase $\phi^U_j$ (similarly for each $\wt{Y}_i$).
Therefore we may refine the above HN-filtration of $U$ \eqref{diag:cone-decomp-step2-U} (resp. of $V$ \eqref{diag:cone-decomp-step2-U}) to a longer cone decomposition in which each semistable factor is isomorphic to an object in the image of some inclusion functor $\Ic^{h_j}$, at the expense that the inequality of phases may only be descending and not necessarily \emph{strictly} descending anymore.
However, the arguments of Step 2 work in the same way for this refined cone decomposition, with the exception that some of the strict inequalities might become non-strict.
Hence, we end up with a decomposition as \eqref{eq:Step-2-Case-2-iterated-cone-decomp-descending-phases} in which some consecutive factors may have the same phase.
If this is the case, we can reverse the previous refinement and replace these consecutive factors by their direct sum, as in Step~1 Case~2 \ref{A4-Step-1-Case-2-i} (i.e.~by reversing how they were originally obtained).
After this process we end up with an iterated cone decomposition with strictly decreasing phases, i.e. with the HN-filtration of $W$.\par

So far we have shown that every cone $W\cong(U\to V)$ over two Yoneda-modules admits a HN-filtration.
The same argument shows that any cone $(Q\to W)$ respectively $(W\to Q)$, where $Q$ is a Yoneda-module, admits a HN-filtration as well.
Hence, by iteration, it follows that any iterated cone over Yoneda-modules admits an HN-filtration.
Since $D\fuk(\C\times M)$ is generated by Yoneda-modules this implies that $\P$ satisfies \ref{axiom-4}.\par

\end{enumerate}

This concludes the proof that $\P$ is a \emph{slicing} on $D\fuk(\C\times M)$.
It remains to check that the slicing $\P$ is locally-finite.
See Appendix \ref{appendix:qac} for some information on quasi-abelian categories.

\begin{claim}\label{claim:(Z,P)-locally-finite}
  If $(Z^M,\P^M)$ is a locally-finite stability condition, then $(Z,\P)$ is locally-finite as well.
\end{claim}

\begin{proof}
  Suppose that $(Z^M,\P^M)$ is locally-finite with parameter $\eta\in (0,\frac12)$ and let $\phi\in\R$.\par
  First note that, for any $\phi'\in (\phi-\eta,\phi+\eta)$, each object $\wt{X}:=\Ic^hX[r] \in \P(\phi')$ is isomorphic to $\Ic^{h}(\mathcal{C'})$, where $\mathcal{C'}$ is a finite cone decomposition coming from a finite Jordan-H\"older filtration of $\Rc_h \wt{X}$, which exists as $(Z^M,\P^M)$ is locally-finite.\par
  Every object of $\P((\phi-\eta,\phi+\eta))$ is isomorphic to a finite extension by objects of the subcategories $\P(\phi')$ with $\phi'\in (\phi-\eta,\phi+\eta)$.
  We will first check that any cone $(\wt{X}\to \wt{Y})$, with
  \begin{align*}
     \wt{X}\cong \Ic^{h'}X[r]\in \P(\phi'-1) \quad \text{and}\quad \wt{Y}\cong \Ic^{h''}Y[t]\in\P(\phi'')
  \end{align*}
and $\phi',\phi''\in (\phi-\eta,\phi+\eta)$ admits a finite Jordan-H\"older filtration.\par

  Let $\epsilon \in (-\eta,\eta)$ be such that $ \phi'+\epsilon =\phi''$.
  Note that $\abs{\epsilon}<1$ and
  \begin{equation*}
   X\in \P^M(\phi'-1-r+\kappa h')\quad \text{and}\quad Y\in \P^M(\phi''-t+\kappa h'').
\end{equation*}
 Since $(Z^M,\P^M)$ is locally finite, there are finite Jordan-H\"older filtrations of $X$ and $Y$:
 \begin{align*}
   X\cong (E_k[-1]\to\ldots\to E_1[-1]\to 0) \quad\text{and}\quad  Y\cong (F_n[-1]\to \ldots\to F_1[-1]\to 0).
 \end{align*}
 with stable factors $E_j\in\P^M(\phi'-1-r+\kappa h')$ for $1\leq j \leq k$ respectively $F_\ell\in \P^M(\phi'' -t + \kappa h'')$ for $1\leq \ell \leq n$.
 Setting $\wt{E}_j := \Ic^{h'}E_j[r]\in \P(\phi'-1)$ and $\wt{F}_\ell:=\Ic^{h''}F_j[t]\in \P(\phi'')$ we get the following JH-filtrations of $\wt{X}$ and $\wt{Y}$:
\begin{align*}
   \wt{X}\cong (\wt{E}_k[-1]\to\ldots\to \wt{E}_1[-1]\to 0) \quad\text{and}\quad  \wt{Y}\cong (\wt{F}_n[-1]\to \ldots\to \wt{F}_1[-1]\to 0).
 \end{align*}
 By Remark~\ref{rem:shifting-iterated-cones} we obtain the following iterated cone decomposition
 \begin{align}\label{eq:JH-filtration-locfin}
   (\wt{X}\to \wt{Y}) \cong &  ((\wt{E}_k[-1]\to\ldots\to \wt{E}_1[-1]\to 0)\to (\wt{F}_n[-1]\to \ldots\to \wt{F}_1[-1]\to 0))\nonumber \\
                      \cong &  (\wt{E}_k\to\ldots\to \wt{E}_1\to\wt{F}_n[-1]\to \ldots\to \wt{F}_1[-1]\to 0) \\
                      \cong &  (\wt{E}_k\to\ldots\to (\wt{E}_1[-1]\overset{\alpha}{\to}\wt{F}_n[-1])\to \ldots\to \wt{F}_1[-1]\to 0).\nonumber
 \end{align}
  \begin{itemize}
    \item If $h' < h''$, then by Lemma~\ref{lem:hom-inclusion} it follows that $\alpha=0$.\par
          If $\phi'\leq \phi''$, then \eqref{eq:JH-filtration-locfin} is already the desired JH-filtration of $(\wt{X}\to \wt{Y})$.\\
          If $\phi'>\phi''$, then, using Remark~\ref{rem:cone-direct-sum}, we may reorder the factors of \eqref{eq:JH-filtration-locfin} to obtain the following JH-filtration of $(\wt{X}\to \wt{Y})$
          \begin{align*}
            (\wt{X}\to \wt{Y}) \cong (\wt{F}_n[-1]\to \ldots\to \wt{F}_1[-1]\to\wt{E}_k\to\ldots\to \wt{E}_1\to 0).
          \end{align*}

    \item If $h'>h''$, then by Lemma~\ref{lem:hom-inclusion}, we have that 
      $$
      \Hom_{D\fuk(\C\times M)}(\wt{E}_1[-1],\wt{F}_n[-1])\cong \Hom_{D\fuk(M)}(E_1[r-1],F_n[t-1]).
      $$
      But we also get the following inequality of phases:
      \begin{align*}
        \phi(E_1[r-1])&=\phi'-1+\kappa h'-1\\
                     &= \phi'' - \underbrace{(1+\epsilon)}_{<2<\kappa} + \kappa h' -1\\
                     &> \phi'' + \kappa \underbrace{(h'-1)}_{\geq h''} -1 \\
                     &\geq \phi'' + \kappa h'' -1 \\
                     &= \phi(F_n[t-1])
      \end{align*}
      Therefore, by \ref{axiom-3}, $\alpha=0$ in this case as well and we may conclude as in the previous case.

    \item If $h'=h''$, then $\wt{X}$ and $\wt{Y}$ are on the same height $h$.
      Any cone $(\wt{X}\to\wt{Y})$ is isomorphic to $\Ic^h(\mathcal{C})$, where $\mathcal{C}$ denotes a finite iterated cone decomposition of $C:=(\Rc_h \wt{X}\to \Rc_h\wt{Y})$ in $D\fuk(M)$, coming from a Jordan-H\"older filtration of $C$.
      A finite Jordan-H\"older filtration of $C$ exists, since $(Z^M,\P^M)$ is locally-finite.
      Note that this argument also works for longer iterated cones, all of whose objects are on the same height.
  \end{itemize}
  
  As noted above, every object of $\P((\phi-\eta,\phi+\eta))$ is isomorphic to a finite extension by objects of the subcategories $\P(\phi')$ with $\phi'\in (\phi-\eta,\phi+\eta)$.
  Thus, by similar arguments as above we can obtain a finite JH-filtration of any object of $\P((\phi-\eta,\phi+\eta))$.
  In conclusion, each object of $\P(\phi-\eta,\phi+\eta)$ has a finite Jordan-H\"older decomposition and hence is of finite length.
\end{proof}

This completes the proof of Theorem~\ref{thm:(Z,P)-is-SC} and shows that $(Z,\P_\kappa)$ is a locally-finite stability condition on $D\fuk(\C\times M)$ for every $\kappa\in2\cdot \Z_{>1}$.
\end{proof}

%%% Local Variables:
%%% mode: latex
%%% TeX-master: "../SCLC-arxiv"
%%% End:

%% file: Appendix_triangles.tex
\section{Exact triangles and iterated cone decompositions in a triangulated category}\label{subsubsec:cones}
%\chapter[Exact triangles and iterated cone decompositions]{Exact triangles and iterated cone decompositions in a triangulated category}\label{subsubsec:cones}

We collect some remarks about iterated cones in a triangulated category $\D$.
There are many references on triangulated categories, some of which are \cite{weibel-book,Verdier96-PhD,Faisceaux-Pervers}.
 \par

In a triangulated category $\D$ every morphism $\alpha\in \Hom_{\D}(X,Y)$ can be embedded in an exact triangle $X\overset{\alpha}{\To}Y\To Z \To X[1]$.
The object $Z$ (or, by abuse of notation, any object isomorphic to $Z$) is said to be the \emph{cone} over the morphism $\alpha$ and we will denote it by $\mathrm{cone}(\alpha)$ or by $(X\overset{\alpha}{\to} Y)$.
In case the morphism over which we take the cone is not specified we will just write $(X\to Y)$ for the cone.\par
Recall that the \emph{Grothendieck group} (or \emph{$K$-group}) $K_0(\D)$ of a triangulated category $\D$ is the free abelian group generated by the objects of $\D$ modulo the following relations:
$X-Y+Z=0$ whenever there is an exact triangle $X\to Y\to Z\to X[1]$ in $\D$.\par

\begin{rem}\label{rem:cone-direct-sum}
  \begin{enumerate}
    \item The cone over any isomorphism is the zero object, that is:\newline
      $(X\overset{\cong}{\to}Y)\cong 0$ or put differently, $X\overset{\cong}{\to}Y\to 0\to X[1]$ is an exact triangle.\par
    \item The cone over the zero morphism is isomorphic to the direct sum (see \cite[Cor.~II.1.2.6]{Verdier96-PhD}).
          More precisely, we have
          $$
          (X\overset{0}{\to}Y) \cong X[1]\oplus Y \cong Y\oplus X[1] \cong (Y[-1]\overset{0}{\to}X[1]).
          $$
\end{enumerate}
\end{rem}

\begin{rem}\label{rem:K-group-cone-sign-relation}
  If an object $X\in \D$ admits an iterated cone decomposition of the form
$$
X\cong (X_1\to(X_2\to(X_3\to(\ldots\to(X_{m-1}\to X_m)\ldots)
$$
then the relation
$$
[X]=[X_m]-\sum_{i=1}^{m-1} [X_i]
$$
is satisfied in the $K$-group.
In particular, note that $\left[X[1]\right]=-[X]$.\par
We will simply write $(X_1\to X_2\to \ldots\to X_{m-1}\to X_m)$ for iterated cone decompositions as above.
\end{rem}

\begin{rem}\label{rem:shifting-iterated-cones}
  Using the Octahedral-Axiom (see~\cite[Def.~10.2.1]{weibel-book}) one can show that 
  $$(X\to(Y\to Z))\cong ((X[-1]\to Y)\to Z),$$
  where the morphisms will not be further specified.
  This allows us to reorder the brackets in an iterated cone decomposition.
  In particular, we can write
  \begin{align*}
  &(X_1\to X_2\to \ldots\to X_{m-1}\to X_m)\\ \cong &(X_1\to X_2\to \ldots(X_j[-1]\to X_{j+1})\to\ldots\to X_{m-1}\to X_m).
  \end{align*}
  with the convention of the previous remark.
\end{rem}

\begin{rem}\label{rem:nested-cones}
  Nested iterated cone  decompositions behave well in the following sense (cf.~\cite[Lemma~9.3]{Haug-T^2}).
  Suppose that an object $X\in \D$ admits an iterated cone decomposition of the form
$$
X\cong (X_1\to X_2\to \ldots\to X_{m-1}\to X_m)
$$
and that one of the objects $X_j$ admits a cone decomposition (note that, by shifting the last object by $-1$, we can always arrange for such an iterated cone decomposition to end in $0$) 
$$
X_j\cong (X_j^1\to X_j^2\to \ldots\to X_j^{k-1}\to X_j^k\to 0).
$$
Then $X$ admits the following iterated cone decomposition
$$
X\cong (X_1\to\mkern-0.2mu\ldots\mkern-0.2mu\to X_{j-1}\to X_j^1[1]\to\mkern-0.2mu\ldots \mkern-0.2mu\to X_{j}^{k-1}[1]\to X_j^k[1]\to X_{j+1}\to\mkern-0.2mu \ldots\to \mkern-0.2mu X_m).
$$
This follows by reordering the brackets as in the previous remark.
\end{rem}

%%% Local Variables:
%%% mode: latex
%%% TeX-master: "../SCLC-arxiv"
%%% End:

%% file: Appendix_HN-filtration.tex
\section{Uniqueness of the Harder-Narasimhan filtration}\label{appendix:HN-filtration}
%\chapter{Uniqueness of the Harder-Narasimhan filtration}\label{appendix:HN-filtration}

Since the proof of the uniqueness of the Harder-Narasimhan filtration seems to not be readily available in the literature, we will, for completeness' sake, give an argument in this section.\\ \par

For the remainder of this section let $\P$ denote a \emph{slicing} (see Definition~\ref{df:SC}) on a triangulated category $\D$.
Recall that a decomposition of a non-zero object $E$ of $\D$ as in axiom~\ref{axiom-4} is called a \emph{Harder-Narasimhan filtration} (or HN-filtration) of $E$.

\begin{prop}\label{prop:HN-unique}
  The Harder-Narasimhan filtration of any non-zero object of $\D$ is unique up to isomorphism of the semistable factors.
\end{prop}

In order to prove the above assertion we will require the following results.
The first of these is Lemma~\ref{lem:Gelfand-Manin-triangle-extension}, which can be found in \cite[Sect.~IV.1]{Gelfand-Manin} along with its proof.

\begin{lem}\label{lem:Gelfand-Manin-triangle-extension}
 Let $X\overset{u}{\To} Y \overset{v}{\To} Z \overset{w}{\To} X[1]$ and $X'\overset{u'}{\To} Y' \overset{v'}{\To} Z' \overset{w'}{\To} X[1]$ be two exact triangles and $g\in \Hom(Y,Y')$.
 If $v'gu=0$, then $g$ can be completed to a morphism of exact triangles $(f,g,h)$.
 If, in addition, $\Hom(X,Z'[-1])=0$, then this morphism of triangles is unique.
\begin{equation*}
\xymatrix{ 
         X \ar[r]^u \ar@{-->}[d]^{f} &Y \ar[r]\ar[d]^{g} & Z \ar[r]\ar@{-->}[d]^{h} &X[1]\ar@{-->}[d]^{f[1]} \\
          X' \ar[r] & Y'  \ar[r] & Z' \ar[r]  & X'[1] 
 }
\end{equation*}
\end{lem}

\begin{lem}\label{lem:app-HN-filtr-1}
  Suppose that $0\neq E\in \D$ has a HN-filtration
 \begin{equation*} \xymatrix@R16pt@C8pt{
        0=E_0 \ar[rr]  & & E_1 \ar[dl]\ar[rr]  & & E_2 \ar[dl]\ar[r] & \cdots  \ar[r] & E_{n-1} \ar[rr]  & & E_n=E \ar[dl]\\
        & A_1\ar@{-->}[ul] & & A_2\ar@{-->}[ul] &&& & A_n\ar@{-->}[ul]
      }
\end{equation*}
with $A_j\in\P(\phi_j)$ for $j=1,\ldots,n$ and $\phi_1>\ldots>\phi_n$.
Let $A\in\P(\phi)$ with $\phi>\phi_1$, then $\Hom(A,E)=0$.
\end{lem}
\begin{proof}
  For each $j=1,\ldots,n$ we have an exact triangle $A_j[-1]\to E_{j-1}\to E_j \to A_j$ and hence (see \cite[Chap.~2~Prop.~1.2.1]{Verdier96-PhD}) we obtain the following long exact sequences
  \begin{equation*}
   \To \Hom(A,A_j[-1])\To \Hom(A,E_{j-1})\To \Hom(A,E_j) \To \Hom(A,A_j)\To
  \end{equation*}
  Axiom~\ref{axiom-3} implies that $\Hom(A,A_j[-1])=\Hom(A,A_j)=0$ and thus we get that $\Hom(A,E_{j-1})\cong \Hom(A,E_j)$.
  Therefore we conclude that $\Hom(A,E)\cong \Hom(A,E_{n-1})\cong\ldots\cong \Hom(A,E_0)=0$.
\end{proof}

\begin{lem}\label{lem:app-HN-filtr-2}
  Suppose that $E$ and $E'$ are two non-zero objects of $\D$ with HN-filtrations
 \begin{equation*} \xymatrix@R16pt@C8pt{
        0=E_0 \ar[rr]  & & E_1 \ar[dl]\ar[rr]  & & E_2 \ar[dl]\ar[r] & \cdots  \ar[r] & E_{n-1} \ar[rr]  & & E_n=E \ar[dl]\\
        & A_1\ar@{-->}[ul] & & A_2\ar@{-->}[ul] &&& & A_n\ar@{-->}[ul]
      }
\end{equation*}
 \begin{equation*} \xymatrix@R16pt@C8pt{
        0=E'_0 \ar[rr]  & & E'_1 \ar[dl]\ar[rr]  & & E'_2 \ar[dl]\ar[r] & \cdots  \ar[r] & E'_{m-1} \ar[rr]  & & E'_m=E' \ar[dl]\\
        & A'_1\ar@{-->}[ul] & & A'_2\ar@{-->}[ul] &&& & A'_m\ar@{-->}[ul]
      }
\end{equation*}
with $A_j\in\P(\phi_j)$, $A'_k\in\P(\phi'_k)$ satisfying $\phi_1>\ldots > \phi_n>\phi'_1>\ldots \phi'_m$.
Then $\Hom(E,E')=0$.
\end{lem}

\begin{proof}
  The proof is by induction on $n$.
  If $n=1$, then the result holds by Lemma~\ref{lem:app-HN-filtr-1}.
  Assume now that the result holds for $n=N-1$, then $\Hom(E_{N-1},E')=0$.
  The exact triangle $E_{N-1}\to E_N\to A_N \to E_{N-1}[1]$ yields (see \cite[Chap.~2~Prop.~1.2.1]{Verdier96-PhD}) the long exact sequence
  \begin{equation*}
  \to \Hom(E_{N-1}[1],E')\To \Hom(A_N,E')\To \Hom(E_N,E') \To \Hom(E_{N-1},E')\to%\ldots
  \end{equation*}
  By Lemma~\ref{lem:app-HN-filtr-1} we have that $\Hom(A_N,E')=0$.
  Since $\Hom(E_{N-1},E')=0$, we conclude that $\Hom(E_N,E')=0$.  
\end{proof}

\begin{lem}\label{lem:app-HN-filtr-3}
  Let $0\neq E\in \D$ and suppose that $E$ has two HN-filtrations
 \begin{equation*} \xymatrix@R16pt@C8pt{
        0=E_0 \ar[rr]  & & E_1 \ar[dl]\ar[rr]  & & E_2 \ar[dl]\ar[r] & \cdots  \ar[r] & E_{n-1} \ar[rr]  & & E_n=E \ar[dl]\\
        & A_1\ar@{-->}[ul] & & A_2\ar@{-->}[ul] &&& & A_n\ar@{-->}[ul]
      }
\end{equation*}
 \begin{equation*} \xymatrix@R16pt@C8pt{
        0=E'_0 \ar[rr]  & & E'_1 \ar[dl]\ar[rr]  & & E'_2 \ar[dl]\ar[r] & \cdots  \ar[r] & E'_{m-1} \ar[rr]  & & E'_m=E \ar[dl]\\
        & A'_1\ar@{-->}[ul] & & A'_2\ar@{-->}[ul] &&& & A'_m\ar@{-->}[ul]
      }
\end{equation*}
with $A_j\in\P(\phi_j)$, $A'_k\in\P(\phi'_k)$, $\phi_1>\ldots > \phi_n$ and $\phi'_1>\ldots >\phi'_m$.
Then $\phi_1=\phi'_1$.
\end{lem}

\begin{proof}
  Suppose without loss of generality that $\phi_1>\phi'_1$.
  Then Lemma~\ref{lem:app-HN-filtr-1} implies that $\Hom(A_1,E)=0$.
  But, by the same argument as in the proof of Lemma~\ref{lem:app-HN-filtr-1} we obtain the following long exact sequence
  \begin{equation*}
  \To \underbrace{\Hom(A_1,A_j[-1])}_{=0}\To \Hom(A_1,E_{j-1})\To \Hom(A_1,E_j) \To \underbrace{\Hom(A_1,A_j)}_{=0}\To
  \end{equation*}
  for every $j=2,\ldots,n$.
  This implies that $\Hom(A_1,E)\cong \Hom(A_1,E_{n-1})\cong\ldots\cong \Hom(A_1,E_1)\cong \Hom(A_1,A_1)\neq 0$, since $E_1\cong A_1$ and $A_1\neq 0$.
  We have arrived at a contradiction and conclude that $\phi_1=\phi'_1$.
\end{proof}

\begin{notation}
  Given a HN-filtration 
 \begin{equation*} \xymatrix@R16pt@C8pt{
        0=E_0 \ar[rr]  & & E_1 \ar[dl]\ar[rr]  & & E_2 \ar[dl]\ar[r] & \cdots  \ar[r] & E_{n-1} \ar[rr]  & & E_n=E \ar[dl]\\
        & A_1\ar@{-->}[ul] & & A_2\ar@{-->}[ul] &&& & A_n\ar@{-->}[ul]
      }
\end{equation*}
of a non-zero object $E$ of $\D$, we denote by $\phi^-(E)$, resp. $\phi^+(E)$, the \emph{lowest}, resp. \emph{highest}, phase of the semistable factors occurring in this particular filtration, i.e. $\phi^-(E)=\phi_n$ and $\phi^+(E)=\phi_1$.
Note that the filtration is left implicit in this notation (later we will see that the HN-filtration is unique).
\end{notation}

\begin{lem}\label{lem:app-HN-filtr-4}
  Let $E,F$ be two non-zero objects of $\D$, $\psi\in \R$ a real number and $E'\to E \to E''\to E'[1]$, $F'\to F \to F''\to F'[1]$ be two exact triangles.
  Suppose that $E',E''$ and $F',F''$ have HN-filtrations satisfying
  \begin{equation*}
    \phi^-(E')> \psi \geq \phi^+(E'')\quad \text{and}\quad \phi^-(F')> \psi \geq \phi^+(F'').
  \end{equation*}
  Then any morphism $f\in\hom(E,F)$ can be uniquely extended to a morphism of triangles $(f',f,f'')$, i.e. there are unique $f'\in\Hom(E',F')$ and $f''\in\Hom(E'',F'')$ such that the following diagram is commutative.
\begin{equation*}
  \xymatrix{ 
         E' \ar[r] \ar@{-->}[d]^{f'} &E \ar[r]\ar[d]^{f} & E'' \ar[r]\ar@{-->}[d]^{f''} & E'[1]\ar@{-->}[d]^{f'[1]} \\
          F' \ar[r] & F  \ar[r] & F'' \ar[r]  & F'[1] 
 }
  \end{equation*}
  If, moreover, $f\in\hom(E,F)$ is an isomorphism, then $f'$ and $f''$ are isomorphisms as well.
\end{lem}

\begin{proof}
  Since we have
  $$
  \phi^-(E')>\psi\geq\phi^+(F'')>\phi^+(F''[-1])=\phi^+(F'')-1,
  $$
  Lemma~\ref{lem:app-HN-filtr-2} implies that $\Hom(E',F'')=\Hom(E',F''[-1])=0$.
  Hence, by Lem\-ma~\ref{lem:Gelfand-Manin-triangle-extension}, $f$ can be uniquely extended to a morphism of triangles $(f',f,f'')$.\par
  Suppose that $f$ is an isomorphism, then there exists $g\in \hom(F,E)$ such that $gf=\id_E$ and $fg=\id_F$.
  then, by the previous argument, we can uniquely extend $g$ to a morphism of triangles $(g',g,g'')$ with $g'\in\hom(F',E')$ and $g''\in\hom(F'',E'')$.
  Now, $(\id_{E'},\id_E,\id_{E''})$ and $(g'f',gf=\id_E,g''f'')$ are both morphism of triangles extending $\id_E$.
  Hence, uniqueness of this extension implies that $g'f'=\id_{E'}$ and $g''f''=\id_{E''}$.
  Similarly it follows that $f'g'=\id_{F'}$ and $f''g''=\id_{F''}$ and therefore $f'$ and $f''$ are isomorphisms.
\end{proof}

\begin{lem}\label{lem:app-HN-filtr-5}
  Suppose that $0\neq E\in \D$ has the following two HN-filtrations
   \begin{equation*} 
     \xymatrix@R16pt@C8pt{ 0=E_0 \ar[rr] &&E_1=E \ar[dl]\\
       &A_1\ar@{-->}[ul]&
     }
    \quad\mkern-12mu \text{and}\quad \mkern-12mu
     \xymatrix@R16pt@C8pt{
        0=E'_0 \ar[rr]  & & E'_1 \ar[dl]\ar[r]  & \cdots  \ar[r] & E'_{m-1} \ar[rr]  & & E'_m=E \ar[dl]\\
        & A'_1\ar@{-->}[ul] &&& & A'_m\ar@{-->}[ul]
      }
\end{equation*}
with $A_1\in\P(\phi_1)$, $A'_j\in\P(\phi'_j)$ and $\phi'_1>\ldots > \phi'_m$.
Then $A_1$ is isomorphic to $A'_1$ and $m=1$.
\end{lem}

\begin{proof}
  First, by Lemma~\ref{lem:app-HN-filtr-3} we have that $\phi_1=\phi'_1$.
  Since $\phi_1=\phi'_1>\ldots >\phi'_m$, Lemma~\ref{lem:app-HN-filtr-4} implies that the identity morphism $\id_E\in\Hom(E,E)$ can be uniquely extended to a morphism of triangles where all the vertical morphisms are isomorphisms.
  \begin{equation*}
  \xymatrix{ 
         A_1 \ar[r] \ar@{-->}[d]^{\cong} &A_1\cong E \ar[r]\ar[d]^{\id_E} & 0 \ar[r]\ar@{-->}[d]^{\cong} & A_1[1]\ar@{-->}[d]^{\cong} \\
          E'_{m-1} \ar[r] & E'_m=E  \ar[r] & A'_m \ar[r]  & E'_{m-1}[1] 
 }
  \end{equation*}
  By iteration we get that $A'_j=0$ and $E'_{j-1}\cong E$ for every $j=2,\ldots,m$ from which we conclude that $A'_1\cong A_1$ and $E'_1\cong E$.
\end{proof}

\begin{rem}\label{rem:app-HN-filtr-rem1}
  Suppose that $0\neq E\in \D$ has a HN-filtration
   \begin{equation*} \xymatrix@R16pt@C8pt{
        0=E_0 \ar[rr]  & & E_1 \ar[dl]\ar[rr]  & & E_2 \ar[dl]\ar[r] & \cdots  \ar[r] & E_{n-1} \ar[rr]  & & E_n=E \ar[dl]\\
        & A_1\ar@{-->}[ul] & & A_2\ar@{-->}[ul] &&& & A_n\ar@{-->}[ul]
      }
\end{equation*}
with $A_j\in\P(\phi_j)$ and $\phi_1>\ldots > \phi_n$.
By rotating these triangles and shifting the brackets (see Remark~\ref{rem:shifting-iterated-cones}) we may, for every $k=2,\ldots,n$, express $E$ as 
\begin{equation*}
  E\cong(A_n[-1]\to \ldots A_k[-1]\to E_{k-1})\cong ((A_n[-2]\to \ldots A_k[-2]\to 0)\to E_{k-1}).
\end{equation*}
Defining $G_k:=(A_n[-1]\to \ldots A_k[-1]\to 0)$ we have that $E\cong (G_k[-1]\to E_{k-1})$ and $G_k$ admits a HN-filtration with $\phi^+(G_k)=\phi_k$ and $\phi^-(G_k)=\phi_n$.
\end{rem}

With these ingredients we can proceed to prove the uniqueness of the HN-filtration up to isomorphism of the semistable factors.

\begin{proof}[Proof of Proposition~\ref{prop:HN-unique}]
  Suppose that $0\neq E\in \D$ has two HN-filtration as follows
 \begin{equation*} \xymatrix@R16pt@C8pt{
        0=E_0 \ar[rr]  & & E_1 \ar[dl]\ar[rr]  & & E_2 \ar[dl]\ar[r] & \cdots  \ar[r] & E_{n-1} \ar[rr]  & & E_n=E \ar[dl]\\
        & A_1\ar@{-->}[ul] & & A_2\ar@{-->}[ul] &&& & A_n\ar@{-->}[ul]
      }
\end{equation*}
 \begin{equation*} \xymatrix@R16pt@C8pt{
        0=E'_0 \ar[rr]  & & E'_1 \ar[dl]\ar[rr]  & & E'_2 \ar[dl]\ar[r] & \cdots  \ar[r] & E'_{m-1} \ar[rr]  & & E'_m=E \ar[dl]\\
        & A'_1\ar@{-->}[ul] & & A'_2\ar@{-->}[ul] &&& & A'_m\ar@{-->}[ul]
      }
\end{equation*}
with $A_j\in\P(\phi_j)$, $A'_k\in\P(\phi'_k)$, $\phi_1>\ldots > \phi_n$ and $\phi'_1>\ldots \phi'_m$.
Then, by Lemma~\ref{lem:app-HN-filtr-3} we have that $\phi_1=\phi'_1$.
Let $G_2$ and $G'_2$ be the objects constructed from the above filtrations as in Remark~\ref{rem:app-HN-filtr-rem1} for $k=2$.
Choose $\psi\in \R$ such that $\phi_1=\phi'_1>\psi\geq \phi_2$ and $\phi_1=\phi'_1>\psi\geq \phi'_2$ are satisfied.
Note that 
\begin{align*}
\phi^-(E_1)=\phi_1 &> \psi \geq \phi^+(G_2)=\phi_2 \\
\phi^-(E'_1)=\phi'_1=\phi_1 &> \psi \geq \phi^+(G'_2)=\phi'_2
\end{align*}
Now, by Lemma~\ref{lem:app-HN-filtr-4}, $\id_E\in\Hom(E,E)$ can be uniquely extended to a morphism of triangles where the vertical morphisms are isomorphisms.
  \begin{equation*}
  \xymatrix{ 
         E_1\cong A_1 \ar[r] \ar@{-->}[d]^{\cong} &E \ar[r]\ar[d]^{\id_E} & G_2 \ar[r]\ar@{-->}[d]^{\cong} & E_1[1]\ar@{-->}[d]^{\cong} \\
          E'_1\cong A'_1 \ar[r] & E  \ar[r] & G'_2 \ar[r]  & E'_1[1] 
 }
  \end{equation*}
In particular note that $A_1\cong A'_1$.
We may now repeat the same argument with the HN-filtrations of $G_2$ and $G'_2$ and iteratively split off the first semistable factor until, after $\min\{m,n\}-1$ steps one of the two filtrations has length 1.
Then Lemma~\ref{lem:app-HN-filtr-5} implies that the other filtration must also have length 1, i.e. $m=n$.
Thus we conclude that $A_j\cong A'_j$ for every $j=1,\ldots,n$.
This shows that the HN-filtration of a non-zero object is unique up to isomorphism of the semistable factors.
\end{proof}

%%% Local Variables:
%%% mode: latex
%%% TeX-master: "../SCLC-arxiv"
%%% End:

%% file: Appendix_qac.tex
\section{Basics on Quasi-Abelian Categories}\label{appendix:qac}
%\chapter{Basics on Quasi-Abelian Categories}\label{appendix:qac}

We will limit ourselves here to the definition and some elementary facts, without proofs, about quasi-abelian categories following \cite[Sect.~4]{Bridgeland-SC} and \cite{Schneiders-qac}.
This collection of definitions and results is included for the convenience of the reader, for an extensive treatment on the subject we refer to Schneiders \cite{Schneiders-qac}.\\ \par

Before we give the definition of quasi-abelian categories, we will recall how the notions of \emph{subobject} and \emph{quotient object} are defined (see e.g.~\cite{weibel-book,Freyd-Abelian_Categories}).

\begin{df}\label{df:subobject-quotient_object}
  Let $\CC$ be a category.
  \begin{itemize}
    \item A morphism $m:A\to B$ in $\CC$ is a \emph{monomorphism} if for any two distinct morphisms $h_1, h_2:C\to A$ we have $mh_1\neq mh_2$.
    \item Two monomorphisms $m_1:A_1\to B$ and $m_2:A_2\to B$ are \emph{equivalent} if there are morphisms $A_1\to A_2$ and $A_2\to A_1$ such that the following two diagrams commute:
      \begin{align*}
\begin{minipage}{0.4\textwidth}
        \xymatrix{
         A_1\ar[d]\ar[rd]^{m_1} &\\
         A_2\ar[r]_{m_2}  & B       
}
\end{minipage}%
\hspace{2cm}
\begin{minipage}{0.4\textwidth}
        \xymatrix{
         A_2\ar[d]\ar[rd]^{m_2} &\\
         A_1\ar[r]_{m_1}  & B       
}
\end{minipage}                      
      \end{align*}

\item A \emph{subobject} of $B$ is an equivalence class of monomorphisms with target $B$.

\item A morphism $e:A\to B$ in $\CC$ is called \emph{epimorphism} if for any two distinct morphisms $g_1,g_2:B\to C$ we have $g_1e\neq g_2e$.

\item  Two epimorphisms $e_1:A\to B_1$ and $e_2:A\to B_2$ are \emph{equivalent} if there are morphisms $B_1\to B_2$ and $B_2\to B_1$ such that the following two diagrams commute:
      \begin{align*}
\begin{minipage}{0.4\textwidth}
        \xymatrix{
         &B_1\ar[d]\\
         A\ar[r]_{e_2}\ar[ru]^{e_1}&B_2       
}
\end{minipage}%
\hspace{2cm}
\begin{minipage}{0.4\textwidth}
        \xymatrix{
         &B_2\ar[d]\\
         A\ar[r]_{e_1}\ar[ru]^{e_2}&B_1       
}
\end{minipage}                      
      \end{align*}
  
\item A \emph{quotient object} of $A$ is an equivalence class of epimorphisms with source $A$.

\end{itemize}
\end{df}

\begin{rem}
  If a subobject of $B$ is given by the equivalence class of a monomorphism $m:A\to B$, one often says that ``$A$ is a subobject of $B$'', which is slightly imprecise.
  Similarly, if a quotient object of $A$ is given by the equivalence class of an epimorphism $e:A\to B$ one sometimes says that ``$B$ is a quotient object of $A$''.
\end{rem}

\begin{df}
  Let $\B$ be an additive category with kernels and cokernels and let $f:E\to F$ be a morphism in $\B$.
  \begin{itemize}
  \item The \emph{image} of $f$ is the kernel of the canonical morphism $F\to \mathrm{coker} f$.
  \item The \emph{coimage} of $f$ is the cokernel of the canonical morphism $\mathrm{ker}f\to E$.
  \item The morphism $f$ is called \emph{strict} if the canonical morphism $\mathrm{coim}f\to\mathrm{im}f$ induced by $f$ is an isomorphism.
  \item The category $\B$ is \emph{quasi-abelian} if the following two axioms hold
    \begin{itemize}
    \item[(QAI)] In a pullback diagram
      \begin{align*}
        \xymatrix{
          E\ar[r]^f & F\\
          E'\ar[u]\ar[r]^{f'} & F'\ar[u]_g
}
      \end{align*}
      where $f$ is a strict epimorphism, $f'$ is a strict epimorphism as well.
     \item[(QAII)] In a pushout diagram
      \begin{align*}
        \xymatrix{
          E'\ar[r]^{f'} & F'\\
          E\ar[u]^g\ar[r]^f & F\ar[u]
}
      \end{align*}
      where $f$ is a strict monomorphism, $f'$ is a strict monomorphism as well.
   \end{itemize}
  \end{itemize}
\end{df}

We now list some elementary properties without proofs.\par
Let $\B$ be an additive category with kernels and cokernels.
For any morphism $f:E\to F$, the canonical morphism $\mathrm{ker}f\to E$ (resp.~$F\to \mathrm{coker}f$) is a strict monomorphism (resp. epimorphism).
Moreover, if $f:E\to F$ is a strict monomorphism (resp. epimorphism), then $f$ is the kernel (resp. cokernel) of $F\to \mathrm{coker}f$ (resp.~$\mathrm{ker}f\to E$).
In an abelian category, every monomorphism is a kernel and every epimorphism is a cokernel.
Thus an abelian category is a quasi-abelian category in which every monomorphism and every epimorphism is strict.
A morphism $f$ is strict if and only if it factorizes as $f=i\circ q$ where $i$ is a strict monomorphism and $q$ a strict epimorphism.
Therefore an abelian category is a quasi-abelian category where every morphism is strict.
The class of strict epimorphisms (resp.~monomorphisms) is closed under composition (see~\cite[Prop.~1.1.7]{Schneiders-qac}).
Suppose that
\begin{align*}
  \xymatrix{
  &F\ar[rd]^g&\\
E\ar[ru]^f\ar[rr]_h&&G
}
\end{align*}
is a commutative diagram in the category $\B$.
If $h$ is a strict epimorphism, then $g$ is a strict epimorphism and similarly, if $h$ is a strict monomorphism, then $f$ is a strict monomorphism (see~\cite[Prop.~1.1.8]{Schneiders-qac}).

\begin{df}
  A \emph{strict short exact sequence} in a quasi-abelian category $\B$ is a sequence
$$
0\To E\overset{i}\To F \overset{q}\To G\To 0
$$
where $i$ is the kernel of $q$ and $q$ is the cokernel of $i$.
\end{df}
Note that by the remarks above it follows that both $i$ and $q$ are strict.
Furthermore, any strict monomorphism $i$ fits into a strict short exact sequence by taking $j$ to be its cokernel.
Similarly, any strict epimorphism $j$ fits into a strict short exact sequence by taking $i$ to be its kernel.
As for abelian categories, one can define the Grothendieck group $K(\B)$ of a quasi-abelian category $\B$ as the free abelian group generated by the objects of $\B$, subject to the relation $[F]=[E]+[G]$ for each strict short exact sequence $0\To E\To F \To G\To 0$.\par
Let $\B$ be a quasi-abelian category and $E,F,G$ objects of $\B$.
Following Bridgeland \cite{Bridgeland-SC} we say that $E$ is a \emph{strict subobject} of $F$, denoted by $E\subset F$, if there exists a strict monomorphism $E\overset{i}{\To}F$ and we denote the cokernel of $i$ by $F/E$.
Similarly, $G$ is referred to as a \emph{strict quotient} of $E$, denoted by $E\twoheadrightarrow G$, if there is a strict epimorphism $E\overset{q}{\To}G$.

In analogy to Definition~\ref{df:stable-Jordan-Hoelder} we make the following definition. 

\begin{df}
 Let $\B$ be a quasi-abelian category.
  \begin{enumerate}
    \item An object $E$ of $\B$ is \emph{strictly stable} if it does not admit any strict subobjects or strict quotients.
    \item The quasi-abelian category $\B$ is of \emph{finite length} if every object $E\in\B$ admits a \emph{Jordan-H\"older filtration}.
      A Jordan-H\"older filtration of an object $E$ is a finite sequence of \emph{strict} subobjects
      $$
      0=E_0\subset E_{1}\subset \ldots E_{n-1}\subset E_n=E
      $$
      such that each (strict) quotient $E_j/E_{j-1}$ is a strictly stable object.
  \end{enumerate}
\end{df}
A quasi-abelian category is artinian (resp.~noetherian) if any descending (resp.~ascending) chain of strict subobject stabilizes.
As in the abelian case, a quasi-abelian category is of finite length if and only if it is artinian and noetherian.

%%% Local Variables:
%%% mode: latex
%%% TeX-master: "../SCLC-arxiv"
%%% End:

%% file: SCLC-arxiv.bbl
\begin{thebibliography}{FOOO09b}

\bibitem[AS10]{Abouzaid-Smith-T^4}
Mohammed Abouzaid and Ivan Smith.
\newblock Homological mirror symmetry for the 4-torus.
\newblock {\em Duke Math. J.}, 152(3):373{--}440, 2010.

\bibitem[Ati57]{Atiyah-VBEC}
M.~F. Atiyah.
\newblock Vector bundles over an elliptic curve.
\newblock {\em Proc. London Math. Soc. (3)}, 7:414--452, 1957.

\bibitem[Aur14]{Auroux-beginner}
Denis Auroux.
\newblock A beginner's introduction to {F}ukaya categories.
\newblock In {\em Contact and symplectic topology}, volume~26 of {\em Bolyai
  Soc. Math. Stud.}, pages 85{--}136. J{\'a}nos Bolyai Math. Soc., Budapest,
  2014.

\bibitem[BBD82]{Faisceaux-Pervers}
A.~A. {Be\u{\i}linson}, J.~Bernstein, and P.~Deligne.
\newblock Faisceaux pervers.
\newblock In {\em Analysis and topology on singular spaces, {I} ({L}uminy,
  1981)}, volume 100 of {\em Ast{\'e}risque}, pages 5{--}171. Soc. Math.
  France, Paris, 1982.

\bibitem[BBDG06]{Bodnarchuk-Burban-Drozd-Greuel}
Lesya Bodnarchuk, Igor Burban, Yuriy Drozd, and Gert-Martin Greuel.
\newblock Vector bundles and torsion free sheaves on degenerations of elliptic
  curves.
\newblock In {\em Global aspects of complex geometry}, pages 83--128. Springer,
  Berlin, 2006.

\bibitem[BC13]{BC-CobI}
Paul Biran and Octav Cornea.
\newblock Lagrangian cobordism. {I}.
\newblock {\em J. Amer. Math. Soc.}, 26(2):295{--}340, 2013.

\bibitem[BC14]{BC-CobII}
Paul Biran and Octav Cornea.
\newblock Lagrangian cobordism and {F}ukaya categories.
\newblock {\em Geom. Funct. Anal.}, 24(6):1731{--}1830, 2014.

\bibitem[BC17]{BC-Lef}
Paul Biran and Octav Cornea.
\newblock Cone-decompositions of {L}agrangian cobordisms in {L}efschetz
  fibrations.
\newblock {\em Selecta Mathematica}, 2017.

\bibitem[BK06]{Burban-Kreussler-06}
Igor Burban and Bernd Kreussler.
\newblock Derived categories of irreducible projective curves of arithmetic
  genus one.
\newblock {\em Compos. Math.}, 142(5):1231--1262, 2006.

\bibitem[BMT14]{Bayer-Macri-Toda-2014}
Arend Bayer, Emanuele Macri, and Yukinobu Toda.
\newblock Bridgeland stability conditions on threefolds {I}:
  {B}ogomolov-{G}ieseker type inequalities.
\newblock {\em J. Algebraic Geom.}, 23(1):117{--}163, 2014.

\bibitem[Bri07]{Bridgeland-SC}
Tom Bridgeland.
\newblock Stability conditions on triangulated categories.
\newblock {\em Ann. of Math. (2)}, 166(2):317{--}345, 2007.

\bibitem[Bri08]{Bridgeland-SC_on_K3}
Tom Bridgeland.
\newblock Stability conditions on {$K3$} surfaces.
\newblock {\em Duke Math. J.}, 141(2):241{--}291, 2008.

\bibitem[Dou02]{Douglas-stability}
Michael~R. Douglas.
\newblock Dirichlet branes, homological mirror symmetry, and stability.
\newblock In {\em Proceedings of the {I}nternational {C}ongress of
  {M}athematicians, {V}ol. {III} ({B}eijing, 2002)}, pages 395{--}408. Higher
  Ed. Press, Beijing, 2002.

\bibitem[FOOO]{FOOO-Ch10}
Kenji Fukaya, Yong-Geun Oh, Hiroshi Ohta, and Kaoru Ono.
\newblock {\em Lagrangian intersection {F}loer theory: anomaly and obstruction.
  {C}hapter 10}.
\newblock unpublished.

\bibitem[FOOO09a]{FOOO-bookI}
Kenji Fukaya, Yong-Geun Oh, Hiroshi Ohta, and Kaoru Ono.
\newblock {\em Lagrangian intersection {F}loer theory: anomaly and obstruction.
  {P}art {I}}, volume~46 of {\em AMS/IP Studies in Advanced Mathematics}.
\newblock American Mathematical Society, Providence, RI; International Press,
  Somerville, MA, 2009.

\bibitem[FOOO09b]{FOOO-bookII}
Kenji Fukaya, Yong-Geun Oh, Hiroshi Ohta, and Kaoru Ono.
\newblock {\em Lagrangian intersection {F}loer theory: anomaly and obstruction.
  {P}art {II}}, volume~46 of {\em AMS/IP Studies in Advanced Mathematics}.
\newblock American Mathematical Society, Providence, RI; International Press,
  Somerville, MA, 2009.

\bibitem[FOOO10]{FOOO-toric-I}
Kenji Fukaya, Yong-Geun Oh, Hiroshi Ohta, and Kaoru Ono.
\newblock Lagrangian {F}loer theory on compact toric manifolds. {I}.
\newblock {\em Duke Math. J.}, 151(1):23--174, 2010.

\bibitem[Fre64]{Freyd-Abelian_Categories}
Peter Freyd.
\newblock {\em Abelian categories. {A}n introduction to the theory of
  functors}.
\newblock Harper's Series in Modern Mathematics. Harper \& Row, Publishers, New
  York, 1964.

\bibitem[Ful98]{Fulton-IT}
William Fulton.
\newblock {\em Intersection theory}, volume~2 of {\em Ergebnisse der Mathematik
  und ihrer Grenzgebiete. 3. Folge. A Series of Modern Surveys in Mathematics
  [Results in Mathematics and Related Areas. 3rd Series. A Series of Modern
  Surveys in Mathematics]}.
\newblock Springer-Verlag, Berlin, second edition, 1998.

\bibitem[FvdP04]{Fresnel-vdPut}
Jean Fresnel and Marius van~der Put.
\newblock {\em Rigid analytic geometry and its applications}, volume 218 of
  {\em Progress in Mathematics}.
\newblock Birkh{\"a}user Boston, Inc., Boston, MA, 2004.

\bibitem[GH94]{Griffiths-Harris-AG}
Phillip Griffiths and Joseph Harris.
\newblock {\em Principles of Algebraic Geometry}.
\newblock Wiley Classics Library. John Wiley \& Sons, Inc., New York, 1994.
\newblock Reprint of the 1978 original.

\bibitem[GM03]{Gelfand-Manin}
Sergei~I. Gelfand and Yuri~I. Manin.
\newblock {\em Methods of homological algebra}.
\newblock Springer Monographs in Mathematics. Springer-Verlag, Berlin, second
  edition, 2003.

\bibitem[Har77]{Hartshorne-AG}
Robin Hartshorne.
\newblock {\em Algebraic geometry}.
\newblock Springer-Verlag, New York-Heidelberg, 1977.
\newblock Graduate Texts in Mathematics, No. 52.

\bibitem[Hau15]{Haug-T^2}
Luis Haug.
\newblock The {L}agrangian cobordism group of {$T^2$}.
\newblock {\em Selecta Math. (N.S.)}, 21(3):1021{--}1069, 2015.

\bibitem[HN75]{Harder-Narasimhan}
G.~Harder and M.~S. Narasimhan.
\newblock On the cohomology groups of moduli spaces of vector bundles on
  curves.
\newblock {\em Math. Ann.}, 212:215{--}248, 1974/75.

\bibitem[Huy06]{Huybrechts-Fourier-Mukai}
D.~Huybrechts.
\newblock {\em Fourier-{M}ukai transforms in algebraic geometry}.
\newblock Oxford Mathematical Monographs. The Clarendon Press, Oxford
  University Press, Oxford, 2006.

\bibitem[Huy14]{huybrechts-introSC}
D.~Huybrechts.
\newblock Introduction to stability conditions.
\newblock In {\em Moduli spaces}, volume 411 of {\em London Math. Soc. Lecture
  Note Ser.}, pages 179{--}229. Cambridge Univ. Press, Cambridge, 2014.

\bibitem[Joy15]{joyce-conj}
Dominic Joyce.
\newblock Conjectures on {B}ridgeland stability for {F}ukaya categories of
  {C}alabi-{Y}au manifolds, special {L}agrangians, and {L}agrangian mean
  curvature flow.
\newblock {\em EMS Surv. Math. Sci.}, 2(1):1{--}62, 2015.

\bibitem[Kon95]{Kontsevich-94}
Maxim Kontsevich.
\newblock Homological algebra of mirror symmetry.
\newblock In {\em Proceedings of the {I}nternational {C}ongress of
  {M}athematicians, {V}ol.\ 1, 2 ({Z}{\"u}rich, 1994)}, pages 120{--}139.
  Birkh{\"a}user, Basel, 1995.

\bibitem[KS06]{Kashiwara-Schapira-Categories_and_Sheaves}
Masaki Kashiwara and Pierre Schapira.
\newblock {\em Categories and sheaves}, volume 332 of {\em Grundlehren der
  Mathematischen Wissenschaften [Fundamental Principles of Mathematical
  Sciences]}.
\newblock Springer-Verlag, Berlin, 2006.

\bibitem[Lan02]{Lang-Algebra}
Serge Lang.
\newblock {\em Algebra}, volume 211 of {\em Graduate Texts in Mathematics}.
\newblock Springer-Verlag, New York, third edition, 2002.

\bibitem[LC07]{Le-Chen}
Jue Le and Xiao-Wu Chen.
\newblock Karoubianness of a triangulated category.
\newblock {\em J. Algebra}, 310(1):452{--}457, 2007.

\bibitem[{Le }97]{LePotier-VB}
J.~{Le Potier}.
\newblock {\em Lectures on vector bundles}, volume~54 of {\em Cambridge Studies
  in Advanced Mathematics}.
\newblock Cambridge University Press, Cambridge, 1997.
\newblock Translated by A. Maciocia.

\bibitem[LM89]{Lawson-Blaine-spin_geometry}
H.~Blaine Lawson, Jr. and Marie-Louise Michelsohn.
\newblock {\em Spin geometry}, volume~38 of {\em Princeton Mathematical
  Series}.
\newblock Princeton University Press, Princeton, NJ, 1989.

\bibitem[MP15]{Maciocia-Piyaratne-I}
Antony Maciocia and Dulip Piyaratne.
\newblock Fourier-{M}ukai transforms and {B}ridgeland stability conditions on
  abelian threefolds.
\newblock {\em Algebr. Geom.}, 2(3):270{--}297, 2015.

\bibitem[MP16]{Maciocia-Piyaratne-II}
Antony Maciocia and Dulip Piyaratne.
\newblock Fourier-{M}ukai transforms and {B}ridgeland stability conditions on
  abelian threefolds {II}.
\newblock {\em Internat. J. Math.}, 27(1):1650007, 27, 2016.

\bibitem[Muk81]{Mukai81-Duality}
Shigeru Mukai.
\newblock Duality between {$D(X)$}\ and {$D(\hat X)$}\ with its application to
  {P}icard sheaves.
\newblock {\em Nagoya Math. J.}, 81:153--175, 1981.

\bibitem[Mum63]{Mumford-63}
David Mumford.
\newblock Projective invariants of projective structures and applications.
\newblock In {\em Proc. {I}nternat. {C}ongr. {M}athematicians ({S}tockholm,
  1962)}, pages 526{--}530. Inst. Mittag-Leffler, Djursholm, 1963.

\bibitem[PZ98]{Polishchuk-Zaslow}
Alexander Polishchuk and Eric Zaslow.
\newblock Categorical mirror symmetry: the elliptic curve.
\newblock {\em Adv. Theor. Math. Phys.}, 2(2):443{--}470, 1998.

\bibitem[Sch99]{Schneiders-qac}
Jean-Pierre Schneiders.
\newblock Quasi-abelian categories and sheaves.
\newblock {\em M{\'e}m. Soc. Math. Fr. (N.S.)}, (76):vi+134, 1999.

\bibitem[Sei00]{seidel-gradedlag}
Paul Seidel.
\newblock Graded {L}agrangian submanifolds.
\newblock {\em Bull. Soc. Math. France}, 128(1):103{--}149, 2000.

\bibitem[Sei08]{seidel-book}
Paul Seidel.
\newblock {\em Fukaya categories and {P}icard-{L}efschetz theory}.
\newblock Z{\"u}rich Lectures in Advanced Mathematics. European Mathematical
  Society (EMS), Z{\"u}rich, 2008.

\bibitem[She15]{Sheridan-HMS_for_CY_HS}
Nick Sheridan.
\newblock Homological mirror symmetry for {C}alabi-{Y}au hypersurfaces in
  projective space.
\newblock {\em Invent. Math.}, 199(1):1{--}186, 2015.

\bibitem[Sil94]{Silverman-Advanced}
Joseph~H. Silverman.
\newblock {\em Advanced topics in the arithmetic of elliptic curves}, volume
  151 of {\em Graduate Texts in Mathematics}.
\newblock Springer-Verlag, New York, 1994.

\bibitem[Sil09]{Silverman-AEC}
Joseph~H. Silverman.
\newblock {\em The arithmetic of elliptic curves}, volume 106 of {\em Graduate
  Texts in Mathematics}.
\newblock Springer, Dordrecht, second edition, 2009.

\bibitem[Tho01]{Thomas-MMMMM}
R.~P. Thomas.
\newblock Moment maps, monodromy and mirror manifolds.
\newblock In {\em Symplectic geometry and mirror symmetry ({S}eoul, 2000)},
  pages 467{--}498. World Sci. Publ., River Edge, NJ, 2001.

\bibitem[TY02]{Thomas-Yau}
R.~P. Thomas and S.-T. Yau.
\newblock Special {L}agrangians, stable bundles and mean curvature flow.
\newblock {\em Comm. Anal. Geom.}, 10(5):1075{--}1113, 2002.

\bibitem[Ver96]{Verdier96-PhD}
Jean-Louis Verdier.
\newblock Des cat{\'e}gories d{\'e}riv{\'e}es des cat{\'e}gories
  ab{\'e}liennes.
\newblock {\em Ast{\'e}risque}, (239):xii+253 pp. (1997), 1996.
\newblock With a preface by Luc Illusie, Edited and with a note by Georges
  Maltsiniotis.

\bibitem[Wei94]{weibel-book}
Charles~A. Weibel.
\newblock {\em An introduction to homological algebra}.
\newblock Cambridge studies in advanced mathematics. Cambridge University
  Press, Cambridge, England, New York, 1994.

\end{thebibliography}
